\documentclass{article}
\usepackage{latexsym, amscd, amsfonts, eucal, mathrsfs, amsmath, amssymb, amsthm, xypic,xr, makeidx, stmaryrd, color, enumerate, tikz}
\usepackage{appendix}
\usepackage{multicol}
\usepackage[all]{xy}
\usepackage{hyperref}
\newtheorem*{thm}{Theorem}
\newtheorem{theorem}{Theorem}[section]
\newtheorem{lemma}[theorem]{Lemma}
\newtheorem{proposition}[theorem]{Proposition}
\newtheorem{corollary}[theorem]{Corollary}

\theoremstyle{definition}
\newtheorem*{defn}{Definition}
\newtheorem{definition}[theorem]{Definition}
\newtheorem{construction}[theorem]{Construction}
\newtheorem{convention}[theorem]{Convention}
\newtheorem{example}[theorem]{Example}
\newtheorem*{ex}{Example}

\newtheorem{counterexample}[theorem]{Counterexample}

\newtheorem{notation}[theorem]{Notation}
\newtheorem{warning}[theorem]{Warning}
\newtheorem{remark}[theorem]{Remark}
\newtheorem*{rmk}{Remark}

\DeclareMathOperator*{\hocolim}{hocolim}

\DeclareMathOperator*{\colim}{colim}

\newcommand{\fin}{\mathsf{Fin}}
\newcommand{\aeff}{\mathsf{A}^{\textit{eff}}}

\usepackage{cleveref}
\usepackage{scrextend}

\usepackage{filecontents}

\usepackage[citestyle=alphabetic,bibstyle=alphabetic,backend=bibtex,sorting=nty]{biblatex}
\renewbibmacro{in:}{}
\bibliography{slice-descr-2}

\DeclareFieldFormat{postnote}{#1}
\DeclareFieldFormat{multipostnote}{#1}

\begin{document}

\title{On categories of slices}
\author{Dylan Wilson}
\maketitle


\begin{abstract}
In this paper we give an algebraic description
of the category of $n$-slices for an arbitrary
group $G$, in the sense of Hill-Hopkins-Ravenel.
Specifically, given a finite group $G$ and an integer $n$, we construct
an explicit $G$-spectrum 
$W$ (called an \emph{isotropic slice $n$-sphere})
with the following properties: (i) the $n$-slice of a $G$-spectrum $X$
is equivalent to the data of a certain quotient of the Mackey functor
$\underline{[W, X]}$ as a module over the endomorphism Green functor
$\underline{[W,W]}$; (ii) the category of $n$-slices is equivalent to the full
subcategory of right modules over $\underline{[W,W]}$ for which
a certain restriction map is injective. We use this theorem to recover
the known results on categories of slices to date, and exhibit
the utility of our description in several new examples. We go
further and show that the Green
functors $\underline{[W,W]}$ for
certain slice $n$-spheres have a special property
(they are \emph{geometrically split})
which reduces the amount of data necessary
to specify a $\underline{[W,W]}$-module. This step
is purely algebraic and may be of independent interest. 
\end{abstract}

\tableofcontents

\newpage

\section*{Introduction}
\addcontentsline{toc}{section}{Introduction} \markboth{INTRODUCTION}{}

The stable homotopy category has a standard filtration 
$\{\tau_{\ge n}\mathsf{Sp}\}$ by
its subcategories of $n$-connective spectra for 
$n \in \mathbb{Z}$. We have a good
computational handle on this filtration for the following reasons:

	\begin{enumerate}[(i)]
	\item We can easily \emph{build} objects in
	$\tau_{\ge n}\mathsf{Sp}$: all $n$-connective objects
	are obtained from a wedge of copies of $S^n$
	by iteratively attaching cells of dimension $k \ge n$. 
	\item We can \emph{compute} when an object
	$Y$ is $n$-truncated: we need to check that the
	homotopy groups of $Y$ vanish above dimension $n$.
	\item We can \emph{compute} when an object $X$ is in
	$\tau_{\ge n}\mathsf{Sp}$: we need to check that
	the homotopy groups of $X$ vanish below dimension $n$.
	\item The category of $n$-connective and $n$-truncated
	objects has a completely \emph{algebraic} description:
	it is equivalent to the category of abelian groups via the
	Eilenberg-MacLane functor. Moreover, knowledge of
	the $n$th Postnikov
	layer of a spectrum $X$ is equivalent to knowledge of 
	$\pi_nX$.
	\end{enumerate}

Motivated by the Schubert cell structures on complex
Grassmanians, Dan Dugger \cite{dug} defined the slice filtration
in $C_2$-equivariant homotopy theory
to study Atiyah's $\mathbf{R}$eal $K$-theory \cite{atiyah}. This filtration 
was later generalized to a filtration of $G$-equivariant
stable homotopy theory for any finite group $G$
and used to great effect by Hill-Hopkins-Ravenel \cite{HHR}
in their solution to the Kervaire invariant one problem.

Associated to the slice filtration on the
category $\mathsf{Sp}^G$ of $G$-spectra
are the notions of \emph{slice $n$-connective}
and
\emph{slice $n$-truncated} spectra. (The reader
may review the relevant definitions below in
Definition \ref{defn:slice-filtn}.) A $G$-spectrum which is
both slice $n$-connective and slice $n$-truncated
is called an
\emph{$n$-slice}. The relevant features
of the
slice filtration are as follows:
	\begin{enumerate}[(i)]
	\item By design, we can easily \emph{build}
	slice $n$-connective $G$-spectra.
	\item By definition, we can \emph{compute}
	when a $G$-spectrum is slice $n$-truncated
	by computing certain homotopy classes of
	maps in from representation spheres.
	\item Thanks to a recent result of Hill-Yarnall \cite{HY},
	it is possible to \emph{compute} when a $G$-spectrum
	is slice $n$-connective: we need to check certain connectivity
	conditions on each of its geometric fixed point spectra.
	\end{enumerate}
The purpose of
this paper is to give an analogue of the property (iv) of
the Postnikov filtration, in complete generality. That is, we
provide an algebraic description of the layers of the slice filtration
together with a replacement for the functor $\pi_n$ in this context.

As a way to establish notation and provide motivation for our approach,
we begin by reviewing the completely understood case of $G = C_2$.
Recall that, given $G$-spectra $X$ and $Y$ we may form a Mackey
functor $\underline{[X, Y]}$ whose value on finite $G$-sets is given by:
	\[
	\mathsf{Fin}_G \ni T \mapsto [T_+ \wedge X, Y]^G.
	\]
When $X = S^V$ is a representation sphere associated to
a virtual representation, $V$, then we define the 
\emph{$V$th homotopy Mackey functor} by
	\[
	\underline{\pi}_VX := \underline{[S^V, X]}.
	\]
Finally, we denote by $\mathsf{Slice}_n$ the category
of $n$-slices, by $P^n_nX$ the $n$-slice of a $G$-spectrum $X$,
and by $\mathsf{Mack}(G, \mathsf{Ab})$ the category of Mackey functors
valued in abelian groups.

\begin{thm}[\cite{HHR}, \cite{primer}]
Let $G = C_2$, and let $\rho$ denote the regular representation.
	\begin{enumerate}[(a)]
	\item The functor
		\[
		\underline{\pi}_{n\rho-1}:
		\mathsf{Slice}_{2n-1} \longrightarrow 
		\mathsf{Mack}(C_2, \mathsf{Ab})
		\]
	is an equivalence of categories.
	\item The functor
		\[
		\underline{\pi}_{n\rho}:
		\mathsf{Slice}_{2n} \longrightarrow
		\mathsf{Mack}(C_2, \mathsf{Ab})
		\]
	is fully faithful. The essential image consists of those
	Mackey functors $\underline{M}$ such that the
	restriction map
		\[
		\mathrm{res}: \underline{M}(*) 
		\longrightarrow \underline{M}(C_2)
		\]
	is injective.
	\item The slices of a $G$-spectrum $X$ are determined by
	the formulae:
		\[
		\underline{\pi}_{n\rho -1}P^{2n-1}_{2n-1}X
		=
		\underline{\pi}_{n\rho - 1}X.
		\]
		\[
		\underline{\pi}_{n\rho}P^{2n}_{2n}X
		=
		\frac{\underline{\pi}_{n\rho}X}{\mathrm{ker}(\mathrm{res})}.
		\]
	\end{enumerate}
\end{thm}

One might hope that, in general, one may compute $n$-slices
directly in terms of a single $RO(G)$-graded homotopy Mackey functor.
Unfortunately, $n$-slices need not be $RO(G)$-graded suspensions
of Eilenberg-MacLane spectra in 
general (Counterexample \ref{counter:non-em-slice}).
Instead, we will need to probe $G$-spectra
by objects more general than representation spheres. This brings us
to the key definition of the paper.

\begin{defn} An \textbf{isotropic slice $n$-sphere} is a compact
$G$-spectrum $W$ with the property that, 
for every subgroup $H \subseteq G$,
the geometric fixed point spectrum $W^{\Phi H}$ is equivalent to
a nonzero, finite wedge of spheres of dimension $\lfloor n/|H|\rfloor$. 
\end{defn}

\begin{rmk} The appearance of the
floor function here is inspired by the theorem
of \cite{HY} characterizing the slice filtration
in terms of connectivity conditions on geometric
fixed points. We will review that theorem below
in \S0 and generalize it in \S1.3.
\end{rmk}

\begin{ex} For any group $G$, $S^{n\rho}$ is
an isotropic slice $n|G|$-sphere and $S^{n\rho -1}$
is an isotropic slice $(n|G|-1)$-sphere.
\end{ex}

\begin{ex} For any group $G$, the cofiber
of the collapse map $G_+ \to S^0$ is a slice $1$-sphere.
It is not equivalent to a representation sphere or an
induced representation sphere unless $|G| = 2$.
\end{ex}

Next, we will need a way
to state a generalization of the injectivity condition
in part (b) above.

\begin{defn} A subgroup $H \subseteq G$
is called an
\textbf{$n$-jump} if the inequality
	\[
	\left\lfloor \frac{n+1}{|H|} \right\rfloor > \left\lfloor \frac{n}{|H|}\right\rfloor
	\]
holds. We denote by $\mathrm{Jump}_n$ the set of
conjugacy classes of $n$-jumps.
\end{defn}

It is not so obvious that isotropic slice $n$-spheres exist
for arbitrary $G$ and $n$, but indeed they do
(Proposition \ref{prop-sph-is-test}).

Finally, we remark that, for any $G$-spectrum $X$, the Mackey
functor 
	\[
	\underline{\mathrm{End}}(X):= \underline{[X,X]}
	\]
admits the canonical structure of
a Green functor (Definition \ref{defn:green}) under composition.
Moreover, given another $G$-spectrum $Y$, the Mackey
functor $\underline{[X, Y]}$ is naturally a right module
over $\underline{\mathrm{End}}(X)$ via precomposition.

Now we can state a version of our first main result
(Theorem \ref{thm:slices-as-modules}).

\begin{thm} Let $W$ be an isotropic slice $n$-sphere.
Define a $G$-set
	\[
	T^{jump}:= \coprod_{[H] \in \mathrm{Jump}_n} G/H.
	\]
	\begin{enumerate}[(a)]
	\item The functor
		\[
		\underline{[W, -]}:
		\mathsf{Slice}_n
		\longrightarrow
		\mathsf{RMod}_{\underline{\mathrm{End}}(W)}
		\]
	is fully faithful. The essential image consists of those
	$\underline{\mathrm{End}}(W)$-modules $\underline{M}$
	with the property that, for every $G$-set $T$, the restriction map
	associated to the projection $T^{jump} \times T \to T$,
		\[
		\underline{M}(T)
		\longrightarrow \underline{M}(T^{jump} \times T)
		\]
	is injective.
	\item Let 
		\[
		L^{inj}: \mathsf{RMod}_{\underline{\mathrm{End}}(W)}
		\longrightarrow
		\mathsf{RMod}_{\underline{\mathrm{End}}(W)}
		\]
	denote the localization functor which enforces the injectivity
	constraint in (a). Then the $n$-slice of a $G$-spectrum
	$X$ is determined by the formula
		\[
		\underline{[W, P^n_nX]}
		= 
		L^{inj}\underline{[W, X]}. 
		\]
	\end{enumerate}

\end{thm}

This result is enough for many applications. For example,
it is straightforward to deduce from it the previous known
results on categories of slices (see \S\ref{sec:examples}).
Nevertheless, the results proved in the body of the paper are
stronger in several respects:
	\begin{itemize}
	\item Using a construction of MacPherson-Vilonen
	and some special features of
	the Green functor $\underline{\mathrm{End}}(W)$,
	we provide a simpler description (Theorem \ref{thm:slices-as-tw-mack})
	of the category
	$\mathsf{RMod}_{\underline{\mathrm{End}}(W)}$
	which cuts down on the computation necessary
	to determine an $n$-slice.
	\item We do not restrict ourselves to the standard slice filtration,
	but allow more general filtrations. Thus, the user can adapt
	to more situations of interest. 
	\item All of the results are proven in the setting of parameterized
	stable homotopy theory. One reason is to allow for more flexible inductive
	techniques. It also means the results apply to other settings,
	such as Goodwillie calculus \cite{glas-calc}. 
	\end{itemize}

We now give a summary of each of the sections.\newline\newline
\noindent\textbf{Section 0}. We include this section for the
equivariant homotopy theorist eager to find ready-to-use
definitions and statements without the need to unravel
too much notation or contend with the level
of generality used in the body of the paper. In this section,
we survey all of the results of the paper in a form
specialized to the case of the standard slice filtration on $G$-spectra.
We also provide enough of a sketch of the proofs that an expert may reconstruct
the details. After reading this section, a reader versed in equivariant
homotopy theory should be able to fruitfully skip to \S\ref{sec:examples}
and understand
the examples presented therein without contemplating the words
`parameterized $\infty$-category' or `inductive orbital category'. 
\newline
\noindent\textbf{Section 1}. In this section we set up the formal
backdrop for our work. We review and develop gluing techniques,
inductive techniques, and a sort of `six-functor' formalism for
manipulating homotopy theories parameterized over certain
bases which we call \emph{inductive orbital categories}
(Definition \ref{defn:ioc}). We define
a notion of slice filtration in this context and prove
some elementary results about these.
As a quick application of this formalism,
we produce a characterization of slice filtrations which, as a corollary,
provides a streamlined proof of the main result in \cite{HY}. 
The framework of parameterized
homotopy theory we use is due to Barwick-Dotto-Glasman-Nardin-Shah.
They have obtained most of the results in this section independently,
using slightly different terminology. To the best of the author's knowledge,
however, the remaining sections, including the main theorem and
description of slices, is new.\newline
\noindent\textbf{Section 2}. This section contains the bulk of the work.
We begin by introducing slice spheres
(Definition \ref{def-slice-sphere})
which generalize (induced) representation spheres. The transition from
the category of $n$-slices, $\mathsf{Slice}_n$, to our final
algebraic description takes several steps. First, using slice spheres,
or more generally \emph{testing subcategories} of slice spheres,
as the collection of free algebras for a Lawvere theory, we prove
that $\mathsf{Slice}_n$ is a localization of the category of models
for this Lawvere theory (Theorem \ref{thm:slices-as-models}). After showing
that slice spheres exist in sufficient supply,
we prove this using a `many-object'
variant of a classical theorem dating back to Freyd
and Gabriel. The key step is to establish that equivalences are detected
by the testing subcategory. 

The next
step moves from the Lawvere theory to a category
of modules over a Green functor. This is entirely
algebraic, and the argument is essentielly
a parameterized version of the
aforementioned result of Freyd-Gabriel.
From here,
we prove that $\mathsf{Slice}_n$ is a localization of
the category $\mathsf{RMod}_{\underline{\mathrm{End}}(W)}$
where $W$ is an isotropic slice $n$-sphere (Definition \ref{def-slice-sphere},
Theorem \ref{thm:slices-as-modules}). 

Next, we distill a special feature of
the Green functor $\underline{\mathrm{End}}(W)$
which shows that its structure is strongly controlled
by the endomorphisms of
its geometric fixed points $\mathrm{End}(W^{\Phi T})$
for each orbit $T$. We
call Green functors that share this property \emph{geometrically
split} (Definition \ref{defn:geom-split}) and digress
to prove a purely algebraic result about the structure
of modules over geometrically split Green functors.
While this work is purely algebraic,
it is not trivial. For example, one must contend with the
combinatorics of the Burnside category
(Proposition \ref{prop:pre-hereditary}).

By definition, a slice $n$-sphere $W$ has the property
that $W^{\Phi T}$ is a finite wedge of spheres,
so the work ultimately reduces to understanding
modules over the matrix rings $M_n(\mathbb{Z})$.
Of course, these rings are Morita equivalent to $\mathbb{Z}$.
This observation leads us to formulate a description
of $\mathsf{Slice}_n$ that does not require
computing $\underline{\mathrm{End}}(W)$ or
understanding its action $\underline{[W, X]}$. 
Instead, one need only understand the homology
of the wedge of spheres $W^{\Phi T}$
as a module over the group $\mathrm{Aut}(T)$
of automorphisms of the orbit $T$. In practice,
this is much easier.
Our main general result identifies
$\mathsf{Slice}_n$ with a localization of the
category of \emph{twisted Mackey functors}
(Definition \ref{defn:tw-mack}, Theorem \ref{thm:slices-as-tw-mack}).
\newline
\noindent\textbf{Section 3}. A general theory is no good without
examples. We show that the machinery in the first two sections
can be made put to use in cases of interest. First, we show in
\S\ref{ssec:G}
how to recover the known results on slices of $G$-spectra
to date, i.e. the cases of $(n|G|-\epsilon)$-slices where
$\epsilon= 0,1,2$ (\cite{HHR, Ullman}). In \S\ref{ssec:Cp}
we carry out our theory in the case $G = C_p$. The categories
of slices for $C_p$ were previously determined by \cite{HY}
in a slightly different form, and we show how to recover their
description from ours. Finally, in \S\ref{ssec:C4} we move on to
a new example. This is the first case where we see
slices that are \emph{not} $RO(G)$-graded suspensions
of Eilenberg-MacLane spectra, and so are not amenable to
previous methods of attack.
\newline
\newline
\noindent\textbf{Acknowledgements}. Part
of this work was supported
by NSF grant DGE-1324585.

I would
like to extend a special thanks to Saul Glasman and 
Mike Hill. Saul Glasman's work is responsible for my current
understanding of what equivariant homotopy theory
looks like, and the relationship between equivariant
spectra and constructible sheaves was one
of the motivations for the results of this paper.
Mike Hill has patiently fielded my questions
about equivariant homotopy theory
and the slice filtration for many years now,
and I learned most of the tricks of the trade from him.

I am also
grateful to Clark Barwick and Denis Nardin
for conversations related to this work, and
to Peter May for his comments on an earlier
draft. Finally, as these results
have their precursors in my thesis,
I would like to thank my advisor Paul Goerss
for his support and wisdom. 
\newline\newline
\textbf{Notations and conventions}.
	\begin{itemize}
	\item If $G$ is a group we denote by $\rho_G$
	(or just $\rho$, if $G$ is understood) the
	real, regular representation of $G$.
	\item If $X$ and $Y$ are objects in a model category
	(or, more generally, a relative category)
	we will use $\mathrm{map}(X,Y)$ to denote the \emph{derived}
	mapping space between $X$ and $Y$.
	\item We will use $[X,Y]$ to denote the set
	of maps in a homotopy category, with no
	further decoration if it is clear where these objects
	live. So, for example, if $X$ and $Y$ are $G$-spectra,
	then $[X,Y]$ is the set of maps between
	$X$ and $Y$ in the homotopy category of
	$G$-spectra, not the set of maps between
	the underlying spectra in the
	homotopy category of spectra. 
	\item We say that a spectrum $X$ is \textbf{$n$-connective}
	if $\pi_k(X) = 0$ for all $k < n$. 
	\item We say a spectrum $X$ is \textbf{$n$-truncated}
	if $\pi_k(X) = 0$ for all $k> n$. 
	\item Since $W$ is being used a lot, we will write $\mathrm{Aut}(G/H)$
	for the automorphisms of the $G$-set $G/H$ (i.e. the Weyl group)
	instead of
	$WH$ or $W_GH$.
	\item If $\mathcal{C}$ and $K$ are $\infty$-categories,
	we denote by $\mathsf{Psh}_{\mathcal{C}}(K)$
	the $\infty$-category of $\mathcal{C}$-valued presheaves.
	If $\mathcal{C}$ happens to equivalent to a 1-category, for example
	if $\mathcal{C} = \mathsf{Set}$ is the category of sets,
	then there is a natural equivalence
	$\mathsf{Psh}_{\mathcal{C}}(K) \cong \mathsf{Psh}_{\mathcal{C}}(hK)$
	where $hK$ denotes the homotopy category of $K$.
	\end{itemize}

\setcounter{section}{-1}

\section{An overview in the equivariant case}\label{sec:eqvt-overview}

This section contains a survey of our main results and definitions
in the case of $G$-spectra. We assume the reader is familiar with
the basic definitions of equivariant stable homotopy theory.
A nice overview can be found in \S1-\S3 of \cite{HHR}.
We begin by reviewing the definition of the slice filtration.

\begin{definition}
	\begin{itemize}
	\item A $G$-spectrum $X$ is \textbf{slice $n$-connective} if
	it belongs to the full subcategory of $\mathsf{Sp}^G$ generated
	under extensions and homotopy colimits by the
	objects 
		\[
		G/H_+ \wedge S^{k\rho_H - \epsilon}, \quad \epsilon = 0,1
		\text{ and } \,\, k|H|-\epsilon \ge n.
		\] 
	We denote
	the homotopy theory of slice $n$-connective $G$-spectra
	by $\mathsf{Sp}^G_{\ge n}$. We will often write $X \ge n$
	to indicate that $X$ is slice $n$-connective.
	\item A $G$-spectrum $Y$ is \textbf{slice $n$-truncated}
	if, for every $X \ge n+1$, the mapping space
	$\mathrm{map}(X, Y)$ is weakly contractible. We denote
	the homotopy theory of slice $n$-truncated $G$-spectra by
	$\mathsf{Sp}^G_{\le n}$ and indicate that $Y$ is slice $n$-truncated
	by writing $Y \le n$.
	\item We denote by $P_n: \mathsf{Sp}^G \to \mathsf{Sp}^G_{\ge n}$
	a right adjoint to the inclusion, and by $P^n: \mathsf{Sp}^G
	\to \mathsf{Sp}^G_{\le n}$ a left adjoint to the inclusion.
	\item We say that $A \in \mathsf{Sp}^G$ is an \textbf{$n$-slice}
	if $A \le n$ and $A \ge n$. We denote the category of $n$-slices
	by $\mathsf{Slice}_n$. 
	\end{itemize}
\end{definition}

\begin{remark} If $A$ and $B$ are $n$-slices, then the mapping space
$\mathrm{map}(A, B)$ is discrete, which justifies the use of the word 
\emph{category} of $n$-slices. 
\end{remark}

\begin{remark} If $G$ is trivial, then $X \ge n$ (resp. $Y \le n$) if
and only if $X$ is $n$-connective (resp. $n$-truncated)
in the classical sense. Hence the slice filtration is the usual Postnikov
filtration.
\end{remark}

The slice filtration for non-trivial groups is \emph{not} the
filtration associated to a $t$-structure. Specifically, we have
an inclusion
	\[
	\Sigma \left(\mathsf{Sp}^G_{\ge n}\right) \subseteq \mathsf{Sp}^G_{\ge n+1}
	\]
which is usually strict. Nevertheless, each subcategory $\mathsf{Sp}^G_{\ge n}$
is the collection of connective objects for a $t$-structure on
$\mathsf{Sp}^G$. We denote the heart of this $t$-structure by $\heartsuit_n$. 
The following is elementary from the inclusion above:

\begin{lemma} There is an inclusion
$\mathsf{Slice}_n \longrightarrow \heartsuit_n$ which exhibits
the source as an accessible localization of the target, i.e.
it admits an accessible left adjoint.
\end{lemma}

While the heart of a $t$-structure is always an abelian category,
the same is not true in general for $\mathsf{Slice}_n$. Our results
on $\mathsf{Slice}_n$ are obtained by first identifying $\heartsuit_n$
with a more algebraic category, and then identifying the localization
that yields $\mathsf{Slice}_n$. 
\newline

When the slice filtration was first introduced, one had to exhibit slice
$n$-connectivity by an explicit construction. Recently, a much
simpler characterization of slice connectivity was given
by Hill-Yarnall \cite{HY}.

\begin{theorem}[Hill-Yarnall] Let $X$ be a $G$-spectrum.
Then $X \ge n$ if and only if, for all $H \subseteq G$,
the spectrum $X^{\Phi H}$ is $(\lfloor n/|H|\rfloor)$-connective.
\end{theorem}

This motivates the following definition.

\begin{definition} Let $W$ be a compact $G$-spectrum. 
	\begin{itemize}
	\item $W$ is a \textbf{slice sphere} if,
	for every $H \subseteq G$, the spectrum $W^{\Phi H}$
	is a finite (possibly trivial) wedge of spheres.
	\item $W$ is a \textbf{slice $n$-sphere} if,
	for every $H \subseteq G$, the spectrum $W^{\Phi H}$
	is a finite (possibly trivial) wedge of
	spheres of dimension $\lfloor n/|H|\rfloor$.
	\item $W$ is an \textbf{isotropic slice $n$-sphere}
	if it is a slice $n$-sphere and $W^{\Phi H}$
	is nonzero for each $H \subseteq G$.
	\end{itemize}
\end{definition}

To motivate the next steps, let us imagine that we are
trying to identify the heart of $\mathsf{Sp}^G$ with its
standard $t$-structure, i.e. the $t$-structure for
which $X$ is $0$-connective if the (genuine)
fixed points $X^H$ are $0$-connective for all $H \subseteq G$. 
So we would like to understand $G$-spectra $A$
such that $A$ is $0$-connective and $0$-truncated.

Now, by definition, a map $X \to Y$ is an equivalence
if and only if $X^H \to Y^H$ is an equivalence for
all $H \subseteq G$. We know that $S^0$ detects
equivalences between $0$-connective and $0$-truncated
spectra. The restriction-induction adjunction implies that
equivalences between $0$-connective, $0$-truncated
$G$-spectra are detected by the subcategory
$\{G/H_+ \wedge S^0\}_{H \subseteq G}$, and
hence also by the subcategory $\{T_+ \wedge S^0\}$
where $T$ ranges over all finite $G$-sets.

Thus, the restricted Yoneda embedding:
	\[
	\underline{\pi}_0: \left(\mathsf{Sp}^G\right)^{\heartsuit}
	\longrightarrow
	\mathsf{Psh}^{\times}_{\mathsf{Set}}(\{T_+ \wedge S^0\}) 
	\]
is conservative, where $\mathsf{Psh}_{\mathsf{Set}}^{\times}$ denotes the category
of product-preserving presheaves of sets. Using the fact that $[T_+ \wedge S^0, Y] = 0$
when $Y$ is 1-connective, and in particular that 
	\[
	[T_+ \wedge S^0, Y] = [T_+ \wedge S^0, \tau_{\le 0}Y],
	\]
one shows that $\underline{\pi}_0$ preserves cokernels. The Barr-Beck theorem implies
$\underline{\pi}_0$
is monadic, and the same equality above applied to $Y = G/H_+ \wedge S^0$,
shows that the monad in question is the identity functor. We conclude that
the assignment $A \mapsto \{[T_+ \wedge S^0, A]\}$ yields an equivalence of categories:
	\[
	\left(\mathsf{Sp}^G\right)^{\heartsuit} \cong 
	\mathsf{Psh}^{\times}_{\mathsf{Set}}(\{T_+ \wedge S^0\}).
	\]
This is already a sort of algebraic description: categories of product-preserving
presheaves are known as \emph{models
for Lawvere theories} and behave like
categories of algebraic objects. 
Our understanding of a Lawvere theory
is only as good as our understanding of 
the (maps between) \emph{free} objects,
which in this case are the objects $\{T_+ \wedge S^0\}$.

\begin{theorem}[Segal, tom Dieck, Lewis-May-Steinberger]
The category $\{T_+ \wedge S^0\}$ is equivalent to the Burnside category
of finite $G$-sets. Hence, from the definition of Mackey functors,
	\[
	\mathsf{Psh}^{\times}_{\mathsf{Set}}(\{T_+ \wedge S^0\})
	\cong \mathsf{Mack}(G; \mathsf{Ab}).
	\]
\end{theorem}

Our first algebraic description of the category of $n$-slices
follows this outline closely. First, we require a replacement
for the category $\{T_+ \wedge S^0\}$ of test objects.

\begin{definition} We say that a subcategory
$\mathsf{Test}_n$ of the category of slice $n$-spheres
is a \textbf{testing subcategory} if
	\begin{enumerate}[(i)]
	\item For all finite $G$-sets $T$ and $W \in \mathsf{Test}_n$,
	the $G$-spectrum $T_+ \wedge W$ is also in $\mathsf{Test}_n$.;
	\item For every $H \subseteq G$, there is some
	$W \in \mathsf{Test}_n$ such that $W^{\Phi H}$
	is nonzero.
	\end{enumerate}
\end{definition}

It is not obvious from the definition that testing subcategories exist.
To show that they do, it suffices to produce a single isotropic slice $n$-sphere.

\begin{proposition}(Proposition \ref{prop-sph-is-test}) For any $G$ and
$n$, there exists an isotropic slice $n$-sphere.
\end{proposition}
\begin{proof}[Proof sketch] Begin with the sphere
$S^k$ where $k = \lfloor n/|G|\rfloor$ and then inductively attach
induced spheres to bump up the dimension of each geometric
fixed point spectrum without messing up the work you've already done.
In order to ensure that the geometric fixed points
remain wedges of spheres, one must choose attaching maps
which split upon taking geometric fixed points. The simplest way
to achieve this in general is to use the counit
$G/H_+ \wedge X \to X$. 
\end{proof}

Now we may define an analog of the functor $\underline{\pi}_0$:

\begin{definition} If $\mathsf{Test}_n$ is a testing subcategory,
define the category of \textbf{model $n$-slices} by
	\[
	\mathsf{Model}_n:=
	\mathsf{Psh}_{\mathsf{Set}}^{\times}(\mathsf{Test}_n).
	\]
Let $\hat{\pi}_n$ denote the restricted Yoneda embedding:
	\[
	\hat{\pi}_n: 
	\mathrm{h}\mathsf{Sp}^G \longrightarrow
	\mathsf{Model}_n.
	\]
\end{definition}

\begin{theorem}[Theorem \ref{thm:slices-as-models}]
	\begin{enumerate}[(i)]
	\item The restriction of $\hat{\pi}_n$ to $\heartsuit_n$ yields
	an equivalence of categories
		\[
		\hat{\pi}_n:
		\heartsuit_n \stackrel{\cong}{\longrightarrow} \mathsf{Model}_n.
		\]
	\item Under this equivalence, the category
	$\mathsf{Slice}_n$ corresponds to the full subcategory
	of $\mathsf{Model}_n$ spanned by presheaves
	$F$ with the property that, for all $W \in \mathsf{Test}_n$,
	the projection $T^{jump}_+ \wedge W \to W$ produces
	an \emph{injective} map
		\[
		F(W) \longrightarrow F(T^{jump}_+ \wedge W).
		\]
	\end{enumerate}
\end{theorem}
\begin{proof}[Proof sketch] Part (i) is established following the
same outline as in the computation of the heart of the standard
$t$-structure above. The main step is to prove that the functor
is conservative, and this is done by standard induction arguments
using isotropy separation.

For part (ii), let $L^{inj}$ denote the monad on $\mathsf{Model}_n$
which enforces the injectivity constraint in the statement. Then
the main computation is that, for any $G$-spectrum $X$,
$L^{inj}\hat{\pi}_nX \cong \hat{\pi}_nP^nX$. The key yet elementary
fact used here is the following: while it is \emph{not} true that $\hat{\pi}_n$ vanishes
on all slice $(n+1)$-connective spectra, it is nevertheless the case that
$\hat{\pi}_n\Sigma X = 0$ when $X \ge n$. (Proposition \label{prop:slice-htpy-vanish}).
This allows us to control the fact that $\hat{\pi}_n$ is not quite exact
on $\mathsf{Slice}_n$. 
\end{proof}

This brings the category of $n$-slices into the realm of algebra. However,
as before, our understanding of $\mathsf{Model}_n$
is only as good as our understanding of the testing subcategory.
For certain examples, it is possible to analyze this category directly.

\begin{example} If $n = k|G|$, then $S^{k\rho}$ generates a testing subcategory
equivalent to the category $\{T_+ \wedge S^0\}$, since smashing with $S^{k\rho}$
is invertible. Thus,
	\[
	\heartsuit_{k|G|} \cong \mathsf{Mack}(G; \mathsf{Ab})
	\]
and, unwinding the definitions, we recover the equivalence \cite{HHR, primer}:
	\[
	\mathsf{Slice}_{k|G|} \cong \{\text{Mackey functors with injective restriction maps}\}.
	\]
\end{example}

In general, however, it is convenient to have an alternative description.
We observe that we already understand $\{T_+ \wedge S^0\}$ and attempt
to understand $\mathsf{Test}_n$ in terms of what we know, so as not
to reinvent the wheel. We restrict to the case of testing subcategories
generated by a single isotropic slice $n$-sphere, $W$, for ease.

In this case, $\hat{\pi}_nX$ is essentially the data of:
	\begin{itemize}
	\item the Mackey functor $\underline{[W, X]}$, and
	\item its interaction with maps $T_+ \wedge W \to U_+ \wedge W$
	for $G$-sets $T, U$.
	\end{itemize}

The name for such a thing is a module over a Green functor.
So we show, very formally, that

\begin{theorem}[Theorem \ref{thm:slices-as-modules}]
	\begin{enumerate}[(i)]
	\item The functor $\underline{[W, -]}$ induces an equivalence
	of categories
		\[
		\underline{[W, -]}: \heartsuit_n \stackrel{\cong}{\longrightarrow}
		\mathsf{RMod}_{\underline{\mathrm{End}}(W)}.
		\]
	\item Under this equivalence, the category $\mathsf{Slice}_n$
	corresponds to the full subcategory of
	$\mathsf{RMod}_{\underline{\mathrm{End}}(W)}$
	spanned by those modules $\underline{M}$ such that the restriction map
		\[
		\underline{M}(T) \longrightarrow \underline{M}(T^{jump} \times T)
		\]
	is injective for all finite $G$-sets $T$.
	\end{enumerate}
\end{theorem}
\begin{proof}[Proof sketch] In the theory
of ordinary abelian categories, one can recognize categories
of modules as those abelian categories that admit a compact,
projective generator. A similar result is true
in the context of $G$-categories to recognize
$G$-categories of modules over a Green functor,
and the claim follows.
\end{proof}

Again, for some applications, this result suffices. 

\begin{example} If $n= k|G|$, we get yet another proof that
$\mathsf{Slice}_{k|G|}$ is the category of Mackey functors
with injective restriction maps. Indeed, since $S^{k\rho}$ is
invertible,
	\[
	\underline{\mathrm{End}}(S^{k\rho}) \cong 
	\underline{\mathrm{End}}(S^0)
	\cong \underline{A}.
	\]
But the Burnside Mackey functor $\underline{A}$ is the unit in the category
of Mackey functors, so a right $\underline{A}$-module is just a Mackey functor.
\end{example}

\begin{remark} This argument applies more generally whenever one can find an
\emph{invertible} isotropic slice $n$-sphere. In this case, the category of
$n$-slices will be equivalent, via  
the corresponding ($\mathrm{Pic}$-graded) homotopy
Mackey functor,
to the category of Mackey functors
satisfying certain injectivity conditions on their restriction maps.
\end{remark}

Now we come to the final simplification
which will ultimately remove the need to
compute the action of $\underline{\mathrm{End}}(W)$
entirely. A (right) module $\underline{M}$ over 
a Green functor $\underline{R}$ consists of, in particular,
a collection of $\underline{R}(G/H)$-modules, $\underline{M}(G/H)$,
for each $H \subseteq G$. However, much of the description of this action
is redundant due to the Frobenius relation dictating the action of a transfer:
	\[
	m \cdot \mathrm{tr}(r) = \mathrm{tr}(\mathrm{res}(m) \cdot r).
	\]
In general, we may not be able to untangle the transferred ring elements from
those not in the image of the transfer, but occasionally we are lucky. We
give a name to this situation.

If $\underline{R}$ is a Green
functor then, for each $H \subseteq G$ we can form the quotient:
	\[
	\underline{R}(G/H) \longrightarrow \frac{\underline{R}(G/H)}
	{\langle\mathrm{tr}_K^H\underline{R}(G/K) | K\text{ subconjugate to }H\rangle}
	=: \underline{R}^{\Phi H}. 
	\]

\begin{definition} We say that $\underline{R}$ is \textbf{geometrically split}
if the map $\underline{R}(G/H) \to \underline{R}^{\Phi H}$ admits
an $\mathrm{Aut}(G/H)$-equivariant ring section.
\end{definition}

If we choose splittings for a geometrically split Mackey functor $\underline{R}$,
then a right $\underline{R}$-module gives rise to a sequence of
$\underline{R}^{\Phi H}$-modules $\underline{M}(G/H)$ with a
compatible action of $\mathrm{Aut}(G/H)$. The precise structure
remaining is explained in Definition \ref{defn:r-phi-modules}
below. We denote the resulting category by 
$\mathsf{RMod}_{\underline{R}^{\Phi}}$. 

We then have the following piece of algebra: 

\begin{theorem}[Theorem \ref{thm:split-structure}]
With notation as above, the forgetful functor
	\[
	\mathsf{RMod}_{\underline{R}}
	\longrightarrow
	\mathsf{RMod}_{\underline{R}^{\Phi}}
	\]
is an equivalence of categories.
\end{theorem}
\begin{proof}[Proof sketch] We prove this
using the main theorem of
\cite{FP} on comparisons of abelian category recollements.
The key step is showing that $\mathsf{Rmod}_{\underline{R}}$
admits a pre-hereditary recollement, in the language of
\cite{FP}, and proving this fact requires actually
digging into the structure of the Burnside category.
(Proposition \ref{prop:pre-hereditary}).
\end{proof}

To apply this to our situation we need
to identify $\underline{\mathrm{End}}(W)^{\Phi H}$.

\begin{proposition} There is an isotropic
slice $n$-sphere $W$ with the property
that, for every $H \subseteq G$, the following
conditions are satisfied:
	\begin{enumerate}[(a)]
	\item The natural map
		\[
		[W, W]^H \longrightarrow
		\mathrm{End}(W^{\Phi H})
		\]
	admits an $\mathrm{Aut}(G/H)$-equivariant
	ring section.
	\item The natural map
		\[
		\underline{\mathrm{End}}(W)^{\Phi H}
		\longrightarrow 
		\mathrm{End}(W^{\Phi H})
		\]
	is an isomorphism.
	\item Let $J_H = \pi_{\lfloor n/|H|\rfloor}W^{\Phi H}$
	as an $\mathrm{Aut}(G/H)$-module and left
	$\mathrm{End}(W^{\Phi H})$-module. Then the
	map
		\[
		\mathrm{End}(W^{\Phi H})
		\longrightarrow
		\mathrm{End}_{\mathbb{Z}}(J_H)
		\]
	is an isomorphism of $\mathrm{Aut}(G/H)$-modules.
	\end{enumerate}
\end{proposition}
\begin{proof}[Proof sketch] For part (ii), we use
isotropy separation and some connectivity
arguments. Part (iii) is elementary because
$W^{\Phi H}$ is a finite wedge of spheres of dimension
$\lfloor n/|H|\rfloor$. Both of these statements
are true for every isotropic slice $n$-sphere.
Part (i) relies on an inductive argument 
applied to a specific construction of an isotropic
slice $n$-sphere. 
\end{proof}

\begin{remark} We believe part (i) of this proposition holds
for \emph{every} isotropic slice $n$-sphere,
but we have not tried to prove this.
\end{remark}

By definition, we know that $W^{\Phi H}$ is 
a finite, nonzero wedge of spheres. In particular,
$J_H$ is a free abelian group. Morita theory,
slightly generalized to account for the action
of $\mathrm{Aut}(G/H)$, then yields an
equivalence
	\[ 
	\mathsf{RMod}_{\mathrm{End}(W^{\Phi H})
	\text{-}\mathrm{Aut}(G/H)}
	\cong
	\mathsf{Mod}_{\mathrm{Aut}(G/H)}
	\]
Explicitly the equivalence is given by the functors:
	\[
	\mathsf{RMod}_{\mathrm{End}(W^{\Phi H})
	\text{-}\mathrm{Aut}(G/H)}\ni
	N \mapsto N \otimes_{\mathrm{End}(W^{\Phi H})} J_H
	\]
	\[
	\mathsf{Mod}_{\mathrm{Aut}(G/H)}
	\ni
	M \mapsto M \otimes J_H^{*}
	\]
where tensor products are given the diagonal action and
$J_H^{*}:= \mathrm{Hom}_{\mathbb{Z}}(J_H, \mathbb{Z})$. 

Combining this equivalence with the definition
of $\mathsf{RMod}_{\underline{R}^{\Phi}}$ we get
an equivalent category whose objects consist of
the data $\{M_{(G/H)}\}$ of a collection of $\mathrm{Aut}(G/H)$-modules
together with maps between them after tensoring
with iterations of the $J_K$ and their duals, satisfying various
properties. We call this category the category of 
\textbf{twisted Mackey functors}, denoted
$\mathsf{TwMack}_n$, and the precise definition
is contained in Definition \ref{defn:tw-mack} below.

In order to apply this algebra
 to homotopy theory,
it will be helpful to do some unraveling. If we choose a
(non-equivariant) $\mathbb{Z}$-summand of $J_H$,
then this produces an idempotent in $\mathrm{End}(J_H)$.
We can carry this across a chosen splitting to an idempotent
in $[W, W]^H$, and from there split off an $H$-equivariant
summand:
	\[
	W_{(G/H)} \to W \to W_{(G/H)}.
	\]
If $X$ is a $G$-spectrum, then $[W_{(G/H)}, X]^H$ still
has an $\mathrm{Aut}(G/H)$-action coming from
the one on (the restriction of) $X$. These modules
are essentially the objects $M_{(G/H)}$ described above.\footnote{
One may have to alter the action to account
for how the action on $J_H$ intertwines the chosen idempotent.
In our examples, we will usually have $W_{(G/H)} = \downarrow_HW$, in
which case we take the usual action on homotopy groups,
or we'll have $\downarrow_HW = W^{\Phi H}$,
in which case we may indeed take the trivial action on the
summand.}

If $K$ is subconjugate
to $H$, then composing the inclusions
and retractions gives $K$-equivariant maps:
	\[
	W_{(G/K)} \to W \to W_{(G/H)},
	\]
	\[
	W_{(G/H)} \to W \to W_{(G/K)}.
	\]

Using these, we get analogues of restriction and
transfer maps:
	\[
	[W_{(G/H)}, X]^H
	\stackrel{\mathrm{res}}{\longrightarrow} [W_{(G/H)}, X]^K
	\longrightarrow [W_{(G/K)}, X]^K,
	\]
	\[
	[W_{(G/K)}, X]^K \longrightarrow
	[W_{(G/H)}, X]^K
	\stackrel{\mathrm{tr}}{\longrightarrow}
	[W_{(G/H)}, X]^H.
	\]

The relations these maps satisfy satisfy are somewhat more
involved than their Mackey functor cousins, but
they are about as manageable in practice. The main
point of the algebra above is that these
maps are the only data necessary to describe the
associated $\underline{\mathrm{End}}(W)$-module,
and hence determine the $n$-slice of $X$.

We summarize this discussion in the following
paraphrased theorem:

\begin{theorem}[Theorem \ref{thm:slices-as-tw-mack}]
	\begin{enumerate}[(i)]
	\item The procedure above yields an equivalence
	of categories:
		\[
		\heartsuit_n \cong \mathsf{TwMack}_n.
		\]
	\item Under this equivalence, the category
	$\mathsf{Slice}_n$ corresponds to the full
	subcategory of $\mathsf{TwMack}_n$
	spanned by those objects $\{M_{(G/H)}\}$
	satisfying an explicit injectivity constraint
	on their restriction maps.
	\end{enumerate}
\end{theorem}

The upshot is the following procedure for computing
the $n$-slice of a $G$-spectrum $X$:

	\begin{enumerate}[Step 1.]
	\item Find or construct an isotropic slice $n$-sphere,
	$W$.
	\item Compute the $\mathrm{Aut}(G/H)$
	action on each $W^{\Phi H}$ to determine
	the modules $J_H$.
	\item Choose a summand of $J_H$ and
	determine the corresponding $H$-equivariant
	summand
	$W_{(G/H)}$ of $W$.
	\item Compute the groups $[W_{(G/H)}, X]^H$
	and the `twisted' restrictions and transfers
	between them.
	\item Make the requisite (direct sum of) restriction maps injective.
	\end{enumerate}

At this point the reader should be prepared for the examples
in \S\ref{sec:examples}.

\section{Filtrations on stratified categories}\label{sec:filtns}

A common strategy for proving statements about $\mathsf{Sp}^G$ is
to induct over the poset of subgroups using geometric fixed points,
and eventually reduce to a statement about non-equivariant spectra.
We can axiomatize
the structure necessary to make arguments like this and arrive at the
notion of a stratified homotopy theory, which we review in \S\ref{ssec:recoll}. 

In the situation of $\mathsf{Sp}^G$ we actually have \emph{two} methods for
reducing to non-equivariant considerations: geometric fixed points and genuine fixed points.
Again, we can distill the requisite properties into a definition, that of a homotopy
theory of Mackey functors on an inductive orbital category 
(Definitions \ref{defn:mack} and \ref{defn:ioc}).
We discuss this in \S\ref{ssec:mack}.

In \S\ref{ssec:slice-filtn}, we define slice filtrations
(Definition \ref{defn:slice-filtn}) associated to dimension functions
for homotopy 
theories of Mackey functors
and explore some of their elementary properties. 
We give a recognition theorem (Theorem \ref{thm:slice-recognition})
for comparing a given filtration
to the slice filtration associated to a dimension function.
As a corollary, we obtain
a streamlined proof of the theorem of Hill-Yarnall characterizing the original
slice filtration in terms of geometric fixed points (Corollary \ref{cor:HY-char}).

We pause now to collect some justifications for our chosen level of generality:
	\begin{itemize}
	\item When making inductive arguments about $G$-spectra, one is often
	led to consider homotopy theories associated to families of subgroups of $G$.
	These homotopy theories are usually not equivalent to $\mathsf{Sp}^H$ for
	any group $H$.
	\item There are several homotopy theories of Mackey functors that do not fall
	under the direct purview of equivariant homotopy theory, e.g.
	the homotopy theory of cyclotomic spectra \cite{BG-cycl, BM}
	and of $n$-excisive functors \cite{glas-calc}.
	Since it does not require extra work, it seems prudent to
	develop the theory in a way that applies to these examples.
	\item Even when restricting to $\mathsf{Sp}^G$, it is convenient
	to consider filtrations other than the standard slice filtration.
	For example, the \emph{regular} slice filtration was used to great
	effect by Ullman \cite{Ullman}
	in his thesis. We allow for yet further variants, so that one may choose
	a filtration suited to the application at hand.
	\end{itemize}

Finally, as we mentioned in the introduction, most of the results in \S1.3 were obtained
independently by Barwick-Dotto-Glasman-Nardin-Shah. We will try to indicate the major
overlaps where they occur.

\subsection{Review of recollements and stratifications}\label{ssec:recoll}

There are many situations where we study objects of a homotopy theory $\mathcal{C}$
by breaking it up into two pieces. Here are some examples.

\begin{itemize}
\item Let $\sigma$ denote the
sign representation of $C_2$.
If $X$ is a $C_2$-spectrum, then $X$ sits in a cofiber sequence
	\[
	S(\infty \sigma)_+ \wedge X \longrightarrow X\longrightarrow
	S^{\infty \sigma} \wedge X
	\]
called the \emph{isotropy separation sequence}. The first term has the property
that it is built out free $C_2$-spectra, in the sense that it has a  filteration
with associated graded a wedge of copies of suspensions of $C_{2+} \wedge X$. The
last term has the property that its underlying spectrum vanishes, so all its information
is contained in its fixed points. 

Alternatively, we can recover $X$ from the homotopy Cartesian square:
	\[
	\xymatrix{
	X \ar[r]\ar[d] & S^{\infty\sigma} \wedge X\ar[d]\\
	F(S(\infty \sigma)_+, X) \ar[r] & S^{\infty\sigma} \wedge F(S(\infty \sigma)_+, X)
	}
	\]
which is sometimes called the \emph{Tate fracture square}.
The bottom
piece of this square can be studied entirely in terms of the local system
on $BC_2$ which underlies $X$, and the left vertical map is a sort of
completion while the right horizontal map is a sort of localization.

\item Let $X$ be a space, $j: U \hookrightarrow X$ an open embedding
with closed complement $i: Y \hookrightarrow X$, and denote by
$\mathsf{Shv}(X; \mathsf{Sp})$ the homotopy theory of sheaves of spectra
on $X$. Then the restriction functor $j^*: \mathsf{Shv}(X; \mathsf{Sp}) \to
\mathsf{Shv}(U; \mathsf{Sp})$ admits a left adjoint $j_!$ (extension
by zero) and every sheaf
decomposes into a natural cofiber sequence:
	\[
	j_!j^*\mathcal{F} \longrightarrow \mathcal{F} \longrightarrow i_*i^*\mathcal{F}.
	\]
The first term is supported on $U$, and the last term is 
set-theoretically supported on $Y$
(in the sense that it is annhilated by $j^*$).

Alternatively, we can recover $\mathcal{F}$ from the homotopy Cartesian square:
	\[
	\xymatrix{
	\mathcal{F} \ar[r]\ar[d] & i_*i^*\mathcal{F} \ar[d]\\
	j_*j^*\mathcal{F} \ar[r]& i_*i^*j_*j^*\mathcal{F}
	}
	\]
The right vertical map provides
gluing data and is a generalization of a \emph{clutching function} from the
classical study of vector bundles.

\item If $M$ is a complex of abelian groups, then there is a
cofiber sequence in the derived category $\mathsf{D}(\mathbb{Z})$:
	\[
	\Gamma_pM \longrightarrow M \longrightarrow M\left[\frac{1}{p}\right]
	\]
where $\Gamma_pM$ has the property that every element of $H_k(\Gamma_pM)$
is annihilated by a power of $p$, and $M[1/p]$ has the property that $p$ acts
invertibly.

Alternatively, we can recover $M$ from the arithmetic fracture square:
	\[
	\xymatrix{
	M \ar[r]\ar[d] & M\left[\frac{1}{p}\right]\ar[d]\\
	\hat{L}M \ar[r] & (\hat{L}M)\left[\frac{1}{p}\right]
	}
	\]
Here, $\hat{L}M$ denotes the derived functor of $p$-completion, which
plays a prominent role in $K(1)$-local homotopy theory.
\end{itemize}

We collect some common features of these examples into a definition. It is
the evident adaptation of the Grothendieck school's notion of a recollement to
our setting.

\begin{definition}\label{defn:recoll} (\cite[A.8.1]{HA}) 
Let $\mathcal{C}$ be an $\infty$-category
which admits finite limits, and let $\mathcal{C}_0, \mathcal{C}_1 
\subseteq \mathcal{C}$
denote full subcategories. We say that $\mathcal{C}$ is a \textbf{recollement} of
$\mathcal{C}_0$ and $\mathcal{C}_1$ if the following conditions are satisfied:
	\begin{enumerate}[(a)]
	\item The subcategories $\mathcal{C}_0$ and $\mathcal{C}_1$ are
	closed under equivalence.
	\item The inclusion functors 
	$\mathcal{C}_i \hookrightarrow \mathcal{C}$
	admit left adjoints $L_i$.
	\item The functors $L_0$ and $L_1$ are left exact.
	\item The functor $L_1$ carries every object of $\mathcal{C}_0$
	to a final object of $\mathcal{C}_1$.
	\item The functors $L_0$ and $L_1$ are jointly conservative.
	\end{enumerate}
\end{definition}

We will need a slight generalization of this definition where $\mathcal{C}$ is glued
together from more than just two subcategories. Recall that an
\textbf{interval} in a poset $\mathsf{P}$ is a subset $I \subseteq \mathsf{P}$
such that if $x,y \in I$ and $x<z<y$, then $z \in I$.

\begin{definition}[\cite{Glas}] \label{def-strat-cats}
Let $\mathcal{C}$ be an $\infty$-category which admits
finite limits, and let $\mathsf{P}$ be a poset. Then a \textbf{$\mathsf{P}$-stratification}
of $\mathcal{C}$ is a map of posets
	\[
	\mathfrak{S}: \{\text{intervals in }\mathsf{P}\} \longrightarrow \{\text{reflective subcategories
	of }\mathcal{C}\}
	\]
such that:
	\begin{itemize}
	\item $\mathfrak{S}(\mathsf{P}) = \mathcal{C}$,
	\item $\mathfrak{S}(\varnothing) \subseteq \mathcal{C}$ is
	the full subcategory of final objects,
	\item whenever an interval $I\subseteq \mathsf{P}$
	is decomposed as $I = I_0 \amalg I_1$ in such a way that no element
	of $I_0$ is strictly greater than any element of $I_1$, then
	$\mathfrak{S}(I)$ is a recollement of $\mathfrak{S}(I_0)$ and
	$\mathfrak{S}(I_1)$.
	\end{itemize}
\end{definition}

These definitions are meant to generalize the fracture squares that appeared in
our motivating examples. In order to get the entire package, cofiber sequences
and all, we need to move to a stable setting. In this case, there is yet another
characterization of stratifications.

\begin{theorem} Let $\mathcal{C}$ be a stable $\infty$-category and let
	\[
	\mathfrak{S}_0: \{\text{downward closed intervals in }\mathsf{P}\}
	\longrightarrow \{\text{full subcategories of }\mathcal{C}\text{ closed
	under equivalences}\}
	\]
be a map of posets. The following conditions are equivalent:
\begin{enumerate}[(i)]
\item Each $\mathfrak{S}_0(I)$ is reflective and coreflective; i.e. the inclusion
admits both left and right adjoints, and $\mathfrak{S}_0(\mathsf{P}) = \mathcal{C}$.
\item The function $\mathfrak{S}_0$ extends to a $\mathsf{P}$-stratification
of $\mathcal{C}$.
\end{enumerate}
If either of these conditions are satisfied then, when $I$ is downward closed and
$J$ is its upwardly closed complement, we may identify $\mathfrak{S}(J)$
with the full subcategory of $\mathcal{C}$ spanned by those objects $Y$ such
that the mapping space $\mathrm{map}_{\mathcal{C}}(X, Y)$ is contractible
for each $X \in \mathfrak{S}(I)$.
\end{theorem}
\begin{proof} The proof in Barwick-Glasman \cite[Lemma 3]{recoll} works just as well in
the case of a poset.
\end{proof}

When $\mathsf{P} = \Delta^1$ it will be useful to set some notation down for
the morass of adjoints in play. The reader is encouraged to keep in mind the
example where $X$ is a space, $j: U\hookrightarrow X$
is an open embedding, and $i: Y=X-U \hookrightarrow X$ is the closed
complement. Then $\mathcal{C} = \mathsf{Shv}(X; \mathsf{Sp})$, $\mathcal{C}_0
=\mathsf{Shv}(Y; \mathsf{Sp})$, and $\mathcal{C}_1 = 
\mathsf{Shv}(U;\mathsf{Sp})$.

\begin{remark} In the case where $\mathcal{C}$ is
stable and a recollement of $\mathcal{C}_0$
and $\mathcal{C}_1$, we have the following diagram where each arrow is left
adjoint to the arrow below it:
	\[
	\xymatrix{
	\mathcal{C}_1 \ar@<3ex>[r]^{j_!}\ar@<-3ex>[r]^{j_*}& \mathcal{C} \ar[l]_{j^*}
	\ar@<3ex>[r]^{i^*}\ar@<-3ex>[r]^{i^!}& \mathcal{C}_0\ar[l]_{i_*}
	}
	\]
With this notation, $\mathcal{C}$ is the oplax limit \cite[2.8]{lax}
of the
exact functor $i^*j_*: \mathcal{C}_1 \to \mathcal{C}_0$. The $\infty$-category
$\mathcal{C}_1$ is now embedded into $\mathcal{C}$ in two different ways
as a full subcategory: first as a reflective
subcategory via $j_*$, as in the definition of a recollement, but also
as a coreflective subcategory via $j_!$. The essential images are characterized
as the two different orthogonal complements of $\mathcal{C}_0$. We will
denote
the image of $j_*$ by $\mathcal{C}_1^{\wedge}$ and the image of $j_!$ by
$\mathcal{C}_1^{\vee}$. The notation is supposed to suggest that the reflective
subcategory contains \emph{complete} objects while the coreflective
subcategory contains \emph{nilpotent} objects. An explicit equivalence between
these categories is obtained by the two inverse functors:
	\[
	j_*j^*: \mathcal{C}_1^{\vee} \longrightarrow \mathcal{C}_1^{\wedge}
	\]
	\[
	j_!j^*: \mathcal{C}_1^{\wedge}\longrightarrow \mathcal{C}_1^{\vee}
	\]
See Barwick-Glasman \cite{recoll} for details.
\end{remark}

\begin{remark} A stratification of $\mathcal{C}$ gives rise to a lax
functor $\mathsf{P}^{op} \to \mathsf{Cat}_{\infty}$
recording the atomic localizations, and $\mathcal{C}$
can be recovered as the \emph{op}lax limit of this lax diagram.
In fact, this process yields an equivalence
between the homotopy theory of stratified $\infty$-categories
and the homotopy theory of lax functors
$\mathsf{P}^{op} \to \mathsf{Cat}_{\infty}$ (i.e.
locally cocartesian fibrations over $\mathsf{P}^{op}$).
Making this precise would take us too far afield, but for a version
of this reconstruction theorem without the language of lax functors
see \cite[3.18]{Glas}.
The idea is that a lax functor
out of $\mathsf{P}^{op}$ is the same data as an ordinary
functor out of the \emph{relaxation} of $\mathsf{P}^{op}$,
which is modeled on the poset of nonempty subsets of $\mathsf{P}^{op}$. 
\end{remark}

\subsection{Mackey functors}\label{ssec:mack}

In the homotopy theory of $G$-spectra there are two a priori unrelated
inductive approaches to understanding a $G$-spectrum $X$. On the one hand,
we can reduce questions about $X$ to questions about its geometric fixed points,
$X^{\Phi H}$. This is the point of view that motivated the previous section.
On the other hand, we can reduce questions about $X$ to questions about
its \emph{genuine} fixed points, $X^H$. At a key point below (namely in
our construction of isotropic slice spheres) we will utilize the interplay
between these two approaches. 

First, however, we develop a general setting where these two sorts of inductive
tools- genuine and geometric fixed points- can both be defined.

\begin{remark}
Many of the definitions and examples in the beginning of this section
are pulled directly from Glasman \cite{Glas}, with the
notable exception of Definition \ref{defn:ioc}.
\end{remark}

\begin{definition} An \textbf{epiorbital category} is
an essentially finite category $\mathcal{O}$ satisfying the following
conditions:
	\begin{itemize}
	\item Every morphism in $\mathcal{O}$ is an epimorphism.
	\item $\mathcal{O}$ admits pushouts and coequalizers.
	\end{itemize}
Define a relation on the set of isomorphism classes of objects in $\mathcal{O}$
by $[X]\ge [Y]$ if $\mathrm{Hom}(X, Y)$ is nonempty. It is easy to check that
this forms a poset, which we denote $\mathsf{P}_{\mathcal{O}}$. 
\end{definition}

Given an essentially small $\infty$-category $\mathcal{C}$, we will denote by
$\fin_{\mathcal{C}}$ the $\infty$-category obtained by freely adjoining finite
coproducts. An explicit model can be obtained as the full subcategory of
$\mathsf{Psh}(\mathcal{C})$ spanned by finite coproducts of representable functors.
We remark that, in the case when $\mathcal{C}$ is an ordinary category, it doesn't
matter if we use the $\infty$-category of presheaves of spaces, or the ordinary
category of presheaves of sets in this construction.

\begin{definition} An \textbf{orbital $\infty$-category} is an essentially small
$\infty$-category $\mathcal{O}$ such that $\fin_{\mathcal{O}}$ admits
pullbacks. We will often call elements of $\mathcal{O}$ \textbf{orbits}.
\end{definition}

\begin{proposition} (\cite[2.14]{Glas}) Every epiorbital category is orbital.
\end{proposition}

Unfortunately, neither of these two levels of generality is quite right for what we
need. We will prove
most of our theorems by induction
on the size of the poset of an epiorbital
category, but unfortunately sometimes
the inductive procedure
takes us outside the realm of
epiorbital categories. We will take a middle ground, and propose the following.

\begin{definition}\label{defn:ioc} An \textbf{inductive orbital category} is an essentially finite,
orbital, (discrete)
category $\mathcal{O}$
with the property that every endomorphism is an isomorphism.
Notice that the isomorphism classes of objects again form a poset, 
$\mathsf{P}_{\mathcal{O}}$.
\end{definition}

\begin{warning} In this definition we use `essentially finite' in the
sense of \emph{ordinary} category theory. Our inductive
orbital categories will generally \emph{not} be finite as $\infty$-categories,
i.e. they will not usually have only finitely many non-degenerate simplices.
\end{warning}

\begin{remark} In fact, the condition that the isomorphism classes
of objects form a poset under the relation 
	\[
	[X]\ge [Y]
	\iff
	 \mathrm{Hom}(X, Y)\ne \varnothing
	\]
is equivalent
to the condition that every endomorphism is an isomorphism.
Both of these, in turn, are equivalent to the condition that
the category admits a conservative map to a poset.
Categories in which every endomorphism is an
isomorphism are sometimes called EI-categories in the literature.
\end{remark}

\begin{remark} Barwick-Dotto-Glasman-Nardin-Shah
define and study the more general but related
notion of a \emph{perfect orbital $\infty$-category} which is likely the
correct setting for the sorts of inductive arguments used below. All of our results 
should hold in this generality with minimal change.
\end{remark}

We choose this definition because of the following closure properties.

\begin{lemma} Let $\mathcal{O}$ be an inductive orbital category.
\begin{enumerate}[(i)]
\item If $I \subseteq \mathsf{P}_{\mathcal{O}}$ is an interval, then
the corresponding full subcategory $\mathcal{O}_I \subseteq \mathcal{O}$
is an inductive orbital category.
\item It $T \in \fin_{\mathcal{O}}$, then $\mathcal{O}_{/T}$
is an inductive orbital category.
\end{enumerate}
\end{lemma}
\begin{proof} In both cases, it is clear that
every endomorphism is still an isomorphism, and
that the categories are still discrete,
so we just
need to check that these categories are orbital.
For (ii) this is immediate since
$\fin_{\mathcal{O}_{/T}} \cong \left(\fin_{\mathcal{O}}\right)_{/T}$. 
So we are left with (i). To that end, consider a pullback square in
$\mathsf{Fin}_\mathcal{O}$
	\[
	\xymatrix{
	U'\ar[r]\ar[d] & U\ar[d]\\
	V'\ar[r] & V
	}
	\]
where $U, V, V' \in \mathcal{O}_I$. We can write $U'$ as a 
finite coproduct $U' = \coprod_{\alpha \in A} S_\alpha$ where each $S_\alpha \in \mathcal{O}$.
Define $U'' \in \mathcal{O}_I$ to be the coproduct over the subset $B \subseteq A$
of $\beta$ with $S_{\beta} \in \mathcal{O}_I$. We claim that the diagram
	\[
	\xymatrix{
	U''\ar[r]\ar[d] & U\ar[d]\\
	V' \ar[r] & V
	}
	\]
is a pullback square in $\mathcal{O}_I$. Indeed, since each $S_i$
maps to $V' \in \mathcal{O}_I$, we know that $[S_i] \ge p$ for some
$p \in I$. Since $I$ is an interval, the only way that $[S_i]$ could
fail to be in $I$ is if every element of $I$ is not greater than $[S_i]$. 
So, for every $W \in \fin_{\mathcal{O}_{I}}$, $\mathrm{Hom}(W, S_i)$
is empty when $S_i \notin \mathcal{O}_I$. The claim now follows from the
string of isomorphisms, for $W \in \mathcal{O}_I$,
	\[
	\mathrm{Hom}(W, U')
	= \mathrm{Hom}(W, \coprod_{\alpha} S_\alpha)
	= \coprod_{\alpha\in A} \mathrm{Hom}(W, S_\alpha)
	= \coprod_{\beta \in B} \mathrm{Hom}(W, S_{\beta})
	= \mathrm{Hom}(W, U'').
	\]
\end{proof}

We now recall the examples of interest.

\begin{example} If $G$ is a finite group, then the category $\mathcal{O}_G$
of non-empty transitive $G$-sets, i.e. orbits,
is an epiorbital category. The partial order on isomorphism classes is opposite
to the poset of conjugacy classes of subgroups. That is:
	\[
	[G/H] \ge [G/K] \iff H\text{ is subconjugate to }K.
	\]
More generally, any subgroup $H \subseteq
G$ yields a full subcategory $\mathcal{O}_{G/H} \subseteq \mathcal{O}_G$ of those
orbits with stabilizers that contain $H$ up to conjugacy (i.e. the full subcategory
corresponding to the downward closed
interval $\{[T]\le [G/H]\}$). This is also an epiorbital category,
and when $N$ is normal the two possible interpretations of the symbol $\mathcal{O}_{G/N}$
agree.
\end{example}

\begin{example} Crucially, if $\mathcal{F}$ is a family of subgroups of $G$
closed under
conjugation and passage to subgroups, then the full subcategory
$\mathcal{O}_{\mathcal{F}} \subseteq \mathcal{O}_G$ of transitive $G$-sets
with stabilizers in $\mathcal{F}$ is an inductive orbital category. It is \emph{not} an
epiorbital category in general.
\end{example}

\begin{example} The category $\mathsf{Surj}_{\le n}$ of finite sets of cardinality
at most $n$ and surjective maps between them is an epiorbital category.
\end{example}

\begin{example} If $\mathcal{G}$ is an $\infty$-groupoid, then it is an orbital
$\infty$-category. Unfortunately, even when $\mathcal{G}$ is finite and discrete,
it is not epiorbital because coequalizers do not exist except in trivial cases.
In this case, however, it is an inductive orbital category.
\end{example}

Recall that the twisted arrow category
$\mathsf{TwArr}(\mathcal{C})$ of an $\infty$-category
$\mathcal{C}$ is a specific model of a left fibration
	\[
	\mathsf{TwArr}(\mathcal{C}) \longrightarrow \mathcal{C}^{op} \times \mathcal{C}
	\]
classifying the functor 
	\[
	\mathrm{map}_{\mathcal{C}}:
	\mathcal{C}^{op} \times \mathcal{C} \to \mathsf{Spaces}
	\]  

\begin{warning} Different authors have
different conventions for the twisted arrow category. For example, ours agrees
with Mac Lane, and with Barwick et. al., but is \emph{dual} to
the fibration used by Lurie and by Gepner-Haugseng-Nikolaus.
\end{warning}

\begin{example} When $\mathcal{C}= \Delta^n$,
the twisted arrow category is the poset of
intervals in $[n]$, ordered by inclusion. 
\end{example}

\begin{definition}[Barwick \cite{mackI}]
 For any orbital $\infty$-category, $\mathcal{O}$,
define a simplicial set $\aeff(\mathcal{O})$
by declaring the $n$-simplices to be the set
of functors $F: \mathsf{TwArr}(\Delta^n)^{op}
\to \fin_{\mathcal{O}}$ such that, for all $0\le i \le j \le k \le \ell \le n$, the
square
	\[
	\xymatrix{
	F_{i\ell} \ar[r]\ar[d] & F_{ik}\ar[d]\\
	F_{j\ell}\ar[r] & F_{jk} 
	}
	\]
is a pullback. We call $\aeff(\mathcal{O})$ the \textbf{effective Burnside
$\infty$-category of $\mathcal{O}$}.
\end{definition}

\begin{remark} Even though $\mathcal{O}$ is a discrete category, 
$\aeff(\mathcal{O})$ will not be discrete in general.
Instead, in this case $\aeff(\mathcal{O})$
will be a $(2,1)$-category. The homotopy
category $\mathrm{h}\aeff(\mathcal{O})$ has a more familiar
description: the objects are objects of $\fin_{\mathcal{O}}$, the
morphisms are \emph{isomorphism classes} of spans, and composition
is given by pullback.
\end{remark}

We will need a condition on the targets of our Mackey functors.

\begin{definition} Let $\mathcal{C}$ be an
$\infty$-category which admits finite products and finite coproducts.
Suppose moreover that $\mathcal{C}$ is pointed, i.e. it admits
an object $0$ which is both initial and final.
	\begin{itemize}
	\item We say that $\mathcal{C}$ is \textbf{semi-additive}
	if for all $X, Y \in \mathcal{C}$, the canonical map
		\[
		\begin{pmatrix}
		\mathrm{id}_X & 0\\
		0 & \mathrm{id}_Y
		\end{pmatrix}:
		X \amalg Y \longrightarrow 
		X \times Y
		\]
	is an equivalence. In this case we denote finite coproducts
	and products by $X \oplus Y$ and refer to them
	as direct sums.
	\item We say that $\mathcal{C}$ is
	\textbf{additive} if it is semi-additive 
	and each of the resulting commutative monoids
	$[X, Y]$ have inverses.
	\end{itemize}
\end{definition}

\begin{example} Any ordinary abelian category, viewed
as an $\infty$-category, is additive.
\end{example}

\begin{example} Any stable $\infty$-category is additive.
\end{example}

\begin{example} By Barwick \cite[4.3]{mackI}, the $\infty$-category $\aeff(\mathcal{O})$
is semi-additive.
\end{example}

\begin{definition}\cite{mackI}\label{defn:mack}
If $\mathcal{C}$ is a semi-additive $\infty$-category and
$\mathcal{O}$ is an orbital $\infty$-category, we denote by
$\mathsf{Mack}(\mathcal{O}; \mathcal{C})$ the $\infty$-category
of functors $\aeff(\mathcal{O}) \to \mathcal{C}$ which
preserve finite direct sums. 

When $\mathcal{C}=\mathsf{Sp}$
we will denote this $\infty$-category more simply by $\mathsf{Sp}^{\mathcal{O}}$.
\end{definition}

\begin{example} If $\mathcal{G}$ is a connected
groupoid, then $\mathsf{Mack}(\mathcal{G}; \mathcal{C})$
is canonically equivalent to $\mathsf{Fun}(\mathcal{G}, \mathcal{C})$ \cite[2.27]{Glas}.
\end{example}

\begin{example} For a finite group, $G$,
the $\infty$-category $\mathsf{Mack}(\mathcal{O}_G; \mathsf{Sp})$
of spectral Mackey functors
is equivalent to the $\infty$-category underlying the model category
of orthogonal $G$-spectra (cf. Nardin or Guilloiu-May \cite{denis-stable, GuM}). 
This justifies the notation $\mathsf{Sp}^G$.
\end{example}

\begin{example} The $\infty$-category $\mathsf{Mack}(\mathsf{Surj}_{\le n}; \mathcal{C})$
is equivalent to the $\infty$-category of $n$-excisive functors from
$\mathsf{Sp}$ to $\mathcal{C}$. This is proven by Glasman in
\cite{glas-calc}, and in unpublished work of Dwyer-Rezk.
\end{example}

Let $I \subseteq \mathsf{P}_{\mathcal{O}}$ denote a downward closed interval,
and denote by $\mathsf{Mack}(\mathcal{O}; \mathcal{C})_{\Phi I}$ 
the full
subcategory of $\mathsf{Mack}(\mathcal{O}; \mathcal{C})$ spanned by
those functors which take each element
in the complement of $I$ to a zero object in $\mathcal{C}$.

\begin{theorem}[\cite{Glas}] \label{thm-mack-strat}
If $\mathcal{O}$ is an inductive orbital category, 
and $\mathcal{C}$ is stable, then the stable $\infty$-category
$\mathsf{Mack}(\mathcal{O}; \mathcal{C})$ admits a canonical 
$\mathsf{P}_{\mathcal{O}}$-stratification
which, for downward closed intervals, takes the form:
	\[
	\mathfrak{S}_0(I) := \mathsf{Mack}(\mathcal{O}; \mathcal{C})_{\Phi I}.
	\]
\end{theorem}
\begin{proof} Either observe that the proof given by Glasman \cite{Glas} for epiorbital
categories works verbatim for inductive orbital categories, or else
apply Proposition 3.13 of \cite{Glas} to each downward closed $I$ separately
to deduce that $\mathfrak{S}_0(I)$ is the reflective and coreflective
piece of a recollement.
\end{proof}

\begin{warning} This theorem is false if $\mathcal{C}$ is only assumed
to be additive. The issue
is that the left adjoint to the inclusion 
$\mathsf{Mack}(\mathcal{O}; \mathcal{C})_{\Phi I} \subseteq \mathsf{Mack}(\mathcal{O};
\mathcal{C})$ is not exact in general. If $\mathcal{C}$ is abelian, then a related fact is true using the theory
of stratifications of abelian categories (which is not the same
as the notion of stratification we are using.) Unfortunately,
recollements in the theory of abelian categories are less well-behaved
than their $\infty$-categorical counterparts: it is not possible, in general,
to recover a stratified abelian category from its atomic localizations,
even in the case of a recollement.
\end{warning}

We will identify the strata of this stratification momentarily. But first we
take some time to introduce a lot of notation.

\begin{notation}\label{notation-so-many-adjoints}
We now collect together our conventions on the various
functors that show up in the theory of Mackey functors. We wouldn't make
such a fuss, but we will use essentially all of these at some point.

Below, $\mathcal{O}$
will denote an inductive orbital category unless otherwise specified. All Mackey
functors will take values in a fixed presentable, semi-additive
$\infty$-category $\mathcal{C}$
which we suppress (temporarily breaking our convention) to avoid yet more clutter.
	\begin{enumerate}[(a)]
	\item If $\mathcal{F} \subseteq 
	\mathsf{P}_\mathcal{O}$ is \emph{upward closed}, then
	let $j_{\mathcal{F}}: \mathcal{O}_{\mathcal{F}} \hookrightarrow \mathcal{O}$
	denote the inclusion of the full subcategory spanned
	by objects whose isomorphism class lies in $\mathcal{F}$. The
	functor $j_{\mathcal{F}}^{\amalg}$ preserves pullbacks, so we get
	adjoint functors:
		\[
		\xymatrix{
		(j_\mathcal{F})_! : \mathsf{Mack}(\mathcal{O}_{\mathcal{F}}) \ar@<.7ex>[r]&
		\mathsf{Mack}(\mathcal{O}) : (j_{\mathcal{F}})^* \ar@<.7ex>[l]
		}
		\]
		\[
		\xymatrix{
		(j_{\mathcal{F}})^*: \mathsf{Mack}(\mathcal{O}) \ar@<.7ex>[r]&
		\mathsf{Mack}(\mathcal{O}_{\mathcal{F}}) : (j_{\mathcal{F}})_*\ar@<.7ex>[l]
		}
		\]
	given by left Kan extension, restriction, and right Kan extension, respectively.
	
	\item If $\widetilde{\mathcal{F}} \subseteq \mathsf{P}_{\mathcal{O}}$ 
	is \emph{downward
	closed}, then let 
	$\psi_{\widetilde{\mathcal{F}}}: \mathcal{O}_{\widetilde{\mathcal{F}}}
	\hookrightarrow \mathcal{O}$ denote the inclusion of the evident full subcategory. 
	The associated embedding
	$\psi_{\widetilde{\mathcal{F}}}^{\amalg}$ admits a right adjoint
	$i_{\widetilde{\mathcal{F}}}$. We then get the following adjoint pairs:
		\[
		\xymatrix{
		(i_{\widetilde{\mathcal{F}}})^*: 
		\mathsf{Mack}(\mathcal{O}) \ar@<.7ex>[r]&
		\mathsf{Mack}(\mathcal{O}_{\widetilde{\mathcal{F}}})
		: (i_{\widetilde{\mathcal{F}}})_* \ar@<.7ex>[l]
		}
		\]
		\[
		\xymatrix{
		(i_{\widetilde{\mathcal{F}}})_*: 
		\mathsf{Mack}(\mathcal{O}_{\widetilde{\mathcal{F}}}) \ar@<.7ex>[r]&
		\mathsf{Mack}(\mathcal{O})
		: (i_{\widetilde{\mathcal{F}}})^! \ar@<.7ex>[l]
		}
		\]
	We note that, perhaps confusingly, $(i_{\widetilde{\mathcal{F}}})^*$ is given by
	left Kan extension.
	
	If $\mathcal{F}$ is the upward closed complement of $\widetilde{\mathcal{F}}$
	we will sometimes abuse notation and denote by $\Phi^{\mathcal{F}}X$
	\emph{either} $(i_{\widetilde{\mathcal{F}}})^*$ or 
	$(i_{\widetilde{\mathcal{F}}})_*i_{\widetilde{\mathcal{F}}}^*$ when
	we believe there is no chance of confusion.
	
	When $\widetilde{\mathcal{F}} = (-\infty, T]$ is the set of all $p \le [T]$ for 
	some $T \in \mathcal{O}$, then 
	we denote $i_{\widetilde{\mathcal{F}}}$ by $i_{T}$. We will
	sometimes denote the value
	$(i_{T})^*X(T)$ by 
	$X^{\Phi T} \in \mathcal{C}$.\footnote{This leads to an unfortunate
	clash with the standard equivariant notation, 
	but we don't know of a way to avoid it.}
	
	\item In the event that $\psi^{\amalg}_{\widetilde{\mathcal{F}}}$ preserves
	pullbacks, we get even more:
		\[
		\xymatrix{
		(\psi_{{\widetilde{\mathcal{F}}}})_!:
		\mathsf{Mack}(\mathcal{O}_{\widetilde{\mathcal{F}}}) \ar@<.7ex>[r]&
		\mathsf{Mack}(\mathcal{O})
		: (\psi_{{\widetilde{\mathcal{F}}}})^* \ar@<.7ex>[l]
		}
		\]
		\[
		\xymatrix{
		(\psi_{{\widetilde{\mathcal{F}}}})^*:
		\mathsf{Mack}(\mathcal{O}) \ar@<.7ex>[r]&
		\mathsf{Mack}(\mathcal{O}_{\widetilde{\mathcal{F}}})
		:(\psi_{{\widetilde{\mathcal{F}}}})_* \ar@<.7ex>[l]
		}
		\]
	\item Given an object $T \in \fin_{\mathcal{O}} \subseteq \mathsf{Psh}(\mathcal{O})$,
	form the category $\mathcal{O}_{/T}$ of pairs $(x, f)$ where $x \in \mathcal{O}$
	and $f \in T(x)$. This is also an inductive orbital category, and the map
	$\mathcal{O}_{/T} \to \mathcal{O}$ 
	induces a restriction map
	$\mathrm{res}_T: \mathsf{Mack}(\mathcal{O}) \to 
	\mathsf{Mack}(\mathcal{O}_{/T})$.
	The restriction admits both a left and 
	right adjoint (given by left and right Kan extension
	respectively).
		\[
		\xymatrix{
		\mathrm{ind}_T: \mathsf{Mack}(\mathcal{O}_{/T}) \ar@<.7ex>[r]&
		\mathsf{Mack}(\mathcal{O})
		 :\mathrm{res}_T \ar@<.7ex>[l]
		 }
		\]
		\[
		\xymatrix{
		\mathrm{res}_T : \mathsf{Mack}(\mathcal{O}) \ar@<.7ex>[r]&
		\mathsf{Mack}(\mathcal{O}_{/T})
		 : \mathrm{coind}_T \ar@<.7ex>[l]
		 }
		\]
	A key feature of Mackey functors with values in an additive $\infty$-category
	(or, more generally, a semiadditive $\infty$-category) is that the canonical map
	$\mathrm{ind}_T \to \mathrm{coind}_T$ is an equivalence.
	We will often abbreviate (co)induction and restriction by $\uparrow_T$ and 
	$\downarrow_T$, possibly decorated further when there is ambiguity.
	\end{enumerate}
\end{notation}

\begin{remark}\label{rem-induct} It will be very useful in inductive arguments to note that
the poset $\mathsf{P}_{\mathcal{O}_{/T}} \subseteq \mathsf{P}_{\mathcal{O}}$
is strictly smaller unless $T$ contains a representative of each minimal object
as a retract.
If $T = \coprod_i T_i$ for $T_i \in \mathcal{O}$, then
$\mathsf{P}_{\mathcal{O}_{/T}} = \cup_i \{p \le T_i\}$. Beware, however,
that $\mathcal{O}_{/T}$ is \emph{not}
the same as $\mathcal{O}_{\cup_i \{p\le T_i\}}$, using the notation in (b) above.
The latter is the essential image of the former under the projection
$\mathcal{O}_{/T} \to \mathcal{O}$, but the projection
is not full in general.
\end{remark}

\begin{remark} The condition that $\psi_{\widetilde{\mathcal{F}}}^{\amalg}$ preserve
pullbacks is
satisfied in the following important cases:
	\begin{enumerate}[(i)]
	\item whenever $\widetilde{\mathcal{F}}$ is a set consisting of minimal elements
	in $\mathsf{P}_{\mathcal{O}}$,
	\item when $\mathcal{O} = \mathcal{O}_G$ and $\widetilde{\mathcal{F}} \subseteq
	\mathsf{P}_{\mathcal{O}_G}$ is an arbitrary downwardly closed subset.
	\end{enumerate}
We will use (i) frequently. 

It is not true in general
that $\psi_{\widetilde{\mathcal{F}}}^{\amalg}$ preserves pullbacks.
For example, this fails in the case $\mathcal{O}_{\widetilde{\mathcal{F}}}
= \mathsf{Surj}_{\le n}
\subseteq \mathsf{Surj}_{\le n+1}$ when $n>1$. 
\end{remark}

\begin{definition} With notation as above, 
we will refer to the essential image of $(j_{\mathcal{F}})_!$ as the \textbf{subcategory
of $\mathcal{F}$-nilpotent} objects, the essential image
of $(j_{\mathcal{F}})_*$ as the \textbf{subcategory of
$\mathcal{F}$-complete} objects, and the essential image of
$(i_{\widetilde{\mathcal{F}}})_*$ as the \textbf{subcategory of
$\widetilde{\mathcal{F}}$-geometric} objects. We denote these in the following way:
	\[
	\mathsf{Mack}(\mathcal{O}; \mathcal{C})^{\mathcal{F}-\mathrm{nil}}
	= \text{subcategory of }\mathcal{F}\text{-nilpotent objects},
	\]
	\[
	\mathsf{Mack}(\mathcal{O}; \mathcal{C})^{\mathcal{F}-\mathrm{cpl}}
	= \text{subcategory of }\mathcal{F}\text{-complete objects},
	\]
	\[
	\mathsf{Mack}(\mathcal{O}; \mathcal{C})_{\Phi\widetilde{\mathcal{F}}}
	= \text{subcategory of }\widetilde{\mathcal{F}}\text{-geometric objects}.
	\]
\end{definition}

The following is straightfoward from the definitions and the non-trivial
\cite[2.27]{Glas}.

\begin{lemma} Let $I \subseteq \mathcal{O}$ be an interval, and write
it as $I = \mathcal{F} \cap \widetilde{\mathcal{F}}$ 
where $\mathcal{F}$ and $\widetilde{\mathcal{F}}$ are
the smallest upward closed and downward closed intervals, respectively,
containing $I$.
In the stratification of $\mathsf{Mack}(\mathcal{O};\mathcal{C})$
determined by Theorem \ref{thm-mack-strat}, 
	\[
	\mathfrak{S}(I) = \mathsf{Mack}(\mathcal{O}; \mathcal{C})^{\mathcal{F}-\mathrm{cpl}}
	\cap \mathsf{Mack}(\mathcal{O}; \mathcal{C})_{\Phi\widetilde{\mathcal{F}}}.
	\]
When $I=\{T\}$ for $T \in \mathcal{O}$ we can identify this intersection
with $\mathrm{Fun}(\mathrm{Aut}(T), \mathcal{C})$.
\end{lemma}

To make the notation more memorable, we instantiate each symbol in 
the example that the reader likely cares about.

\begin{example}[Equivariant spectra] Let $G$ be a finite group. Note that
$\mathsf{P}_{\mathcal{O}_G} = \mathsf{Sub}^{op}_G$ is the poset of conjugacy
classes of subgroups
of $G$ ordered by reverse inclusion. 
	\begin{itemize}
	\item An upward closed subset of $\mathsf{Sub}^{op}_G$ is just
	a \emph{family of subgroups} in the sense of, e.g., tom Dieck \cite[7.2]{td}.
	So there is a universal $G$-space $\mathrm{E}\mathcal{F}$ for the family,
	characterized by the property that
		\[
		(\mathrm{E}\mathcal{F})^H \cong
		\begin{cases}
		* & H \in \mathcal{F} \\
		\varnothing & H \notin \mathcal{F}
		\end{cases}.
		\]
	We then have identifications:
		\[
		(j_{\mathcal{F}})_! = \mathrm{E}\mathcal{F}_+ \wedge (-),
		\]
		\[
		(j_{\mathcal{F}})_* = F(\mathrm{E}\mathcal{F}_+, -).
		\]
	\item A downward closed subset of $\mathsf{Sub}^{op}_G$ is the complement
	of a family $\mathcal{F}$
	of subgroups. We can then form the $G$-space $\widetilde{\mathrm{E}\mathcal{F}}$
	as the cofiber:
		\[
		\mathrm{E}\mathcal{F}_+ \to S^0 \to \widetilde{\mathrm{E}\mathcal{F}}.
		\]
	The various functors in \ref{notation-so-many-adjoints}(b) are given classically
	by:
		\begin{align*}
		(i_{\widetilde{\mathcal{F}}})_* &= \widetilde{\mathrm{E}\mathcal{F}} \wedge (-) \\
		(i_{\widetilde{\mathcal{F}}})^* 
		&= \left(\widetilde{\mathrm{E}\mathcal{F}} \wedge (-)\right)^{\mathcal{F}}
		= \Phi^{\mathcal{F}}(-)\\
		(i_{\widetilde{\mathcal{F}}})^! &= F(\widetilde{\mathrm{E}\mathcal{F}}, -)^{\mathcal{F}}\\
		(\psi_{\widetilde{\mathcal{F}}})^* &= (-)^{\mathcal{F}}\\
		\end{align*}
	Where $(-)^{\mathcal{F}}$ is the Lewis-May
	categorical fixed point functor \cite[I.3]{LMS}.
	The functors
		\[
		(\psi_{\widetilde{\mathcal{F}}})_*, \,\, (\psi_{\widetilde{\mathcal{F}}})_!
		\]
	give two different ways of taking an object with some amount of symmetries,
	and adding more. The latter is probably more familiar, and corresponds, in the
	case when $\mathcal{F}$ is the set of subgroups subconjugate to $H$, 
	to the process
	of taking a $(N_GH/H)$-spectrum, regarding it as an $N_GH$ spectrum, and then
	inducing up to $G$.
	\item An object $T \in \fin_{\mathcal{O}_G}$ is just a finite $G$-set, so
	we will feel no guilt denoting the category instead by $\fin_G$ from now on.
	If $T = G/H$, then $\mathcal{O}_{/T}$ is equivalent
	to the orbit category $\mathcal{O}_H$.
	(Co)induction and restriction are as you'd expect. The asserted equivalence
	between in induction and coinduction
	is a special case of the Wirthm\"uller isomorphism.
	\end{itemize}
\end{example}

We end this section by recording some properties and relations between these
functors for later use.

\begin{lemma}\label{lem-geo-ind-res} Fix $T \in \fin_{\mathcal{O}}$ and
let $\widetilde{\mathcal{F}} \subseteq \mathsf{P}_{\mathcal{O}}$ be a downward
closed family. Let $\widetilde{\mathcal{F}}_T = \widetilde{\mathcal{F}} \cap
\mathsf{P}_{\mathcal{O}_{/T}}$. Then there are essentially canonical commutative
diagrams:
	\[
	\xymatrix{
	\mathsf{Mack}(\mathcal{O}_{/T}) \ar[r]^{\mathrm{ind}_T}
	\ar[d]_{(i_{\widetilde{\mathcal{F}}_T})^*} &
	\mathsf{Mack}(\mathcal{O})\ar[d]^{(i_{\widetilde{\mathcal{F}}})^*}\\
	\mathsf{Mack}(\mathcal{O}_{\widetilde{\mathcal{F}}_T})
	\ar[r]_{\mathrm{ind}_T}
	& \mathsf{Mack}(\mathcal{O}_{\widetilde{\mathcal{F}}})
	}
	\]
	\[
	\xymatrix{
	\mathsf{Mack}(\mathcal{O}) \ar[r]^{\mathrm{res}_T}
	\ar[d]_{(i_{\widetilde{\mathcal{F}}})^*} &
	\mathsf{Mack}(\mathcal{O}_{/T})\ar[d]^{(i_{\widetilde{\mathcal{F}}_T})^*}\\
	\mathsf{Mack}(\mathcal{O}_{\widetilde{\mathcal{F}}})
	\ar[r]_{\mathrm{res}_T}
	& \mathsf{Mack}(\mathcal{O}_{\widetilde{\mathcal{F}}_T})
	}
	\]
In particular, if $[T] \notin \widetilde{\mathcal{F}}$, then
$(i_{\widetilde{\mathcal{F}}})^*\mathrm{ind}_T = 0$.
\end{lemma}

The next proposition is a generalization of the fact that,
in equivariant homotopy theory,
$\mathrm{E}\mathcal{F}_+ \wedge X$ is 
built out of inductions of restrictions of $X$ to subgroups in the family $\mathcal{F}$.

\begin{proposition}\label{prop-nil-induced}
Let $\mathcal{F} \subseteq \mathsf{P}_\mathcal{O}$ be upward closed.
Let $T = \coprod_{S \in \mathcal{O}_{\mathcal{F}}} S \in \fin_{\mathcal{O}_{\mathcal{F}}}
\subseteq \fin_{\mathcal{O}}$. Then there is a functor
	\[
	\mathbf{L}j_! : \mathsf{Mack}(\mathcal{O}_{\mathcal{F}})
	\to \mathsf{Fun}(\Delta^{op}, \mathsf{Mack}(\mathcal{O}))
	\]
with the following properties:
	\begin{enumerate}[(i)]
	\item for each $n\ge 0$, we have $(\mathbf{L}j_!)_n \cong (\mathrm{ind}_T
	\circ \mathrm{res}_T)^{\circ n+1}$,
	\item there is a natural equivalence of functors
		\[
		\colim_{\Delta^{op}} \mathbf{L}j_! \cong j_!.
		\]
	\end{enumerate}
\end{proposition}
\begin{proof} We note that $(\mathcal{O}_{\mathcal{F}})_{/T} = \mathcal{O}_{/T}$
since there are no maps from smaller objects to larger objects. This means
that the target of the restriction functors associated to $T$ agree.

To avoid ambiguity, we will temporarily denote by $\mathrm{ind}'_T$
the induction functor with target $\mathsf{Mack}(\mathcal{O}_{\mathcal{F}})$ and
$\mathrm{ind}_T$ the induction functor with target $\mathsf{Mack}(\mathcal{O})$. 
The endofunctor 
	\[
	\mathrm{ind}'_T \circ \mathrm{res}_T : \mathsf{Mack}(\mathcal{O}_{\mathcal{F}})
	\to \mathsf{Mack}(\mathcal{O}_{\mathcal{F}})
	\]
admits a canonical structure of a monad, and so we can form the bar construction
$\mathrm{Bar}: \mathsf{Mack}(\mathcal{O}_{\mathcal{F}}) \to 
\mathsf{Fun}(\Delta^{op}, \mathsf{Mack}(\mathcal{O}_{\mathcal{F}}))$
\cite[4.4.2.7]{HA}. We define $\mathbf{L}j_!$ as $j_! \circ \mathrm{Bar}$. To verify (i),
note that there is an equivalence $j_! \circ \mathrm{ind}'_T \cong \mathrm{ind}_T$
since a composite of left Kan extensions is a left Kan extension of the composite. To
verify (ii) it suffices, since $j_!$
preserves colimits, to check that $\colim_{\Delta^{op}} \mathrm{Bar} \cong \mathrm{id}$.
This follows from the fact that the adjunction $\mathrm{ind}'_T \dashv \mathrm{res}_T$
is monadic for our choice of $T$. Indeed, $\mathrm{res}_T$
preserves all colimits, so we
need only check that $\mathrm{res}_T:
\mathsf{Mack}(\mathcal{O}_{\mathcal{F}}) \to \mathsf{Mack}(\mathcal{O}_{/T})$ is conservative. 
But equivalences of Mackey
functors are detected objectwise, and every object in $\mathcal{O}_{\mathcal{F}}$ is accounted for in $T$.
\end{proof}

\subsection{Slice filtrations and basic properties}\label{ssec:slice-filtn}

In this section we develop a generalization of the slice
filtration suitable for stratified homotopy theories. 

\begin{definition} A \textbf{filtration} of a stable $\infty$-category $\mathcal{C}$ is
a sequence of full subcategories
	\[
	\cdots \mathcal{C}_{\ge n} \subseteq \mathcal{C}_{\ge n-1} \subseteq \cdots 
	\subseteq \mathcal{C}
	\]
such that each $\mathcal{C}_{\ge n}$ is coreflective in $\mathcal{C}$ and closed
under extensions. We say the filtration
is $\textbf{separated}$ if $\bigcap \mathcal{C}_{\ge n}$ is trivial.
We say that the filtration is \textbf{compatible with suspension} if
$\Sigma \mathcal{C}_{\ge n} \subseteq \mathcal{C}_{\ge n+1}$. Objects $X \in \mathcal{C}_{\ge n}$ will be called \textbf{$n$-connective} and we'll indicate this property by writing
$X \ge n$. We will call a filtration \textbf{presentable} if each
of the $\mathcal{C}_{\ge n}$
and $\mathcal{C}$ are presentable.

If $\mathcal{C}$ and $\mathcal{C}'$ are equipped with filtrations and $F: \mathcal{C}
\to \mathcal{C}'$ is a functor we will say that $F$ is \textbf{filtration preserving} if
$F(\mathcal{C}_{\ge n}) \subseteq \mathcal{C}'_{\ge n}$.
\end{definition}

\begin{example} If $\mathcal{C}$ has a $t$-structure then the
sequence of subcategories $\{\tau_{\ge n}\mathcal{C}\}$
is a filtration on $\mathcal{C}$ compatible with suspensions. 
If $\mathcal{C}$ admits countable products and $\tau_{\ge0}\mathcal{C}$
is stable under these, then separability of the filtration is equivalent
to left completeness of the $t$-structure \cite[1.2.1.19]{HA}.
If $\mathcal{C}$ is presentable, then the filtration is presentable
if and only if the $t$-structure is accessible in the sense of  \cite[1.4.4.12]{HA}.
\end{example}

Not every filtration arises from a $t$-structure, but every presentable
filtration gives rise to a \emph{sequence} of $t$-structures.

\begin{definition} Let $\{\mathcal{C}_{\ge n}\}$ be a presentable
filtration on a presentable, stable $\infty$-category $\mathcal{C}$.
Then each subcategory $\mathcal{C}_{\ge n}$ determines
an accessible $t$-structure with $\mathcal{C}_{\ge n}$
as the subcategory of 0-connective objects for that $t$-structure.
The heart is a Grothendieck abelian category \cite[1.3.5.23]{HA}
which we denote by $\mathcal{C}^{\heartsuit_n}$.
If $\mathcal{C}$ is understood, we will abbreviate this to $\heartsuit_n$.
We denote the truncation functors associated to the $n$th $t$-structure
by $\tau^{(n)}_{\le k}$ and $\tau^{(n)}_{\ge k}$ for $k \in \mathbb{Z}$. 
\end{definition}

We offer the following generalization of a perversity suited to our examples.

\begin{definition}\label{def-dim-func} Let $\mathsf{P}$ be a poset.
A \textbf{dimension function} for $\mathsf{P}$ is a function:
	\[
	\nu: \mathbb{Z} \times \mathsf{P} \longrightarrow \mathbb{Z}
	\]
such that for any $p \in \mathsf{P}$, $\nu(-,p): \mathbb{Z} \to \mathbb{Z}$
	is weakly increasing and surjective. We say that $p \in \mathsf{P}$ is an
	\textbf{$n$-jump} if $\nu(n+1, p)> \nu(n, p)$, otherwise we say
	that $p$ is an \textbf{$n$-rest}. We say that \textbf{$\nu$ jumps at $n$}
	if every $p \in \mathsf{P}$ is an $n$-jump.
\end{definition}

\begin{remark} Barwick-Dotto-Glasman-Nardin-Shah study
the almost identital notion of a \emph{generalized perversity}
in their forthcoming work.
\end{remark}

\begin{definition}\label{defn:slice-filtn}
Suppose given a presentable, stable, $\mathsf{P}$-stratified
$\infty$-category $\mathcal{C}$, a dimension function $\nu$ for $\mathsf{P}$,
and separated, presentable filtrations on the strata
$\mathcal{C}_p$ for each $p \in \mathsf{P}$,
compatible with suspension. Denote by $L_p$ the localization $L_p: \mathcal{C}
\to \mathcal{C}_p$. Then define the \textbf{$\nu$-slice filtration} on $\mathcal{C}$ by declaring
$X \ge n$ if and only if $L_pX \ge \nu(n, p)$ for all $p \in \mathsf{P}$. Attached
to this filtration we will use the following terminology:
	\begin{itemize}
	\item We will say $X$ is \textbf{slice $n$-connective} and
	write $X \ge n$ if $L_pX\ge \nu(n, p)$
	for all $p \in \mathsf{P}$. The full subcategory of slice $n$-connective
	objects is denoted $^{\nu}\mathcal{C}_{\ge n}$ or just $\mathcal{C}_{\ge n}$
	if $\nu$ is understood.
	\item We denote the right adjoint to the inclusion $\mathcal{C}_{\ge n} \subseteq
	\mathcal{C}$ by $P_{n}$ and call $P_{n}X$ the \textbf{slice $n$-connective
	cover of $X$}.
	\item We will say $Y$ is \textbf{slice $n$-truncated}
	and write $Y \le n$ if, for every $X \ge n+1$,
	the mapping space $\mathrm{map}_{\mathcal{C}}(X, Y)$ is contractible.
	We denote the full subcategory of slice $n$-truncated objects by
	$^{\nu}\mathcal{C}_{\le n}$ or $\mathcal{C}_{\le n}$ if $\nu$ is understood.
	\item We denote the cofiber of $P_{n+1} \to \mathrm{id}$ by $P^n$ and call
	$P^nX$ the \textbf{$n$th slice section of $X$} or the \textbf{slice $n$-stage of $X$}.
	\item We will say $A$ is an \textbf{$n$-slice} if $A\le n$ and $A\ge n$.
	We denote the full subcategory of $n$-slices by $\mathsf{Slice}_n$, and
	will further decorate this symbol if either $\nu$ or $\mathcal{C}$ is unclear
	from the context.
	\item There is a canonical equivalence $P^nP_n \cong P_nP^n$ and
	we denote either of these functors by $P^n_n$. We call $P^n_nX$ the
	\textbf{$n$-slice of $X$}.
	\end{itemize}
\end{definition}

We record a few basic consequences of the definition before turning to
examples.

\begin{lemma} The $\nu$-slice filtration is indeed
a filtration. As such, it is separated and compatible
with suspension.
\end{lemma}
\begin{proof} Each $\mathcal{C}_{\ge n}$ is closed under colimits, extensions,
and equivalence since $L_p$ preserves colimits. Since we've assumed $\mathcal{C}$
is presentable, this provides the right adjoint. That the filtration is separated
follows from the joint conservativity of the functors $L_p$ together with the fact that
$\nu$ is weakly increasing and surjective. Finally, in order for $\nu$ to be both weakly
increasing and surjective, we must have $\nu(n+1,p) \le \nu(n, p) +1$. This, together
with compatibility with suspension on each stratum, completes the proof.
\end{proof}

\begin{lemma} If $\nu$ jumps at $n$, then $\Sigma \mathcal{C}_{\ge n} = 
\mathcal{C}_{\ge n+1}$.
\end{lemma}

\begin{lemma} The functor $P^n$ is a 
left adjoint to the inclusion $\mathcal{C}_{\le n} \subseteq
\mathcal{C}$.
\end{lemma}
\begin{proof} This is classical from the theory of Bousfield localizations, but
we recall the proof here since it displays where we use the presentability hypotheses.
The subcategory
$\mathcal{C}_{\le n}$ is evidently closed under limits. From our presentability conditions, we
also 
see that $\mathcal{C}_{\le n} \subseteq \mathcal{C}$ is accessible. It follows that the inclusion
admits a left adjoint
$\tilde{P}^n$ \cite[5.5.2.9]{HTT}, and we need to identify it with $P^n$. To that end,
define $\tilde{P}_{n+1}$ as the fiber of the unit $\mathrm{id} \to \tilde{P}^n$. For any $X \in
\mathcal{C}$ we have a fiber sequence
	\[
	\tilde{P}_{n+1}X \to X \to \tilde{P}^nX.
	\]
So for any $W \ge n+1$, we get a fibration:
	\[
	\mathrm{map}(W, \tilde{P}_{n+1}X)
	\to
	\mathrm{map}(W, X) \to \mathrm{map}(W, \tilde{P}^nX).
	\]
The last term vanishes by definition of slice coconnectivity, so the first map is an equivalence.
It follows that $\tilde{P}_{n+1}X \ge n+1$ and that $\tilde{P}_{n+1}$ is a right adjoint
to the inclusion $\mathcal{C}_{\ge n+1} \subseteq \mathcal{C}$. The result follows.
\end{proof}

\begin{proposition} For each $n$, $\mathsf{Slice}_n$ is a reflective subcategory
of $\heartsuit_n$. In particular, $\mathsf{Slice}_n$ is an ordinary, presentable,
additive category. If $\nu$ jumps at $n$, then $\mathsf{Slice}_n = \heartsuit_n$ 
and is thus Grothendieck abelian.
\end{proposition}
\begin{proof} The inclusion $\Sigma \mathcal{C}_{\ge n} \subseteq
\mathcal{C}_{\ge n+1}$ yields an inclusion
	\[
	\mathsf{Slice}_n \subseteq \heartsuit_n,
	\]
and $P^n$ provides the desired left adjoint. The last claim 
is immediate from the definition of jump.
\end{proof}

\begin{example}[Perverse $t$-structures] If $X$ is a space equipped with
a finite stratification $X = \coprod_{s \in S} X_s$,
then the $\infty$-category
of $S$-constructible sheaves (valued in the derived category of abelian groups, say)
$\mathsf{D}_{S-cstr}(X)$ is an $S$-stratified $\infty$-category.
Here we consider $S$ as a poset by
declaring $s \le s'$ if the closure
of $X_{s'}$ contains $X_s$.
If $p: S \to \mathbb{Z}$
is a function (the \emph{perversity}), then we can define a perverse $t$-structure on
$\mathsf{D}_{S-cstr}(X)$ as the slice filtration associated to the dimension function
$\nu(n, s) = n + p(s)$.
This filtration is a $t$-structure because $\nu$ 
jumps at every $n \in \mathbb{Z}$.
\end{example}

The homotopy theory of Mackey functors with values in $\mathcal{C}$ has the
feature that all of its strata are presheaves valued in $\mathcal{C}$. Thus,
a filtration on $\mathcal{C}$ determines a filtration on all the strata in a canonical
way. This provides us with the most important class of examples for our work.

\begin{example} Let $\mathcal{O}$ be an inductive orbital category and 
$\mathcal{C}$ a presentable, stable
$\infty$-category equipped with a presentable, separated filtration $\tau$, and
$\nu$ a dimension function on $\mathsf{P}_{\mathcal{O}}$. Then the 
\textbf{$\nu$-slice filtration on $\mathsf{Mack}(\mathcal{O};\mathcal{C})$} is defined by
	\[
	X \ge n \iff \text{ for all }T \in \mathcal{O}, \,\,X^{\Phi T} \in \tau_{\ge \nu(n, T)}\mathcal{C}.
	\]
\end{example}

\begin{convention} 
	\begin{itemize}
	\item For the remainder of this section,
	unless otherwise stated, $\mathcal{C}$ will denote a stable,
	presentable, $\infty$-category equipped with a presentable,
	separated filtration.
	\item For the remainder of the paper,
	the homotopy theory $\mathsf{Sp}$ will be equipped
	with its standard $t$-structure filtration unless otherwise specified. 
	With this convention,
	there is an unambiguous $\nu$-slice filtration on $\mathsf{Sp}^{\mathcal{O}}$ for
	any inductive orbital category $\mathcal{O}$ and dimension function $\nu$.
	\end{itemize}
\end{convention}

\begin{definition} If $\mathcal{O}$ is
an inductive orbital category, then an \textbf{$\mathcal{O}$-parameterized filtration}
on $\mathsf{Mack}(\mathcal{O}; \mathcal{C})$ is the data of a filtration, for every $T \in
\fin_{\mathcal{O}}$, on $\mathsf{Mack}(\mathcal{O}_{/T}; \mathcal{C})$ with the
following properties:
	\begin{itemize}
	\item induction and restriction preserve filtration,
	\item if $T = \coprod_i T_i$, then the filtration on $\mathsf{Mack}(\mathcal{O}_{/T}; \mathcal{C})$
	is identified with the product filtration under the canonical equivalence
		\[
		\mathsf{Mack}(\mathcal{O}_{/T}; \mathcal{C}) \stackrel{\cong}{\longrightarrow}
		\prod_i \mathsf{Mack}(\mathcal{O}_{T_i}; \mathcal{C}).
		\]
	\end{itemize}
\end{definition}
\begin{remark} An $\mathcal{O}$-parameterized filtration determines and
is determined by a family of filtrations, one on $\mathsf{Mack}(\mathcal{O}_{/T}; \mathcal{C})$ for
each $T \in \mathcal{O}$, for which induction and restriction between orbits preserve filtration.
\end{remark}

The following is immediate from the definitions, and Lemma \ref{lem-geo-ind-res}.

\begin{proposition} Let $\mathcal{O}$ be an inductive orbital category, and suppose we have
a dimension function $\nu$ on $\mathsf{P}_{\mathcal{O}}$. For
any $T \in \fin_{\mathcal{O}}$, denote by $\nu_T$ the restriction of $\nu$
to $\mathsf{P}_{\mathcal{O}_{/T}}$. Then:
	\begin{enumerate}[(i)]
	\item The family of $\nu_T$-slice filtrations
	is an $\mathcal{O}$-parameterized filtration on 
	$\mathsf{Mack}(\mathcal{O}; \mathcal{C})$.
	\item Restriction and induction preserve slice coconnectivity and
	take $n$-slices to $n$-slices.
	\item Restriction and induction preserve coconnectivty
	for the $t$-structures associated to the slice filtration,
	and take elements of $\heartsuit_n$ to elements of $\heartsuit_n$.
	\end{enumerate}
\end{proposition}

Now we prove a recognition theorem that allows one to identify a given
parameterized filtration with the $\nu$-slice filtration.

\begin{theorem}\label{thm:slice-recognition}
 Suppose $\{F_{\ge n}\}$ is a
parameterized filtration of $\mathsf{Mack}(\mathcal{O}; \mathcal{C})$.
Then it agrees with the $\nu$-slice filtration on
$\mathsf{Mack}(\mathcal{O}; \mathcal{C})$ if and only if the following two conditions
are satisfied:
	\begin{enumerate}[(i)]
	\item for all $n$, $F_{\ge n}
	\subseteq {^{\nu}}\mathsf{Mack}(\mathcal{O}; \mathcal{C})_{\ge n}$,
	\item for every $T \in \mathcal{O}$, the class of objects
	$\{ X^{\Phi T} \in \mathcal{C}_{\ge \nu(n,  T)} | X \in 
	F_{\ge n}\}$ generates
	$\mathcal{C}_{\ge \nu(n, T)}$ under colimits, extensions,
	and equivalences.
	\end{enumerate}
\end{theorem}
\begin{proof} Using the functors $(i_{T})_*$, it is straightforward to see that 
(ii) is satisfied for the $\nu$-slice filtration, and (i) is tautological. So we prove
the other direction.

We proceed by induction on the size of $\mathsf{P}_{\mathcal{O}}$.
So let $T \in \mathsf{P}_{\mathcal{O}}$ be a minimal element, 
with associated functors $(i^*, i_*)$, and let $(j_!, j^*)$ be the functors associated
to the upward closed complement of $T$.

Since, by assumption (i), $F_{\ge n} \subseteq {^{\nu}}\mathsf{Mack}(\mathcal{O}; \mathcal{C})_{\ge n}$,
we need to show that any $X \ge n$ belongs to $F_{\ge n}$. To that end, consider the
cofiber sequence
	\[
	j_!j^*X \to X \to i_*i^*X.
	\]
It suffices to show that both $j_!j^*X$ and $i_*i^*X$ belong to $F_{\ge n}$. That
$j_!j^*X$ lies in $F_{\ge n}$
follows from the fact that $\{F_{\ge n}\}$ is a parameterized filtration, together
with the induction hypothesis and Proposition \ref{prop-nil-induced}. 

To complete the proof, it suffices, by the definition
of the slice filtration for Mackey functors,
to show that if $Y \in \mathsf{Mack}(\mathcal{O}_{\{T\}}; \mathcal{C})_{\ge n}$, then
$i_*Y \in F_{\ge n}$.

To that end, recall that, since $\mathcal{O}_{\{T\}}$
consists of a single object and every endomorphism
in an inductive orbital category
is an isomorphism, $\mathsf{Mack}(\mathcal{O}_{\{T\}}; \mathcal{C})
\cong \mathsf{Fun}(\mathrm{Aut}(T), \mathcal{C})$. Under this
equivalence, 
	\[
	\mathsf{Mack}(\mathcal{O}_{\{T\}}; \mathcal{C})_{\ge n}
	\cong \mathsf{Fun}(\mathrm{Aut}(T), \mathcal{C}_{\ge \nu(n,T)})
	\]
Let $e: * \to \mathrm{Aut}(T)$ denote the inclusion of the identity.
Then the
subcategory $\mathsf{Fun}(\mathrm{Aut}(T), \mathcal{C}_{\ge \nu(n,T)})$
is generated under extensions, equivalence, and colimits
by the essential image of the left Kan extension
	\[
	e_!: \mathcal{C}_{\ge \nu(n, T)} = 
	\mathsf{Fun}(*, \mathcal{C}_{\ge \nu(n, T)})
	\to \mathsf{Fun}(\mathrm{Aut}(T), \mathcal{C}_{\ge \nu(n, T)}),
	\]
also known as $\mathrm{Aut}(T)_+ \wedge (-)$.
By our assumption
(ii), this subcategory is also generated by elements
of the form $e_!X^{\Phi T}$ for $X \in F_{\ge n}$.

Let $A$ denote the class of $Y \in \mathsf{Mack}(\mathcal{O}_{\{T\}}; \mathcal{C})$
such that $i_*Y \in F_{\ge n}$. 
Then $A$ is closed under extensions, equivalence, and colimits. So it suffices
by our assumption (ii)
to show that $A$ contains $e_!Z^{\Phi T}$ for every $Z \in F_{\ge n}$.
Write $Z$ as an extension
	\[
	j_!j^*Z \to Z \to i_*i^*Z.
	\]
By the induction hypothesis and the same argument as before, $j_!j^*Z \in F_{\ge n}$,
and we have assumed $Z \in F_{\ge n}$. Therefore $i^*i_*Z \in F_{\ge n}$, and hence
$Z \in A$, which completes the proof.
\end{proof}

\begin{corollary}[Hill-Yarnall]\label{cor:HY-char}
	\begin{itemize}
	\item The original slice filtration on $\mathsf{Sp}^G$
	agrees with the filtration associated to the dimension function
		\[
		\nu_{sl}(n, H) = \left\lfloor \frac{n}{|H|} \right\rfloor.
		\]
	\item The regular slice filtration on $\mathsf{Sp}^G$
	agrees with the filtration associated to the dimension function
		\[
		\nu_{reg}(n, H) = \left\lceil \frac{n}{|H|} \right\rceil.
		\]
	\end{itemize}
\end{corollary}
\begin{proof} We give the proof for the original slice filtration, the proof in the regular
case is much the same. Let $F_{\ge n}\mathsf{Sp}^G$ denote the subcategory
of spectra which are slice $(n-1)$-positive in the sense of \cite[4.8]{HHR}.
It is elementary to check that this filtration is compatible with restrictions and
induction \cite[4.13]{HHR}, so this defines a parameterized filtration on
$\mathsf{Sp}^G$. To verify condition (i) it suffices to show that
	\[
	(G_+ \wedge_K S^{m\rho_K - \epsilon})^{\Phi H} \ge 
	\left\lfloor \frac{n}{|H|} \right\rfloor
	\]
whenever $\epsilon = 0,1$ and $m|K| - \epsilon \ge n$, and this
computation is routine using the double-coset formula.

To verify (ii), let $m = \lfloor n/|H|\rfloor$ and
notice that $G_+ \wedge_H S^{(m+1)\rho_H - 1}$ is in $F_{\ge n}\mathsf{Sp}^G$
and has
$H$-geometric fixed points a finite wedge of copies of $S^{m}$. The result
follows since $S^m$ generates $\mathsf{Sp}_{\ge m}$ under colimits, extensions,
and equivalences.
\end{proof}

We close with an analog of \cite[2.10]{HY} which, in some cases,
allows us to use the values of the Mackey functor instead of its
geometric fixed points.

\begin{lemma}\label{lem:slice-vs-post}
Let $\nu$ be a dimension function and fix $n \in \mathbb{Z}$. 
Suppose that $\nu(n, -): \mathsf{P}_{\mathcal{O}} \to \mathbb{Z}$ is order-preserving.
Then
	\[
	X\ge n \iff \textup{ for all }T \in \mathcal{O}, \,\, X(T) \in \mathcal{C}_{\ge \nu(n, T)}.
	\]
\end{lemma}
\begin{proof} Let $T \in \mathcal{O}$ be minimal and consider the usual inductive set-up:
	\[
	j_!j^*X \to X \to i_*i^*X.
	\]
Ordinary connectivity behaves well under induction and restriction, so the inductive
hypothesis tells us that either assumption on $X$ leads to a $j^*X$ which is
slice $n$-connective and satisfies $X(S) = j^*X(S) \in \mathcal{C}_{\ge \nu(n, S)}$ for all
$S \ne T$. In particular, each $j^*X(S)$ is in $\mathcal{C}_{\ge \nu(n, T)}$ by
our assumption on $\nu(n, -)$. It follows from (\ref{prop-nil-induced}) that
$j_!j^*X$ is both slice $n$-connective and satisfies $j_!j^*X(T)\in \mathcal{C}_{\ge \nu(n, T)}$.
But now both conditions are closed under extensions and cofibers, and
$X^{\Phi T} = i_*i^*X(T)$, so the proof is complete.
\end{proof}

\section{Categories of slices}\label{sec:slice-cats}

We saw in the last section that every slice filtration gives rise
to a sequence of Grothendieck abelian categories $\heartsuit_n$
and distinguished reflective subcategories $\mathsf{Slice}_n \subseteq \heartsuit_n$.
Every Grothendieck abelian category
$\mathcal{A}$ is a left exact
localization of a category of right modules over a ring (Gabriel-Popescu \cite{GP}). 
The ring in question is the ring of endomorphisms of a chosen generator.
Often it is convenient to use a family of generators instead of a single one, and
there is a mild generalization of the theorem in this case.

\begin{theorem}[Gabriel-Popescu \cite{GP}, Kuhn \cite{Kuhn}]
Let $\mathcal{A}$ be a Grothendieck abelian category and 
$\mathcal{A}_0 \subseteq \mathcal{A}$ an essentially small
subcategory closed under finite direct sums.
Suppose that, for every object $a \in \mathcal{A}$, there is a set
of objects $\{x_\alpha\}$, with each $x_{\alpha} \subseteq \mathcal{A}_0$,
and an epimorphism
$\bigoplus_{\alpha} x_{\alpha} \to a$. Then:
	\begin{enumerate}[(i)]
	\item The restricted Yoneda embedding to the category of
	finite product preserving functors
		\[
		G: \mathcal{A} \longrightarrow 
		\mathsf{Psh}^{\times}(\mathcal{A}_0, \mathsf{Set})
		\]
	is fully faithful.
	\item The functor $G$ admits an exact left adjoint, $F$. 
	\end{enumerate}
\end{theorem}

The following corollary was actually known before the above theorem,
and can be found in Freyd's thesis \cite{Freyd-thesis} and
a paper of Gabriel \cite{gabriel}, at least in the case of a single generator.

\begin{corollary}[Freyd-Gabriel] In the situation of the above theorem, if every
object of $\mathcal{A}_0$ is compact and projective in $\mathcal{A}$,
then $G$ is an equivalence of categories.
\end{corollary}

Thus, in order to understand $\heartsuit_n$ and its localization
$\mathsf{Slice}_n$, we should search for generators. In the case
of the standard $t$-structure on $\mathsf{Sp}$ we know that
$S^n$ is sufficient, and $\pi_n$ identifies $\mathsf{Sp}_{\ge n} \cap \mathsf{Sp}_{\le n}$
with the category of abelian groups in this way. In fact, we could have used
a finite wedge of copies of $S^n$ just as well.

From the definition
of the slice filtration, we are lead to consider spectra which are finite
wedges of spheres on each stratum. Since we hope to apply
the corollary above, we restrict attention to \emph{compact}
spectra with this property, and thus arrive at the notion of a slice sphere.

In \S\ref{ssec:slice-sphere}-\ref{ssec:slice-htpy} we define slice $n$-spheres
and study the elementary properties of the functors they corepresent
(called slice homotopy). Perhaps the only parts that require a bit of care
are Proposition \ref{prop-sph-is-test} (showing that slice $n$-spheres
exist in sufficient supply) and Proposition \ref{prop-slice-htpy-loc}
(which identifies the slice homotopy group of $P^nX$ in terms
of the slice homotopy group of $X$.) In \S\ref{ssec:slice-model}
we check the hypotheses of the result of Freyd-Gabriel
and conclude that $\heartsuit_n$ is the category of models
for an explicit Lawvere theory, and $\mathsf{Slice}_n$
is a specified localization thereof.

This bridge to algebra made, we reformulate
this Lawvere theory in terms of modules
over a Green functor
in \S\ref{ssec:slice-module}.
The technique is to use a parameterized
version of an argument going
back to Gabriel and Freyd, which
may be of independent interest.

Finally, we digress in \S\ref{ssec:split-modules}
to analyze the structure of modules over Green functors
with an additional condition that allows us to remove
redundant data in describing a module. We
show in \S\ref{ssec:slices-as-mack} that this condition
is satisfied in the case of interest, and this leads to
our final description of $\heartsuit_n$ and $\mathsf{Slice}_n$
as categories of \emph{twisted Mackey functors}.

\subsection{Slice spheres}\label{ssec:slice-sphere}

\begin{definition}\label{def-slice-sphere}
Let $\nu$ be a dimension function for $\mathsf{P}_{\mathcal{O}}$.
We say that an object $W \in \mathsf{Sp}^{\mathcal{O}}$ is a \textbf{slice sphere}
if
	\begin{enumerate}[(i)]
	\item $W$ is compact,
	\item for each $T \in \mathcal{O}$, the spectrum $W^{\Phi T}$
	is equivalent to a finite wedge of spheres.
	\end{enumerate}
We say that a slice sphere $W$ is \textbf{homogeneous} if the spheres
appearing in the decomposition of $W^{\Phi T}$ are all of the same dimension.
We say that an object $W \in \mathsf{Sp}^{\mathcal{O}}$ is a \textbf{slice $n$-sphere}
if it is a slice sphere satisfying the further requirement that the spheres in condition
(ii) have dimension $\nu(n, T)$. We denote the \emph{homotopy category}
of slice $n$-spheres by $\mathsf{Sph}_n$. 

Finally, we define an \textbf{isotropic slice $n$-sphere} to be a slice $n$-sphere
$W \in \mathsf{Sp}^{\mathcal{O}}$ such that, for each $T \in \mathcal{O}$,
$W^{\Phi T}$ is \emph{nonzero}.
\end{definition}

\begin{warning} The property of being a slice $n$-sphere depends on the dimension
function $\nu$. 
\end{warning}

\begin{warning} A slice sphere need not be a slice $n$-sphere for any $n$. 
\end{warning}

\begin{example} Any representation sphere is a slice sphere in $\mathsf{Sp}^G$.
\end{example}

\begin{example} For any group $G$ and
subgroup $H \subseteq G$, the cofiber of the map $\nabla: G/H_+ \to S^0$ is a slice sphere.
When $H$ is trivial, $\mathrm{cof}(\nabla)$
is an example of a slice 1-sphere for the original slice
filtration and a slice 2-sphere for the regular slice filtration. When $|G|>2$,
this cofiber is not equivalent to a wedge of representation spheres.
\end{example}

\begin{remark} The desire to stick to \emph{compact} objects is
motivated both by computations and to have a good algebraic theory. This
is also the reason why it is difficult to develop the theory for an arbitrary
stratified, stable $\infty$-category. The trouble is visible
already in the case of a recollement. The pushforward
functor
$i_*$ has no reason to preserve compact objects, and in fact
does not in our main example of interest. The pushforward $j_!$
does preserve compact objects, but the ones created this way necessarily
vanish on the closed locus. A related disappointment is that
one cannot test compactness on strata. 

In the case of
spectral Mackey functors, however, we have an alternative way to produce
examples: using the functors of the form $(\psi_{\widetilde{\mathcal{F}}})_!$. For
formal reasons, these functors \emph{do} preserve compact objects, and
can be used, in particular, to build a sufficient number of slice spheres.
\end{remark}

\begin{remark}\label{rmk:test-in-heart}
Since slice $n$-spheres are evidently slice $n$-connective,
the natural map
	\[
	[W, W'] \longrightarrow [\tau^{(n)}_{\le 0}W, \tau^{(n)}_{\le 0}W']
	\]
is an isomorphism.
Thus, we may identify $\mathsf{Sph}_n$ with
a full subcategory of $\heartsuit_n$.
\end{remark}

\begin{definition} A \textbf{testing subcategory for $n$-slices} is a full subcategory
$\mathsf{Test}_n \subseteq \mathsf{Sph}_n$ such that, for any $T \in \fin_{\mathcal{O}}$:
	\begin{itemize}
	\item $\mathrm{ind}_T\mathrm{res}_T(\mathsf{Test}_n) 
	\subseteq \mathsf{Test}_n$, and
	\item if $T \in \mathcal{O}$, then 
	there exists some $W \in \mathsf{Test}_n$ with $W^{\Phi T}$ equivalent
	to a \emph{nonzero}, finite wedge of copies of $S^{\nu(n, T)}$. 
	\end{itemize}
\end{definition}

One way of producing a testing subcategory is to generate one from an
isotropic slice $n$-sphere, which is immediate from the definitions.

\begin{lemma}\label{lem-isotropic-gen}
If $W$ is an isotropic $n$-sphere, let $\mathsf{Test}\langle W\rangle$
be the smallest subcategory of $\mathsf{Sph}_n$ containing $W$ and closed
under $\mathrm{ind}_T\mathrm{res}_T$, for all $T \in \fin_{\mathcal{O}}$.
Then $\mathsf{Test}\langle W\rangle$ is a testing subcategory.
\end{lemma}

At the moment, it is not at all clear that testing subcategories even exist.
Luckily, the definition is not vacuous.

\begin{proposition}\label{prop-sph-is-test}
 $\mathsf{Sph}_n$ is itself a testing subcategory for $n$-slices. More precisely,
 there exists an isotropic $n$-sphere in $\mathsf{Sph}_n$.
\end{proposition}
\begin{proof} We begin with a few reductions. First, by (\ref{lem-isotropic-gen}),
it suffices to construct a single isotropic $n$-sphere. Second, suppose
$\{T_1, ..., T_k\}$ is a set containing a representative for each minimal
element in $\mathsf{P}_{\mathcal{O}}$. If $W_i \in \mathsf{Sp}^{\mathcal{O}_{/T_i}}$ is
an isotropic $n$-sphere, for each $1\le i \le k$, then
$\bigvee_i \mathrm{ind}_{T_i}W_i$ is an isotropic $n$-sphere in
$\mathsf{Sp}^{\mathcal{O}}$. So we may assume, without loss of generality,
that $\mathcal{O}$ has a terminal object. We may also
assume, by induction, that the result holds for posets smaller
than $|\mathsf{P}_{\mathcal{O}}|$. Now we proceed.

Let $T \in \mathcal{O}$ be terminal. Recall
that we have a functor
	\[
	\left(\psi_{\{T\}}\right)_!: \mathsf{Mack}(\mathcal{O}_{\{T\}})
	\cong \mathsf{Sp} \to 
	\mathsf{Mack}(\mathcal{O}) = \mathsf{Sp}^{\mathcal{O}}
	\]
given by left Kan extension.
Define 
$X_0 := \left(\psi_{\{T\}}\right)_!S^{\nu(n, T)}$. 
Then $X_0$ is a slice sphere with the property that 
$(X_0)^{\Phi T'} =
S^{\nu(n, T)}$
for every $T' \in \mathcal{O}$.
Our goal is to modify $X_0$ until it becomes an isotropic slice $n$-sphere.

To that end, choose a conservative, surjective map of posets $\mathsf{P}_{\mathcal{O}}
\to [m]$ for some $m\ge 0$, and an ordering on each fiber. Then list the elements
of $\mathsf{P}_{\mathcal{O}}$ in dictionary order: $\mathsf{P}_{\mathcal{O}}
= \{p_0, p_1, ...., p_N\}$, where $N+1= |\mathsf{P}_{\mathcal{O}}|$.
We'll choose the ordering on the fiber over $0$ so that
$p_0 = T$. 

We propose to inductively build $X_k$ with the property that $X_k$
is a homogeneous slice sphere and, for $i\le k$, $(X_k)^{\Phi p_i}$
is a nonzero wedge of spheres of the form $S^{\nu(n, p_i)}$. So suppose we have
such an $X_k$ for $k< N$, then we describe how to build $X_{k+1}$. Since $X_k$
is a homogeneous
slice sphere, we know that $(X_k)^{\Phi p_{k+1}}$ is a (possibly trivial) finite
wedge
of spheres all of the same dimension, say $(X_k)^{\Phi p_{k+1}}
\cong \bigvee S^m$. To finish the proof, we need only treat each of the following three cases.
	\begin{itemize}
	\item Case 1: $(X_k)^{\Phi p_{k+1}} = 0$. In this case, let
	$A \in \mathsf{Sp}^{\mathcal{O}_{/p_{k+1}}}$ be
	an isotropic $n$ sphere for the restricted slice filtration, which
	exists by the induction, and
	take $X_{k+1}=X_k \vee \mathrm{ind}_{p_{k+1}}A$. 
	\item Case 2: $(X_k)^{\Phi p_{k+1}} \ne 0$ and $m\le \nu(n, p_{k+1})$.
	In this case, let $r = \nu(n, p_{k+1}) - m \ge 0$ and $X_k:= Y_{k, 0}$.
	Given $Y_{k, i}$ for $i<r$, define $Y_{k, i+1}$ as the cofiber of
	the counit:
		\[
		\mathrm{ind}_{p_{k+1}}\mathrm{res}_{p_{k+1}}Y_{k, i}
		\to Y_{k, i}.
		\]
	This construction does not affect geometric fixed points at $p_j$
	for $j\le k$, by the construction of our ordering,
	and it modifies all other geometric fixed points by 
	replacing the previous wedge of spheres with a new wedge of spheres
	of one higher dimension (or doing nothing if the geometric
	fixed points were trivial.) Indeed, on geometric
	fixed points, the unit of induction-restriction admits a section,
	and a summand of a wedge of spheres is still a wedge of spheres.
 	Thus, $X_{k+1}:=Y_{k, r}$ does the trick.
	\item Case 3: $(X_k)^{\Phi p_{k+1}} \ne 0$ and $m> \nu(n, p_{k+1})$.
	This is exactly as before, except we define $Y_{k,i+1}$ as the \emph{fiber}
	of the unit
		\[
		Y_{k, i} \to \mathrm{coind}_{p_{k+1}}\mathrm{res}_{p_{k+1}}Y_{k,i}.
		\]
	Again, we check that this is still a wedge of spheres of appropriate dimension
	using the fact that the geometric fixed points of the unit for the coinduction-restriction
	adjunction admits a retraction.
	\end{itemize}
\end{proof}

\begin{remark}\label{rmk:efficient-isotropic-spheres}
Isotropic slice spheres with fewer cells
may be constructed by modifying the above construction to be
more efficient. Specifically, one may replace the use of
$\mathrm{ind}_T\mathrm{res}_TX \to X$ with
any map $\mathrm{ind}_TY \to X$ which becomes split
upon restriction to $T$. We suspect that all isotropic slice
spheres arise from this more general procedure, but
have not tried to prove it.
\end{remark}

For later use, we record an evident but useful observation.

\begin{lemma}\label{lem-res-test} Fix $T \in \fin_{\mathcal{O}}$ and a testing subcategory
$\mathsf{Test}_n(\mathcal{O}) \subseteq \mathsf{Sph}_n(\mathcal{O})$. Then
$\mathrm{res}_T(\mathsf{Test}_n(\mathcal{O}))$ is a testing subcategory
for $n$-slices in $\mathsf{Sp}^{\mathcal{O}_{/T}}$.
\end{lemma}

\subsection{Slice homotopy groups}\label{ssec:slice-htpy}

Having defined the appropriate notion of spheres in our context, we now
develop the resulting analogue of homotopy groups and their basic properties.

\begin{definition}\label{def-slice-htpy} Fix a testing subcategory $\mathsf{Test}_n \subseteq \mathsf{Sph}_n$.
Then the restricted Yoneda embedding defines a functor:
	\[
	\hat{\pi}_n: \mathsf{Sp}^{\mathcal{O}} \longrightarrow
	\mathsf{Psh}_{\mathsf{Set}}^{\times}(\mathsf{Test}_n). 
	\]
to the category of product-preserving
presheaves on $\mathsf{Test}_n$. We
will call the target of $\hat{\pi}_n$ the category
of \textbf{$n$-models} and denote it by $\mathsf{Model}_n$.   
We call $\hat{\pi}_nX$ the \textbf{$n$-th slice homotopy group} of $X$. 
\end{definition}

\begin{remark} Since $\mathsf{Test}_n$ is additive,
there is a canonical equivalence
	\[
	\mathsf{Model}_n \cong \mathsf{Psh}_{\mathsf{Ab}}^{\oplus}(\mathsf{Test}_n)
	\]
with the category of additive presheaves. We will
often move back and forth between the equivalent
interpretations of $\mathsf{Model}_n$ for convenience.
For example, with this definition, it is clear that
$\mathsf{Model}_n$ is abelian.
\end{remark}

The name `slice homotopy group' is slightly abusive since
$\hat{\pi}_nX$ is really a diagram of groups. We don't think this will cause confusion.

\begin{warning} Both the category $\mathsf{Model}_n$
and the functor $\hat{\pi}_n$ depend
on the dimension function $\nu$ and the choice of testing subcategory.
\end{warning}

The following lemma is evident and will often be used without comment.

\begin{lemma}\label{lem:homological}
The functor $\hat{\pi}_n$ is homological, i.e. $\hat{\pi}_n$ sends
cofiber sequences to exact sequences.
\end{lemma}

\begin{warning} Despite the notation, it is very much \emph{not} the case that
$\hat{\pi}_n(\Sigma X)$ is the same as $\hat{\pi}_{n-1}X$ in general.
\end{warning}

\begin{proposition}\label{prop-slice-left-adjoint}
 For any testing subcategory $\mathsf{Test}_n \subseteq \mathsf{Sph}_n$,
the functor
	\[
	\hat{\pi}_n: \heartsuit_n \to \mathsf{Model}_n
	\]
admits a left adjoint, denoted $H: \mathsf{Model}_n
\to \heartsuit_n$. It is determined by the property that, if $W \in \mathsf{Test}_n$,
then 
$H[-,W] = \tau^{(n)}_{\le 0}W$. Composing with $P^n$ yields a functor
	\[
	P^nH : \mathsf{Model}_n \longrightarrow \mathsf{Slice}_n
	\]
left adjoint to the restriction $\hat{\pi}_n \vert_{\mathsf{Slice}_n}$.
\end{proposition}
\begin{proof} The inclusion 
$\heartsuit_n \subseteq \mathsf{Sp}^{\mathcal{O}}_{\ge n}$
preserves limits (since it
admits $\tau^{(n)}_{\le 0}$ as a left adjoint), and so does the Yoneda embedding and restriction, whence
$\hat{\pi}_n$ preserves limits. Since each object of $\mathsf{Test}_n$ is compact, 
by definition, the functor $\hat{\pi}_n$ also preserves filtered colimits. The result now
follows from the adjoint functor theorem, since the source and target are presentable
categories. The formula for $H$ follows from checking that $\tau^{(n)}_{\le 0}W$ 
corepresents the expected functor on $\heartsuit_n$.
\end{proof}

\begin{warning} The functor $\hat{\pi}_n$ is generally not right exact
as a functor with domain $\mathsf{Slice}_n$.
\end{warning}

We will find much use out of the following 
elementary vanishing conditions for slice homotopy groups.

\begin{proposition}\label{prop-slice-htpy-vanish}
\begin{enumerate}[(a)]
\item If $X \ge n$, then, for any slice $n$-sphere $W$,
$[W, \Sigma X] = 0$.
\item If $T \in \mathcal{O}$ is minimal
and an $n$-jump, and $X \in \mathsf{Sp}_{\ge \nu(n+1, T)}$,
then $[W,i_*X] = 0$ for any slice $n$-sphere $W$.
\end{enumerate}
\end{proposition}
\begin{proof} First we prove (a).
By induction on the size of $\mathsf{P}_{\mathcal{O}}$,
and standard arguments,
we're reduced to checking that $\hat{\pi}_ni_*i^*\Sigma X = 0$
where $(i^*, i_*)$ is the adjoint pair associated to a minimal element
$T \in \mathcal{O}$. But $\hat{\pi}_n(i_*i^*\Sigma X)(W) = [W^{\Phi T}, (\Sigma X)^{\Phi T}]$
by adjunction. Since $W$ is an $n$-slice sphere, $W^{\Phi T}$ is either 0 or a wedge
of copies of $S^{\nu(n, T)}$. On the other hand, $\Sigma X^{\Phi T}$ is 
$(\nu(n, T)+1)$-connective since $X \ge n$, so in either case we get zero.
The proof of (b) is similar and easier.
\end{proof}

\begin{proposition}\label{prop-inj-res} If $Y\le n$, let $T^{\mathrm{jump}} \in \fin_{\mathcal{O}}$ be the coproduct over all the $n$-jumps for $\nu$, in $\mathcal{O}$. 
Then, for any $T \in \fin_{\mathcal{O}}$ with $T^{\mathrm{jump}}$ as a summand,
the map
	\[
	\hat{\pi}_nA(W) \to \hat{\pi}_nA(\mathrm{ind}_{T}\mathrm{res}_{T}W)
	\]
induced by the counit $\mathrm{ind}_{T}\mathrm{res}_{T}W \to W$
is injective.
\end{proposition}
\begin{proof} The assumptions precisely
imply that the cofiber of the counit map is slice $(n+1)$-connective.
\end{proof}

\begin{definition} Let $\mathsf{Slice}_n^{alg} \subseteq 
\mathsf{Model}_n$ denote the full subcategory
of functors which satisfy the conclusion of (\ref{prop-inj-res}). We temporarily
call this
the \textbf{category of algebraic $n$-slices}.
\end{definition}
\begin{warning} The category $\mathsf{Slice}_n^{alg}$ depends on the
choice of testing subcategory. We justify this abuse of notation by
Theorem \ref{thm:slices-as-models}, below, which implies that changing the testing subcategory
yields an equivalent category of algebraic $n$-slices.
\end{warning}

From Proposition \ref{prop-inj-res} we get:

\begin{corollary} The functor $\hat{\pi}_n: \mathsf{Slice}_n
\to \mathsf{Model}_n$ factors through $\mathsf{Slice}_n^{alg}$.
\end{corollary}

Given an additive presheaf on $\mathsf{Test}_n$, it is easy to change it into
an algebraic $n$-slice.

\begin{lemma}\label{lem-alg-loc} The inclusion $\mathsf{Slice}_n^{alg} \subseteq
\mathsf{Model}_n$ admits a left adjoint
$L^{inj}$ described explicitly as
	\[
	L^{inj}\pi (W) = 
	\frac{\pi(W)}{\mathrm{ker}: 
	\pi(W) \to \pi(\mathrm{ind}_T\mathrm{res}_TW)}
	\]
where $T = T^{\textup{jump}}$ as in (\ref{prop-inj-res}).
\end{lemma}

Next, we give some preliminary evidence for the strong relationship between
$n$-slices and algebraic $n$-slices.

\begin{proposition}\label{prop-slice-htpy-loc}
For $C \in \mathsf{Sp}^{\mathcal{O}}$, the localization map
$C \to P^nC$ induces an isomorphism
	\[
	L^{inj}\hat{\pi}_nC \stackrel{\cong}{\longrightarrow} \hat{\pi}_nP^nC.
	\]
\end{proposition}
\begin{proof} By Proposition \ref{prop-inj-res}, $\hat{\pi}_nP^nC$ belongs
to $\mathsf{Slice}_n^{alg}$. By Lemma \ref{lem-alg-loc}, we get a
commutative diagram:
	\[
	\xymatrix{
	&\hat{\pi}_nC\ar[dr]^f\ar[dl]_g &\\
	L^{inj}\hat{\pi}_nC \ar[rr]_h&& \hat{\pi}_nP^nC
	}
	\]
Since $g$ is surjective, we can prove that $h$ is an isomorphism by showing
that $f$ is surjective and $\mathrm{ker}(f) \subseteq \mathrm{ker}(g)$
(the other inclusion is implied by commutativity of the diagram.)

The obstruction
to the surjectivity of $f$ lives in $\hat{\pi}_n\Sigma P_{n+1}C$, but this
group vanishes by Proposition \ref{prop-slice-htpy-vanish}(a), so $f$ is surjective.

Now suppose $W\in \mathsf{Test}_n$ and $W \to C \to P^nC$ is null. Then
we have a factorization $W \to P_{n+1}C \to C \to P^nC$. Let $T = T^{\textup{jump}}$.
Then it suffices to show that the composite
	\[
	\mathrm{ind}_T\mathrm{res}_TW \to W \to P_{n+1}C
	\]
is null. Equivalently, that $\mathrm{res}_TW \to \mathrm{res}_T P_{n+1}C$ is null.
But, by
the definition of $T^{\textup{jump}}$ and the slice filtration, together with the
fact that restrictions preserve slice connective covers, we conclude that
$\mathrm{res}_TP_{n+1}C \in \Sigma \mathsf{Sp}^{\mathcal{O}_{/T}}_{\ge n}$.
Since $\mathrm{res}_TW$ is a slice $n$-sphere, we conclude the vanishing
by Proposition \ref{prop-slice-htpy-vanish}(a), which completes the proof.
\end{proof}

\subsection{Slices as models for a Lawvere theory}\label{ssec:slice-model}

We are now ready for our first algebraic description of slices.

\begin{theorem}\label{thm:slices-as-models} The functors
$\hat{\pi}_n$ and $H$ yield an equivalence of adjoint pairs:
 
	\[
	\xymatrix{
	\heartsuit_n \ar@/^1.5pc/[rr]^{\hat{\pi}_n}_{\cong}\ar[dd]^{P^n} &&
	\mathsf{Model}_n \ar[dd]_{L^{inj}}\ar[ll]^{H}\\
	&&\\
	\mathsf{Slice}_n \ar@/_1.5pc/[rr]_{\hat{\pi}_n}^{\cong} 
	\ar@/^1.5pc/[uu]&& 
	\mathsf{Slice}_n^{alg} \ar@/_1.5pc/[uu] \ar[ll]_{P^nH}
	}
	\]
\end{theorem}

To prove this, we will need the following classical
bit of category theory.

\begin{proposition}[Freyd, Gabriel]\label{prop:freyd-gabriel}
Let $\mathcal{A}$ be a Grothendieck abelian category
and suppose $\mathcal{A}_0 \subseteq \mathcal{A}$ is
a full subcategory closed under finite direct sums
and satisfying the following properties:
	\begin{enumerate}[(i)]
	\item Every object of $\mathcal{A}_0$ is compact.
	\item Every object of $\mathcal{A}_0$ is projective.
	\item The restricted Yoneda embedding
		\[
		G: \mathcal{A}
		\longrightarrow
		\mathsf{Psh}^{\times}_{\mathsf{Set}}(\mathcal{A}_0)
		\]
	is conservative.
	\end{enumerate}
Then $G$
is an equivalence of categories.
\end{proposition}
\begin{proof} By (i) and (ii), $G$ preserves
all colimits.
Since the target of $G$ is generated by representables,
$G$ is essentially surjective. On the other hand, we have
a functor $L: \mathsf{Psh}^{\times}_{\mathsf{Set}}(\mathcal{A}_0)
\longrightarrow \mathcal{A}$ induced from the inclusion
$\mathcal{A}_0 \subseteq \mathcal{A}$ by left Kan extension.
There is a natural map $LG \to \mathrm{id}$ and we will be done
if we can check it is an isomorphism. But $G$ is conservative
by assumption, so we need only check that the map
	\[
	GLG \to G
	\]
is an isomorphism. Now, $GL$ is a colimit preserving
endofunctor of $\mathsf{Psh}^{\times}_{\mathsf{Set}}(\mathcal{A}_0)$
which is the identity on $\mathcal{A}_0$, so it is canonically equivalent to the
identity. This completes the proof.
\end{proof}

\begin{remark} Suppose $\mathcal{A}$ is any category with equalizers,
coproducts, and with the property that monic-epimorphisms are
isomorphisms. If $\mathcal{A}_0 \subseteq \mathcal{A}$
is a full subcategory, then the following are equivalent:
	\begin{enumerate}[(i)]
	\item Every object $x \in \mathcal{A}$ admits an epimorphism
	$\coprod_{\beta} a_{\beta} \to x$ where $a_{\beta} \in \mathcal{A}_0$.
	\item The restricted Yoneda embedding
	$\mathcal{A} \longrightarrow \mathsf{Psh}_{\mathsf{Set}}(\mathcal{A}_0)$
	is faithful.
	\item The restricted Yoneda embedding
	$\mathcal{A} \longrightarrow \mathsf{Psh}_{\mathsf{Set}}(\mathcal{A}_0)$
	is conservative.
	\end{enumerate}
In the literature, people use the phrase `$\mathcal{A}_0$ is a family of generators
for $\mathcal{A}$' inconsistently to mean any of these, even in situations
where they are \emph{not} all equivalent. (Actually, the first two conditions
are always equivalent. It is the equivalence with the third condition which is transient.)
\end{remark}

\begin{proof}[Proof of Theorem \ref{thm:slices-as-models}] Recall that
we have already shown (Proposition \ref{prop-slice-htpy-loc}) that 
$\hat{\pi}_nP^nC = L^{inj}\hat{\pi}_nC$ for any $C$. Thus,
if the top two arrows in the diagram are inverse equivalences,
then the two localizations are necessarily equivalent. 
Now, by Remark \ref{rmk:test-in-heart}, we may view
$\mathsf{Test}_n$ as a full subcategory of $\heartsuit_n$, and
each slice $n$-sphere is compact. We now argue that they
are also projective. Indeed,
suppose that $f: A \to B$ is a map between elements of
$\heartsuit_n$. Then form the cofiber sequence in $\mathsf{Sp}^{\mathcal{O}}$:
	\[
	A \to B \to C.
	\]
The cokernel of $f$ in $\heartsuit_n$ is given by $\tau^{(n)}_{\le 0}C$.
Suppose $f$ is surjective so
that $\tau^{(n)}_{\le 0}C = 0$. Then we 
have $C \in \tau^{(n)}_{\ge 1}\mathsf{Sp}^{\mathcal{O}}$
so that $C$ is the suspension of something slice $n$-connective.
It follows from Proposition \ref{prop-slice-htpy-vanish} that
$\hat{\pi}_nC = 0$ and by Lemma \ref{lem:homological}
that $\hat{\pi}_nA \to \hat{\pi}_nB$ is surjective. So we've shown
that for all $W \in \mathsf{Test}_n$, that $[W, -]$ preserves
surjections, and hence each $W$ is projective.

By the previous proposition, it now suffices to check that
$\hat{\pi}_n$ is conservative on $\heartsuit_n$. This 
we prove below.
\end{proof}

\begin{proposition} The functor $\hat{\pi}_n: \heartsuit_n \to 
\mathsf{Model}_n$ is conservative.
\end{proposition}

\begin{warning} The statement is obviously false on the larger domain
$\mathsf{Sp}^{\mathcal{O}}$.
\end{warning}

\begin{proof} We proceed by induction on the order of $\mathsf{P}_{\mathcal{O}}$.
So let $T \in \mathcal{O}$ be a minimal element with
upward closed complement $\mathcal{F}$, and let $(i^*, i_*)$ and $(j_!, j^*)$
be the usual adjoint pairs associated to this situation. 

Assume that $f: A \to B$ is a map of objects in $\heartsuit_n$
which induces an isomorphism
on $\hat{\pi}_n$ and form the diagram:
	\[
	\xymatrix{
	j_!j^*A \ar[r]\ar[d] & A\ar[r] \ar[d]& i_*i^*A\ar[d]\\
	j_!j^*B\ar[r]& B\ar[r] & i_*i^*B
	}
	\]
It suffices to prove that the left and right vertical maps are equivalences.
	\begin{itemize}
	\item Our induction hypothesis applies to $\mathcal{O}_{/T'}$
	and the restricted testing subcategory (\ref{lem-res-test}) for any
	$T'\in \mathcal{F}$ since, for such $T'$,
	$\mathsf{P}_{\mathcal{O}_{/T'}}$ is strictly smaller than $\mathsf{P}_{\mathcal{O}}$
	(\ref{rem-induct}). This, together with Proposition \ref{prop-nil-induced}, implies
	that $j_!j^*(f)$ is an equivalence.
	\item By the definition of testing subcategory,
	there is some $W \in \mathsf{Test}_n$ with
	$W^{\Phi T}$ a wedge of copies of $S^{\nu(n, T)}$. We have a map
	of fiber sequences:
		\[
		\xymatrix{
		\mathrm{map}(W, A) \ar[r]\ar[d] &
		\mathrm{map}(W, i_*i^*A) \ar[r]\ar[d] &
		\mathrm{map}(W, \Sigma j_!j^*A) \ar[d]\\
		\mathrm{map}(W, B) \ar[r] &
		\mathrm{map}(W, i_*i^*B) \ar[r] &
		\mathrm{map}(W, \Sigma j_!j^*B)
		}
		\]
	The last vertical 
	map is an equivalence because we have already shown
	$j_!j^*$ is an equivalence. We know that
	$[\Sigma^kW, A] = [\Sigma^kW, B] = 0$ for $k>0$
	because $A,B \in \tau^{(n)}_{\le 0}\mathsf{Sp}^{\mathcal{O}}$,
	and $\Sigma^kW \in \tau^{(n)}_{\ge k}\mathsf{Sp}^{\mathcal{O}}$
	as $W$ is slice $n$-connective.
	We also know
	$[W, A] = \hat{\pi}_nA(W) \to
	\hat{\pi}_nB(W) = [W, B]$ is an isomorphism by assumption. 
	Thus, the first vertical map is an equivalence. We deduce
	that the middle vertical map is an equivalence.

	\item Finally, let $e: * \to \mathrm{Aut}(T)$ be the inclusion
	of the identity. Since $S^{\nu(n, T)}$ is a retract of $W^{\Phi T}$,
	we conclude that $e_!S^{\nu(n, T)}$ is a retract of $i^*W$. (Here
	it is important that we are working stably so that
	the norm $e_! \to e_*$
	is an equivalence.) Equivalences are closed under
	retracts, so we conclude from the previous bullet
	and adjunction that
		\[
		\mathrm{map}(S^{\nu(n,T)}, A^{\Phi T})
		\to \mathrm{map}(S^{\nu(n, T)}, B^{\Phi T})
		\]
	is a weak equivalence. But $A^{\Phi T}$ and $B^{\Phi T}$ are
	$\nu(n, T)$-connective by the definition of the slice filtration
	and our assumption that $A,B \ge n$, so we
	deduce that $i^*f$ is an equivalence and the proof is complete.
	\end{itemize}
\end{proof}

\begin{example} Let $\mathsf{A}(G)$ denote the
full subcategory of $\mathsf{Sp}^{G}$ containing
$S^0$ and closed under smashing with finite $G$-sets. 
It is a consequence of the tom Dieck splitting
that $\mathrm{h}\mathsf{A}(G)$ is equivalent to the classical
Burnside category \cite[V.9.6]{LMS}. So if one finds
a representation sphere $S^V$ which is also
an isotropic slice $n$-sphere, then we get an equivalence
of categories:
	\[
	\xymatrix{
	\mathrm{h}\mathsf{A}(G) \ar[rr]^-{S^V \wedge (-)}
	&&
	\mathsf{Test}\langle S^V\rangle},
	\]
and hence an equivalence:
	\[
	\mathsf{Mack}(G; \mathsf{Ab}) \stackrel{\cong}{\longrightarrow}
	\mathsf{Model}_n \cong \heartsuit_n.
	\]
Unwinding the definitions
we learn that in this case the $n$-slice of a spectrum is determined
by a quotient of its $V$th homotopy Mackey functor.

This works more generally for any isotropic slice $n$-sphere in the
Picard group, and more generally still for any inductive orbital
category in place of $\mathcal{O}_G$.
\end{example}

\subsection{Slices as modules over a Green functor}
\label{ssec:slice-module}

In the case when $\mathsf{Test}_n$ is generated by an
isotropic slice $n$-sphere $W$, the functor $\hat{\pi}_nX$
records the following data:
	\begin{itemize}
	\item For each $T \in \mathsf{Fin}_{\mathcal{O}}$,
	the homotopy groups $[\downarrow_T W, \downarrow_TX]
	= [\uparrow_T\downarrow_TW, X]$.
	\item For each map $\uparrow_T\downarrow_T W
	\to \uparrow_{T'}\downarrow_{T'}W$, the induced map
		\[
		[\uparrow_{T'}\downarrow_{T'}W, X]
		\to
		[\uparrow_T\downarrow_TW, X].
		\]
	\end{itemize}

In this section we identify the above
data with the $\underline{\mathrm{End}}(W)$-module
structure on the homotopy Mackey functor $\underline{[W, X]}$. 
The notations and notions in this theorem will be defined
in the body of the section, but we state it now in any case:

\begin{theorem}\label{thm:slices-as-modules}
The functor
$\underline{[W, -]}$ yields an equivalence of adjoint pairs:
 
	\[
	\xymatrix{
	\heartsuit_n \ar[rr]^-{\underline{[W, -]}}_-{\cong}\ar[dd]^{P^n} &&
	\mathsf{RMod}_{\underline{\mathrm{End}}(W)} \ar[dd]_{L^{inj}}\\
	&&\\
	\mathsf{Slice}_n \ar[rr]_-{\underline{[W, -]}}^-{\cong} 
	\ar@/^1.5pc/[uu]&& 
	\mathsf{RMod}^{\mathrm{loc}}_{\underline{\mathrm{End}}(W)} \ar@/_1.5pc/[uu]
	}
	\]
\end{theorem}

This
identification is straightforward, but is conceptually pleasing
and serves as a stepping stone to our simplification at the end of \S2. 
The reader is encouraged to use this section as a quick reminder of the definitions
of Green functors and modules over them, and then proceed to 
\S\ref{ssec:split-modules}.

We begin by reviewing the symmetric monoidal structure
on $\mathsf{Mack}(\mathcal{O}; \mathsf{Ab})$. Classically, one
begins with a symmetric monoidal structure on $\mathrm{h}\aeff(G)$
(or $\mathrm{h}\mathsf{A}(G)$) which arises from the product of finite
$G$-sets. Unfortunately, the categories $\mathsf{Fin}_{\mathcal{O}}$
need not admit products in general. For example, the categories
$\mathsf{Fin}_{\mathcal{O}_{\mathcal{F}}}$ associated to
a family of subgroups of a group $G$ usually do not have terminal objects.

\begin{remark} The author actually does not know of an example
of an inductive orbital category 
where $\mathsf{Fin}_{\mathcal{O}}$ does not admit \emph{nonempty}
finite products. If no such example exists, the discussion
of the symmetric monoidal structure on $\mathsf{Mack}(\mathcal{O}, \mathsf{Ab})$
below could be simplified somewhat.
\end{remark}

Nevertheless, the presheaf that a product would represent can always
be defined, and this puts a \emph{pro}monoidal structure
on $\mathrm{h}\aeff(\mathcal{O})$ which we now describe.

\begin{definition} Given $U,V \in \mathrm{h}\aeff(\mathcal{O})$, define
a presheaf of sets
$(U \times V): \mathrm{h}\aeff(\mathcal{O})^{op} 
\longrightarrow \mathsf{Set}$
by
	\[
	(U \times V)(T):= \{\text{triples }(S \rightarrow T,
	S \rightarrow U,
	S \rightarrow V)\}/\sim
	\]
Here the maps are in $\mathsf{Fin}_{\mathcal{O}}$
and two triples are equivalent if there is an isomorphism
$S \stackrel{\cong}{\to} S'$ commuting with all the specified maps.
Functoriality comes from pullback and composition. This
construction produces
a functor
	\[
	\times:
	\mathrm{h}\aeff(\mathcal{O}) \times
	\mathrm{h}\aeff(\mathcal{O})
	\longrightarrow
	\mathsf{Psh}_{\mathsf{Set}}(\mathrm{h}\aeff(\mathcal{O}))
	\]
(i.e. a profunctor $\mathrm{h}\aeff(\mathcal{O})^{\times 2} \rightarrow 
\mathrm{h}\aeff(\mathcal{O})$.)
\end{definition}

\begin{definition} The \textbf{box product} of abelian group
valued Mackey functors
$\underline{M}$ and $\underline{N}$ on $\mathcal{O}$ is
defined by left Kan extension and restriction via the diagram:
	\[
	\xymatrix{
	\mathrm{h}\aeff(\mathcal{O}) \times \mathrm{h}\aeff(\mathcal{O})\ar[d]_-{\times}
	\ar[rr]^-{(\underline{M}, \underline{N})} &&
	\mathsf{Ab} \times \mathsf{Ab} \ar[rr]^-{\otimes} && \mathsf{Ab}\\
	\mathsf{Psh}_{\mathsf{Set}}(\mathrm{h}\aeff(\mathcal{O}))
	\ar@/_1pc/[urrrr]_-{\times_{!}} &&&&\\
	\mathrm{h}\aeff(\mathcal{O})\ar[u]
	\ar@/_2pc/[uurrrr]_-{\underline{M} \Box \underline{N}} &&&&
	}
	\]
\end{definition}

This gives $\mathsf{Mack}(\mathcal{O}; \mathsf{Ab})$ the structure
of a symmetric monoidal category. For a reference in much
greater generality than we need, see \cite{mack2}.

\begin{definition}\label{defn:burnside-mack} The \textbf{Burnside Mackey functor},
$\underline{A}$, is the unit for the symmetric monoidal
structure on $\mathsf{Mack}(\mathcal{O}; \mathsf{Ab})$
defined above. Explicitly, it is given by
	\[
	T \mapsto \text{Grothendieck group of 
	the maximal subgroupoid of }\left(\mathsf{Fin}_{\mathcal{O}}\right)_{/T}
	\text{ with respect to }\amalg
	\]
Functoriality is given by pullback and composition.
\end{definition}

\begin{definition}\label{defn:green}
A \textbf{Green functor} is an associative
algebra in $\mathsf{Mack}(\mathcal{O}; \mathsf{Ab})$
with respect to the box product. Given a Green functor
$\underline{R}$ we let $\mathsf{RMod}_{\underline{R}}$
denote the \textbf{category of right modules} over the 
associative algebra $\underline{R}$. 
\end{definition}

\begin{example} For any $X, Y \in \mathsf{Sp}^{\mathcal{O}}$ the
assignment 
	\[
	\mathsf{Fin}_{\mathcal{O}} \ni T \mapsto
	[\downarrow_TX, \downarrow_TY]
	\]
extends to an abelian group valued Mackey functor on $\mathcal{O}$. 
In the case $X = Y$, composition endows this Mackey functor
with the structure of a Green functor, the \textbf{endomorphism
Green functor}, denoted $\underline{\mathrm{End}}(X)$. 
For any $Y$, the Mackey functor $\underline{[X, Y]}$
is a right $\underline{\mathrm{End}}(X)$-module in
a canonical way.
\end{example}

Next we'll need the condition that corresponds to
the localization $P^n$.

\begin{definition}\label{defn:slice-local} Let $W$ be an isotropic slice $n$-sphere
and let $T^{\mathrm{jump}} \in \mathsf{Fin}_{\mathcal{O}}$ denote the coproduct of
all (isomorphism classes) of $n$-jumps. We say that
a right $\underline{\mathrm{End}}(W)$-module $\underline{M}$ is 
\textbf{slice local} if,
for all $T \in \mathsf{Fin}_{\mathcal{O}}$, the map
	\[
	\underline{M}(T) \longrightarrow \underline{M}(T^{\mathrm{jump}} \times T)
	\]
induced by the projection $T^{\mathrm{jump}} \times T \to T$ is
injective.
\end{definition}

With these definitions in hand, we can prove the theorem.

\begin{proof}[Proof of Theorem \ref{thm:slices-as-modules}]
Given $T \in \mathsf{Fin}_{\mathcal{O}}$ denote by
$\underline{\mathsf{End}}(W)_T$ the Mackey functor
	\[
	\mathsf{Fin}_{\mathcal{O}} \ni U \mapsto 
	[\uparrow_U\downarrow_{U}W, \uparrow_T\downarrow_{T}W]. 
	\]
Note that evaluation at $T$ gives an isomorphism:
	\[
	\mathrm{Hom}_{\underline{\mathrm{End}}(W)}
	(\underline{\mathsf{End}}(W)_T, \underline{M})
	\stackrel{\cong}{\longrightarrow} \underline{M}(T). 
	\quad (*)
	\]
So each module $\underline{\mathsf{End}}(W)_T$ is compact and projective,
and together they detect isomorphisms. Let $\mathcal{A}_0$ denote
the full subcategory of $\mathsf{RMod}_{\underline{\mathsf{End}}(W)}$
spanned by the objects $\underline{\mathsf{End}}(W)_T$. Then $\mathcal{A}_0$
satisfies the hypotheses of Proposition \ref{prop:freyd-gabriel}. Thus, the
restricted Yoneda embedding
	\[
	\mathsf{RMod}_{\underline{\mathsf{End}}(W)}
	\longrightarrow \mathsf{Psh}_{\mathsf{Set}}^{\times}(\mathcal{A}_0)
	\]
is an equivalence of categories. On the other hand, we have a functor
	\[
	\mathsf{Test}\langle W\rangle \longrightarrow 
	\mathcal{A}_0
	\]
given by $\uparrow_T\downarrow_TW \mapsto \underline{\mathsf{End}}(W)_T$.
This is an equivalence in view of the formula $(*)$ above together with
the induction-restriction adjunction.

So we have a commutative diagram
	\[
	\xymatrix{
	&\heartsuit_n\ar[dl]_{\underline{[W, -]}}\ar[dr]^-{\hat{\pi}_n}&\\
	\mathsf{RMod}_{\underline{\mathsf{End}}(W)} \ar[r]_{\cong}& 
	\mathsf{Psh}^{\times}_{\mathsf{Set}}(\mathcal{A}_0) \ar[r]_{\cong} &
	\mathsf{Model}_n
	}
	\]
By Theorem \ref{thm:slices-as-models} we know that
$\hat{\pi}_n$ is an equivalence, hence so is $\underline{[W,-]}$. 
The identification of $\mathsf{Slice}_n$ as 
$\mathsf{RMod}^{\mathrm{loc}}_{\underline{\mathsf{End}}(W)}$
is immediate from the definitions and the corresponding
statement in Theorem \ref{thm:slices-as-models}.
\end{proof}

\begin{remark} In this extended remark we place the elementary
result in this section into the broader setting of parameterized
category theory. We will assume the reader is familiar
with the notions of parameterized category theory
found in \cite{BDGNS-basics, denis-stable}. The reader
unfamiliar or uninterested in these ideas can safely skip this remark.

First notice that, given a Green functor $\underline{R}$,
we can define an $\mathcal{O}$-category $\underline{\mathsf{RMod}}_{\underline{R}}$
by the assignment
	\[
	\mathcal{O} \ni T \mapsto
	\mathsf{RMod}_{\mathrm{res}_T\underline{R}}.
	\]
This $\mathcal{O}$-category is $\mathcal{O}$-semiadditive
in the sense of \cite[5.3]{denis-stable}. A priori, homomorphisms
in an $\mathcal{O}$-category form an object in
$\underline{\mathsf{Set}}_{\mathcal{O}}$, but $\mathcal{O}$-semiadditivity
canonically promotes these to elements in
$\underline{\mathsf{Mack}}(\mathcal{O}; \mathsf{CMon})$, i.e.
Mackey functors valued in commutative monoids. In fact, each of
the commutative monoids in question are group-like, so we'll
call such an $\mathcal{O}$-category `$\mathcal{O}$-additive.'
Finally, each of the fibers is abelian, and in this case we'll
say that the $\mathcal{O}$-category is $\mathcal{O}$-abelian.

Now, in general, given a cocartesian
section $C: \mathcal{O}^{op} \to \mathcal{A}$
of an $\mathcal{O}$-abelian category $\mathcal{A}$, we get
an $\mathcal{O}$-functor
	\[
	\underline{\mathrm{Hom}}_{\mathcal{A}}(C, -):
	\mathcal{A}
	\longrightarrow \underline{\mathsf{Mack}}(\mathcal{O}; \mathsf{Ab})
	\]
Using this functor we have natural candidates for the notions
of (i) $\mathcal{O}$-compact, (ii) $\mathcal{O}$-projective, and
(iii) being an $\mathcal{O}$-generator. The key point is to use
$\mathcal{O}$-indexed colimits in place of ordinary colimits
for each definition. Now the proof of Proposition \ref{prop:freyd-gabriel}
carries over essentially verbatim to prove:
	\begin{itemize} 
	\item Suppose $\mathcal{A}$
	is an $\mathcal{O}$-presentable, $\mathcal{O}$-abelian
	$\mathcal{O}$-category with an
	$\mathcal{O}$-compact, $\mathcal{O}$-projective,
	$\mathcal{O}$-generator $C$. Then the $\mathcal{O}$-functor
		\[
		\underline{\mathrm{Hom}}(C, -): \mathcal{A}
		\longrightarrow
		\underline{\mathsf{RMod}}_{\underline{\mathrm{End}}(C)}
		\]
	is an equivalence of $\mathcal{O}$-categories. 
	\end{itemize}
The evident generalization to a family of $\mathcal{O}$-generators
also applies, as does the analogue of the Gabriel-Kuhn-Popescu theorem.
All of the proofs are straightforward adaptations of the classical ones,
once you've pinned down the proper definitions (as above).
\end{remark}

\subsection{Digression: Modules over geometrically split Green functors}
\label{ssec:split-modules}

The Green functor $\underline{\mathrm{End}}(W)$ for an isotropic
slice $n$-sphere $W$ has several special features which
simplifies its category of modules. We single out one of these
in this section and study the resulting bit of algebra.

\begin{definition} Given an abelian group
valued Mackey functor $\underline{M}$
and an orbit $T \in \mathcal{O}$. For maps
of orbits $U \to T$ denote the associated transfer
by
	\[
	\mathrm{tr}_U^T: \underline{M}(U) \to \underline{M}(T).
	\]
Then we define an abelian group $\underline{M}^{\Phi T}$ by:
	\[
	\underline{M}^{\Phi T}:= 
	\frac{\underline{M}(T)}{\langle \mathrm{im}(\mathrm{tr}_U^T)\rangle_{U\to T}}.
	\]
\end{definition}

\begin{remark} The action of $\mathrm{Aut}(T)$ on $\underline{M}(T)$
descends to an action on $\underline{M}^{\Phi T}$. 
\end{remark}

\begin{remark} If $\underline{R}$ is a Green functor,
	then $\underline{R}^{\Phi T}$ is a ring because the submodule
	we quotient by is an ideal. Similarly, if $\underline{M}$ is a 
	right $\underline{R}$-module,
	then $\underline{M}^{\Phi T}$ is naturally
	a right $\underline{R}^{\Phi T}$-module. 
\end{remark}

\begin{example} We will prove below (Lemma \ref{lem:endo-geom-fix})
that if $W$ is an isotropic slice $n$-sphere, then
$\underline{\mathrm{End}}(W)^{\Phi T} = \mathrm{End}(W^{\Phi T})$.
\end{example}

If one unwinds the definition of a module over a Green functor,
one finds the formula:
	\[
	m \cdot \mathrm{tr}_U^T(r) = \mathrm{tr}_U^T(\mathrm{res}^T_U(m) \cdot r).
	\]
Thus, the action of transferred elements is somewhat redundant.
Of course, it is not usually possible to systematically
express an element $r \in \underline{R}(T)$ as the sum of transferred
elements and an element not in the image of any transfer. We
will show below (Lemma \ref{lem:endo-geom-split}) that this \emph{does}
happen our case of interest. This leads us to the following definition.

\begin{definition}\label{defn:geom-split}
A Green functor $\underline{R}$ is called \textbf{geometrically splittable}
if, for all $T \in \mathcal{O}$, the ring map
	\[
	\underline{R}(T) \longrightarrow \underline{R}^{\Phi T}
	\]
admits an $\mathrm{Aut}(T)$-equivariant section. A \textbf{geometrically
split} Green functor is a geometrically splittable Green functor
$\underline{R}$ equipped with chosen splittings 
$s_T: \underline{R}^{\Phi T} \longrightarrow \underline{R}(T)$ as above. 
\end{definition}

\begin{example} The Burnside Mackey functor (\ref{defn:burnside-mack})
$\underline{A}$ is geometrically split. Indeed,
$\underline{A}^{\Phi T} = \mathbb{Z}$ with trivial
$\mathrm{Aut}(T)$-action, so there is a \emph{unique}
splitting of the augmentation
$\underline{A}(T) \to \mathbb{Z}$. 
\end{example}

If $\underline{M}$ is a module over a geometrically split Green functor,
then restriction along $s_T$ gives each $\underline{M}(T)$ the structure
of a module over $\underline{R}^{\Phi T}$. Enumerating
the remaining structure on $\underline{M}$ visible to the rings
$\underline{R}^{\Phi T}$, one is lead to define the category
in Definition \ref{defn:r-phi-modules} below. We will
need a little notation before making this definition, however.

\begin{definition} If $M$ is an abelian group
and $S$ is a set, we denote by $M \otimes S$
and $M^S$ the abelian groups $M \otimes \mathbb{Z}\{S\}$
and $\mathrm{Hom}(\mathbb{Z}\{S\}, M)$, respectively. 
There is a canonical map
$M \otimes S \to M^S$ given by sending
$m \otimes s$ to the function with value $m$
at $s$ and zero otherwise; thus we may
view elements of $M \otimes S$ as elements in
$M^S$.

If a finite group $H$ acts on $M$ and $S$, define
	\[
	\mathrm{trace}: M \otimes S \longrightarrow
	M^S
	\]
by $m\otimes s \mapsto \sum_{h \in H} hm \otimes hs$.
This induces a map by the same name:
	\[
	\mathrm{trace}:
	(M \otimes S)_H \longrightarrow
	\left(M^S\right)^H.
	\]
\end{definition}

\begin{definition}\label{defn:r-phi-modules} Let $\underline{R}$ be a geometrically
split Green functor. We define the category of
\textbf{(right) $\underline{R}^{\Phi}$-modules},
$\mathsf{RMod}_{\underline{R}^{\Phi}}$, to consist
of objects $\underline{M}$ which consist of the following
data:
	\begin{enumerate}[(i)]
	\item For each $[T] \in \mathsf{P}_{\mathcal{O}}$
	an $\underline{R}^{\Phi T}$-$\mathrm{Aut}(T)$-module
	$\underline{M}(T)$;
	\item (Restrictions and transfers) For each pair $[T]\ge [T']$, maps of 
	$\underline{R}^{\Phi T'}$-$\mathrm{Aut}(T')$-modules:
		\[
		\left(\underline{M}(T) 
		\otimes \mathrm{Hom}(T, T')\right)_{\mathrm{Aut}(T)}
		\longrightarrow \underline{M}(T')
		\]
		\[
		\underline{M}(T') \longrightarrow
		\left(\underline{M}(T)^{\mathrm{Hom}(T,T')}\right)^{\mathrm{Aut}(T)}
		\]
	of $\underline{R}^{\Phi T'}$-$\mathrm{Aut}(T')$-modules.
	\end{enumerate}
	subject to the following conditions:
	\begin{itemize}
	\item (Double-coset formula) For each pair, $[T] \ge [T']$, the diagram
		\[
		\xymatrix{
		\left(\underline{M}(T) \otimes 
		\mathrm{Hom}(T, T')\right)_{\mathrm{Aut}(T)}
		\ar[rr]^{\mathrm{trace}}\ar[dr] && 
		\left(\underline{M}(T')^{\mathrm{Hom}(T,T')}\right)^{\mathrm{Aut}(T)}
		\\
		&\underline{M}(T')\ar[ur]&
		}
		\]
	commutes;
	\item (Composition of restrictions and transfers)
	 For each triple $[T_0] \ge [T_1]\ge [T_2]$, the diagrams
		\[
		\xymatrix{
		\left[\left(\underline{M}(T_0) 
		\otimes \mathrm{Hom}(T_0, T_1)\right)_{\mathrm{Aut}(T_0)}
		\otimes \mathrm{Hom}(T_1, T_2)\right]_{\mathrm{Aut}(T_1)}
		\ar[rr]\ar[dd]_{\cong} &&
		\left(\underline{M}(T_1) 
		\otimes \mathrm{Hom}(T_1, T_2)\right)_{\mathrm{Aut}(T_1)}
		\ar[dd]\\
		&&\\
		\left(\underline{M}(T_0) 
		\otimes \mathrm{Hom}(T_0, T_2)\right)_{\mathrm{Aut}(T_0)}
		 \ar[rr]
		&&
		\underline{M}(T_2)
		}
		\]
		\[
		\xymatrix{
		\underline{M}(T_2)\ar[rr]\ar[dd]&&
		\left(\underline{M}(T_0)^{\mathrm{Hom}(T_0, T_2)}\right)^{\mathrm{Aut}(T_0)}
		\ar[dd]\\
		&&\\
		\left(\underline{M}(T_1)^{\mathrm{Hom}(T_1, T_2)}\right)^{\mathrm{Aut}(T_1)}\ar[rr]
		&&
		\left[
		\left(\left(\underline{M}(T_0)^{\mathrm{Hom}(T_0,T_1)}\right)^{\mathrm{Aut}(T_0)}
		\right)^{\mathrm{Hom}(T_1, T_2)}\right]^{\mathrm{Aut}(T_1)}
		}
		\]
	commute.
	\end{itemize}
\end{definition}

\begin{remark} Each object $\underline{M}$ is,
in particular, a Mackey functor. The definition above just
keeps track of the interaction with the $\underline{R}^{\Phi T}$-module
structures.
\end{remark}

By design, there is a forgetful functor:
	\[
	\mathsf{RMod}_{\underline{R}} \longrightarrow 
	\mathsf{RMod}_{\underline{R}^{\Phi}}.
	\]

\begin{theorem}\label{thm:split-structure} The forgetful functor
	\[
	\mathsf{RMod}_{\underline{R}} \longrightarrow 
	\mathsf{RMod}_{\underline{R}^{\Phi}}
	\]
is an equivalence.
\end{theorem}

As usual, the proof relies on induction over the poset $\mathsf{P}_{\mathcal{O}}$.
To set up this induction, we'll need to extend
the functors $((j_{\mathcal{F}})_!, j^*_{\mathcal{F}}, (j_{\mathcal{F}})_*)$ and 
$(i^*_{\widetilde{\mathcal{F}}}, 
(i_{\widetilde{\mathcal{F}}})_*, (i_{\widetilde{\mathcal{F}}})^!)$
to the setting of modules over Green functors. 

We begin by recalling a few definitions.

\begin{definition}\cite{BenRoub, HTT} Suppose we
are given a diagram of categories
	\[
	\xymatrix{
	\mathcal{C}'\ar[d]_{\psi^{*}}\ar[r]^{\widetilde{G}} & \mathcal{C}\ar[d]^{\phi^*}\\
	\mathcal{D}'\ar[r]_G & \mathcal{D}
	}
	\]
commuting up to a specified natural isomorphism
$\eta: \phi^* \circ \widetilde{G} \stackrel{\cong}{\to} G \circ \psi^*$. 
We say that the diagram is \textbf{left adjointable}
or satisfies the \textbf{left Beck-Chevalley condition}
if $G$ and $\widetilde{G}$ admit left adjoints $F$ and $\widetilde{F}$
and the exchange transformation
	\[
	F \circ \phi^* \to F \circ \phi^* \circ \widetilde{G} \circ \widetilde{F}
	\stackrel{\eta}{\cong} F \circ G \circ \psi^* \circ \widetilde{F}
	\to \psi^* \circ \widetilde{F}
	\]
is an isomorphism. We say the diagram is \textbf{right adjointable}
if the functors $G$ and $\widetilde{G}$ admit right adjoints
$H$ and $\widetilde{H}$ and the exchange transformation
	\[
	\psi^* \circ \widetilde{H}
	\to H \circ G \circ \psi^* \circ \widetilde{H}
	\stackrel{\eta^{-1}}{\cong}
	H \circ \phi^* \circ \widetilde{G} \circ \widetilde{H}
	\to H \circ \phi^*
	\]
is an isomorphism.
\end{definition}

\begin{remark} A square is right adjointable if and only
if the square becomes left adjointable upon
applying the functor $(-)^{op}: \mathsf{Cat} \to \mathsf{Cat}$.
\end{remark}

The following lemma is elementary and left to the reader.

\begin{lemma}\label{lem:adjointable} Suppose we have a square of functors:
	\[
	\xymatrix{
	\mathcal{C}'\ar[d]_{\psi^{*}}\ar[r]^{\widetilde{G}} & \mathcal{C}\ar[d]^{\phi^*}\\
	\mathcal{D}'\ar[r]_G & \mathcal{D}
	}
	\]
commuting up to a specified natural isomorphism. Suppose further that
$G$ and $\widetilde{G}$ admit left adjoints $F$ and $\widetilde{F}$,
respectively \emph{and} that $\psi^*$ and $\phi^*$ admit
right adjoints $\psi_*$ and $\phi_*$, respectively. Then
the above square is left adjointable if and only if the square
	\[
	\xymatrix{
	\mathcal{C}' \ar[r]^{\psi^*}\ar[d]_{\widetilde{G}} & 
	\mathcal{D}'\ar[d]^{G}\\
	\mathcal{C} \ar[r]_{\phi^*} & \mathcal{D}
	}
	\] 
is right adjointable.
\end{lemma}

\begin{construction}\label{cstr:recoll-of-modules}
Let $\mathcal{F} \subseteq \mathsf{P}_{\mathcal{O}}$
be upward closed with downward closed complement $\widetilde{\mathcal{F}}$
and suppose $\underline{R}$ is a Green functor. Denote by
$\underline{R}_{\mathcal{F}}$ the restriction of the
Green functor to $\mathcal{O}_{\mathcal{F}}$
and by $\Phi^{\mathcal{F}}\underline{R}$ the Green functor
on $\mathcal{O}_{\widetilde{\mathcal{F}}}$ defined by
$\left(i_{\widetilde{\mathcal{F}}}\right)^*\underline{R}$
(as in Notation \ref{notation-so-many-adjoints}).

Restriction defines a functor:
	\[
	j^*: \mathsf{RMod}_{\underline{R}}
	\longrightarrow \mathsf{RMod}_{\underline{R}_{\mathcal{F}}}.
	\]
Extension by zero defines a functor:
	\[
	i_*:
	\mathsf{RMod}_{\Phi^{\mathcal{F}}\underline{R}}
	\longrightarrow
	\mathsf{RMod}_{\underline{R}}.
	\]
Repeated use of the adjoint functor theorem produces a string of adjoints:
	\[
	\xymatrix{
	\mathsf{RMod}_{\underline{R}_{\mathcal{F}}}
	\ar@<3ex>[r]^{j_!}\ar@<-3ex>[r]^{j_*}&
	\mathsf{RMod}_{\underline{R}} \ar[l]_{j^*}
	\ar@<3ex>[r]^{i^*}\ar@<-3ex>[r]^{i^!}&
	\mathsf{RMod}_{\Phi^{\mathcal{F}}\underline{R}}\ar[l]_{i_*}
	}
	\]
\end{construction}

Any $\underline{R}$-module $\underline{M}$ has
an underlying Mackey functor, and a priori it is not
clear how the functors above compare to the similarly named
functors applied to the underlying Mackey functor. Luckily,
the two notions agree.

\begin{proposition} The formation of the functors $(j_!, j^*, j_*)$ and
$(i^*, i_*, i^!)$ on modules commutes with passage
to underlying Mackey functors.  
\end{proposition}
\begin{proof} The forgetful functor
	\[
	\phi: \mathsf{RMod}_{\underline{R}}
	\longrightarrow \mathsf{Mack}(\mathcal{O}; \mathsf{Ab})
	\]
admits a left adjoint, given by $(-) \Box \underline{R}$, and
a right adjoint, given by $\underline{\mathrm{Hom}}(\underline{R}, -)$.
The formation of $j^*$ commutes with both of these, and the
formation of $i^*$ and $i_*$ commutes with the first of these.
The result now follows by repeated application of Lemma \ref{lem:adjointable}.
\end{proof}

\begin{warning} The category $\mathsf{RMod}_{\underline{R}}$ is \emph{not}
a recollement of $\mathsf{RMod}_{\underline{R}_{\mathcal{F}}}$
and $\mathsf{RMod}_{\underline{R}^{\Phi \widetilde{F}}}$ in
the sense of Definition \ref{defn:recoll}, in general, because
$i^*$ fails to be left exact in general. Instead, it is an example
of a \emph{recollement of abelian categories} as defined,
for example, in \cite{FP}, or in Definition \ref{defn:abel-recoll}
below.
\end{warning}

We would like to make an inductive argument by showing
that, in the geometrically split case, the theorem
is true for $\mathcal{O}_{\mathcal{F}}$ and for 
$\mathcal{O}_{\widetilde{\mathcal{F}}}$ and then conclude the
result for $\mathcal{O}$. Unfortunately, recollements of
abelian categories are less well-behaved than their
$\infty$-categorical cousins, and it is not true in general
that a map of recollements is an equivalence if it is so
on each piece of the recollement (see \cite[2.2]{FP}).
However, this equivalence criterion is true under
further hypotheses on the recollement.

\begin{definition}\cite{FP}\label{defn:abel-recoll} Suppose 
we have a collection of additive functors between
abelian categories
	\[
	\xymatrix{
	\mathcal{A}_1 \ar@<3ex>[r]^{j_!}\ar@<-3ex>[r]^{j_*}& \mathcal{A} \ar[l]_{j^*}
	\ar@<3ex>[r]^{i^*}\ar@<-3ex>[r]^{i^!}& \mathcal{A}_0\ar[l]_{i_*}
	}
	\]
where each functor is left adjoint to the functor below it.
We say this presents a \textbf{recollement of abelian categories}
if the following additional conditions are satisfied:
	\begin{enumerate}[(i)]
	\item The functors $j_!, j_*,$ and $i_*$ are fully faithful,
	\item The functor $i_*$ has essential image precisely
	those objects $a \in \mathcal{A}$ such that
	$j^*a = 0$.
	\end{enumerate}
If moreover each category has enough projectives, we say that
the recollement of abelian categories
is \textbf{pre-hereditary} if, for any projective
$V \in \mathcal{A}_0$, we have $(L_2i^*)(i_*V) = 0$.
\end{definition}

We will also require one additional functor which
only exists when $\underline{R}$ is geometrically split.

\begin{construction} Suppose $\underline{R}$
is geometrically split, and
$\widetilde{\mathcal{F}} \subseteq \mathsf{P}_{\mathcal{O}}$
is a set of minimal elements.
We define a functor
	\[
	r: \mathsf{RMod}_{\underline{R}}
	\longrightarrow \mathsf{RMod}_{\Phi^{\mathcal{F}}\underline{R}}
	\]
as follows. For an $\underline{R}$-module $\underline{M}$,
the underlying Mackey functor of $r\underline{M}$
is the restriction $\psi_{\widetilde{\mathcal{F}}}^*\underline{M}$
of $\underline{M}$ to $\mathcal{O}_{\widetilde{\mathcal{F}}}$. 
Since $\mathsf{P}_{\mathcal{O}_{\widetilde{\mathcal{F}}}}$
is a collection of pairwise incomparable elements, the Green
functor $\Phi^{\mathcal{F}}\underline{R}$ amounts
to the data of the rings $\underline{R}^{\Phi T}$
with $\mathrm{Aut}(T)$ action as $[T]$ ranges over
the elements of $\widetilde{\mathcal{F}}$. A module
is just a module over each of these rings
with compatible $\mathrm{Aut}(T)$-action. In our
case, we use the $\underline{R}^{\Phi T}$-$\mathrm{Aut}(T)$-module
structure on $(r\underline{M})(T) = \underline{M}(T)$
defined by restricting the $\underline{R}(T)$-module
structure along $s_T$. 
\end{construction}

\begin{example} In the case $\underline{R} = \underline{A}$
is the Burnside Mackey functor,
$r = \psi^*_{\widetilde{\mathcal{F}}}$ is just given by
restriction.
\end{example}

\begin{lemma} The functor $r$ defined above is
exact and gives a retract of $i_*$. Moreover, $r$
admits a left adjoint, $r_!$, which
lifts the functor $(\psi_{\widetilde{\mathcal{F}}})_!$
on underlying Mackey functors.
\end{lemma}
\begin{proof} Exactness can be checked pointwise,
where it is clear. The $\underline{R}(T)$-module
structure on $i_*\underline{M}(T)$ automatically
factors through the quotient to an $\underline{R}^{\Phi T}$-module
structure, since the source of the relevant transfer maps
on $\underline{M}$ is zero by the definition if $i_*$.
The claim that $r$ is a retraction now follows
from the fact that the composite
	\[
	\underline{R}^{\Phi T} \stackrel{s_T}{\longrightarrow}
	\underline{R}(T)
	\longrightarrow
	\underline{R}^{\Phi T}
	\]
is the identity, by assumption.
\end{proof}

\begin{proposition}\label{prop:pre-hereditary}
Let $\underline{R}$ be a
geometrically split Green functor,
$\widetilde{\mathcal{F}} \subseteq \mathsf{P}_{\mathcal{O}}$
a collection of minimal elements with upward closed complement
$\mathcal{F}$. Then the string
of adjoints from Construction \ref{cstr:recoll-of-modules}
presents a pre-hereditary recollement of abelian categories.
\end{proposition}
\begin{proof} The fact that we have a recollement
of abelian categories is straightforward, so we prove that
the recollement is pre-hereditary. We claim that
the following sequence of functors
$\mathsf{RMod}_{\Phi^{\mathcal{F}}\underline{R}}
\longrightarrow \mathsf{RMod}_{\underline{R}}$
is exact:
	\[
	0 \to j_!j^*r_! \to r_! \to i_* \to 0. \quad (*)
	\]
(Here $r_! \to i_*$ is adjoint to $\mathrm{id} \cong ri_*$.)
Suppose for a moment we have established
this exactness. Then we can mimic
the argument in \cite[8.5]{FP} to
establish the pre-hereditary condition.
To elaborate, we first apply $i^*$
to obtain an exact sequence
	\[
	(L_2i^*)r_! \to (L_2i^*)i_* 
	\to (L_1i^*)j_!j^*r_!
	\]
and then evaluate on a projective $V \in
\mathsf{RMod}_{\Phi^{\mathcal{F}}\underline{R}}$.
The first term vanishes because $r_!$ preserves
projectives, being the left adjoint of
an exact functor. The last term vanishes
because the composite $(L_1i^*)j_!$
always vanishes in a recollement.
(Indeed, this follows formally
from the fact that $j_!$ preserves projectives
together with the identity $i^*j_! = 0$.)
Thus the middle term vanishes, which was to be
shown.

So we are left
with checking
the exactness of $(*)$. Note that for any
recollement of abelian categories, we
have an exact sequence
	\[
	j_!j^* \to \mathrm{id} \to i_*i^* \to 0.
	\]
Since $r_!$ is right exact,
the sequence
	\[
	j_!j^*r_! \to r_! \to i_*i^*r_! \to 0
	\]
is exact. But $i^*r_! \cong \mathrm{id}$ since
it is adjoint to $ri_* \cong \mathrm{id}$, so 
we learn that $(*)$ is exact except possibly at the first
nontrivial term. In other words, it suffices to show that
	\[
	j_!j^*r_! \to r_!
	\]
is injective.

This is detected on underlying Mackey functors,
so we need only check the injectivity of
	\[
	 j_!j^*(\psi_{\widetilde{\mathcal{F}}})_!
	 \to (\psi_{\widetilde{\mathcal{F}}})_!
	\]
We do this by evaluating on each $T \in \mathcal{O}$.
If $T \notin \widetilde{\mathcal{F}}$, then
$i^*$ will evaluate to zero, and 
it is always the case that $j_!j^* \to \mathrm{id}$
is an equivalence for such $T$.

So we are left with showing that, for any
$\underline{M} \in \mathsf{Mack}(\mathcal{O}_{\widetilde{\mathcal{F}}}; \mathsf{Ab})$,
and any $T \in \widetilde{\mathcal{F}}$, the map
	\[
	(j_!j^*(\psi_{\mathcal{F}})_!\underline{M})(T)
	\to ((\psi_{\widetilde{\mathcal{F}}})_!\underline{M})(T)
	\]
is injective. In fact, we will show that it admits a
natural retract. To construct this retract, we will need
to unpack the left Kan extensions taking place on each side.
We set-up some temporary notation to handle this.
	\begin{itemize}
	\item We will denote morphisms in effective
	Burnside categories by $\rightsquigarrow$
	to remind the reader that they are represented
	by spans.
	\item Let $K$ denote the category whose
	objects are strings
	$V \rightsquigarrow U \rightsquigarrow T$
	where 
	$V \in \mathsf{Fin}_{\mathcal{O}_{\widetilde{\mathcal{F}}}}$,
	$U \in \mathsf{Fin}_{\mathcal{O}_{\mathcal{F}}}$,
	and the morphisms take place in $\mathrm{h}\aeff(\mathcal{O})$.
	The arrows are commutative diagrams:
		\[
		\xymatrix{
		V \ar@{~>}[rr]^f\ar@{~>}[d]&&V'\ar@{~>}[d]\\
		U \ar@{~>}[dr]\ar@{~>}[rr]^g&& U'\ar@{~>}[dl]\\
		&T&
		}
		\]
	where $f$ is a morphism in 
	$\mathrm{h}\aeff(\mathcal{O}_{\widetilde{\mathcal{F}}})$
	and $g$ is a morphism in $\mathrm{h}\aeff(\mathcal{O}_{\mathcal{F}})$.
	\item Let $K'$ denote the category whose objects
	are arrows $V \rightsquigarrow T$ in $\mathrm{h}\aeff(\mathcal{O})$
	where $V \in \mathsf{Fin}_{\mathcal{O}_{\widetilde{\mathcal{F}}}}$.
	Morphisms are commutative diagrams:
		\[
		\xymatrix{
		V \ar@{~>}[rr]^f \ar@{~>}[dr]&& V'\ar@{~>}[dl]\\
		&T&
		}
		\]
	where $f$ belongs to $\mathrm{h}\aeff(\mathcal{O}_{\widetilde{\mathcal{F}}})$. 
	\end{itemize}
Composition provides a functor $K \to K'$. From the definitions
of $j_!, j^*,$ and $(\psi_{\widetilde{\mathcal{F}}})_!$ we get:
	\[
	\xymatrix{
	\colim_{K} \underline{M}(V) \ar[rr]\ar[d]^{\cong}&&
	\colim_{K'}\underline{M}(V)\ar[d]^{\cong}\\
	(j_!j^*(\psi_{\widetilde{\mathcal{F}}})_!\underline{M})(T)\ar[rr]
	&&
	((\psi_{\widetilde{\mathcal{F}}})_!\underline{M})(T)
	}
	\]
We now decompose $K'$ into two pieces.
	\begin{itemize}
	\item Let $K'_0$ denote the
	full subcategory of
	$K'$ spanned by objects of the form
	$V \leftarrow S \rightarrow T$
	where $S \in \mathsf{Fin}_{\mathcal{O}_{\widetilde{\mathcal{F}}}}$.
	Note that, since $\widetilde{\mathcal{F}}$ consists of
	minimal elements, this forces $[S]=[V]=[T]$.
	\item Let $K'_1$ denote the
	full subcategory spanned by objects of the form
	$V \leftarrow S \rightarrow T$ where
	$S \in \mathsf{Fin}_{\mathcal{O}_{\mathcal{F}}}$. 
	\end{itemize}
We have a functor $K'_1 \to K$ given by
	\[
	(V \leftarrow S \rightarrow T) \mapsto
	(V \leftarrow S = S =S \rightarrow T)
	\]
and hence a natural factorization:
	\[
	\xymatrix{
	&&\colim_{K'_1}\underline{M}(V)\ar[dll]\ar[d]\\
	\colim_{K} \underline{M}(V) \ar[rr]&&
	\colim_{K'}\underline{M}(V)
	}
	\]
We are thus reduced to
proving the following two claims:
	\begin{enumerate}[(i)]
	\item The category $K'$ decomposes
	as a disjoint union
	$K' = K_0' \coprod K_1'$ so that
	the right vertical map above has a natural splitting.
	\item The functor $K_1' \to K$ is final,
	so that the diagonal arrow above is
	an isomorphism.
	\end{enumerate}

The statement (i) follows from the fact that
the morphisms in $K'$ between
 $(V \rightsquigarrow T)$ and $(T \rightsquigarrow T)$
 (where the latter lies in $K'_0$)
 involve a morphism \emph{in $\mathsf{Fin}_{\mathcal{O}_{\widetilde{\mathcal{F}}}}$}
 $V \rightsquigarrow T$ or $T \rightsquigarrow V$. In order to make the resulting
 diagram commute, we conclude that $V$ must be of the form
 $V \leftarrow S \rightarrow T$ where 
 $S \in \mathsf{Fin}_{\mathcal{O}_{\widetilde{\mathcal{F}}}}$,
 and hence $(V \rightsquigarrow T) \notin K'_1$. 
 
The statement (ii) follows from the fact that $K_1' \to K$ is right
adjoint to the map $K \to K_1'$. 
\end{proof}

We now have a good understanding of the left
hand side of Theorem \ref{thm:split-structure},
so we turn to the right hand side. The category
$\mathsf{RMod}_{\underline{R}^{\Phi}}$
is built by iterated application of a procedure
due to Macpherson-Vilonen \cite{MV}.

\begin{definition}[\cite{MV}]
Let $\mathcal{U}$ and $\mathcal{Z}$
be abelian categories, and let
$\xi: F \to G$ be a natural transformation
between two additive functors
$F, G: \mathcal{U} \to \mathcal{Z}$. Define
a category $\mathcal{A}(\xi)$ as follows:
	\begin{itemize}
	\item Objects
	consist of pairs $U \in \mathcal{U}$
	and $Z \in \mathcal{Z}$
	equipped with a factorization:
		\[
		\xymatrix{
		FU \ar[rr]^{\xi} \ar[dr]&& GU\\
		&Z\ar[ur]&
		}
		\]
	\item Morphisms are maps of
	pairs $(U, Z) \to (U', Z')$
	commuting with the chosen factorizations
	of $\xi$. 
	\end{itemize}
We say that $\mathcal{A}(\xi)$ is the
\textbf{MacPherson-Vilonen recollement}
associated to $\xi$.
\end{definition}

\begin{construction}\label{cstr:iterated-mv}
Let 
$\widetilde{\mathcal{F}} \subseteq \mathsf{P}_{\mathcal{O}}$
be the set of minimal elements,
with
upward closed complement $\mathcal{F}$,
and let $\underline{R}$ be geometrically split, let
$i^*$ and $j^*$ be the usual functors associated to
this situation.
Define $\mathcal{U} := \mathsf{RMod}_{j^*\underline{R}^{\Phi}}$.
Define $\mathcal{Z}:= 
\mathsf{RMod}_{i^*\underline{R}^{\Phi}}$.
Define $F, G: \mathcal{U} \to \mathcal{Z}$
for $T' \in \widetilde{\mathcal{F}}$ by
	\[
	(F\underline{M})(T') =
	\bigoplus_{[T]>[T']}\left(\underline{M}(T) \otimes \mathrm{Hom}(T, T')\right)_{\mathrm{Aut}(T)}
	\]
	\[
	(G\underline{M})(T')
	=
	\prod_{[T]>[T']}\left(\underline{M}(T)^{\mathrm{Hom}(T, T')}\right)^{\mathrm{Aut}(T)}
	\]
and use the trace to define a natural transformation
$\xi: F \to G$.
\end{construction}

By design, we get the following proposition:

\begin{proposition}\label{prop:iterated-mv} With notation
as above, the functor
$\mathsf{RMod}_{\underline{R}^{\Phi}}
\to \mathcal{A}(\xi)$ is an equivalence
of categories.
\end{proposition}

\begin{proof}[Proof of Theorem \ref{thm:split-structure}]
We induct on the size of $\mathsf{P}_{\mathcal{O}}$.
If $\mathsf{P}_{\mathcal{O}}$ is discrete, the
theorem is clear.
For the inductive step, let $\widetilde{\mathcal{F}} \subset
\mathsf{P}_{\mathcal{O}}$
be the set of minimal elements in
$\mathsf{P}_{\mathcal{O}}$. Then the forgetful
functor
	\[
	\mathsf{RMod}_{\underline{R}}
	\longrightarrow
	\mathsf{RMod}_{\underline{R}^{\Phi}}
	\]
respects the recollement
data on source and target. It is an equivalence
on each stratum by the induction
hypothesis.
By Proposition
\ref{prop:pre-hereditary}
and \cite[Prop. 8.6]{FP} combined with
Proposition \ref{prop:iterated-mv},
we know that both the source
and target of the forgetful functor
are pre-hereditary.
Now Theorem 8.4 in loc. cit.
implies that the forgetful functor
is an equivalence of categories,
which was to be shown.
\end{proof}

\subsection{Slices as twisted Mackey functors}
\label{ssec:slices-as-mack}

In order to apply the results from
the previous section, we first
need to compute
$\underline{\mathrm{End}}(W)^{\Phi T}$
for $W$ an isotropic slice $n$-sphere,
and show that $\underline{\mathrm{End}}(W)$
is geometrically splittable, at least for some
choice of $W$. 
These statements follow
by combining the next two lemmas.

\begin{lemma}\label{lem:endo-geom-fix} Let $W$ be
an isotropic slice $n$-sphere,
and $T \in \mathcal{O}$ be arbitrary. Then
taking geometric fixed points of
endomorphisms gives an
$\mathrm{Aut}(T)$-equivariant 
isomorphism
	\[
	\underline{\mathrm{End}}(W)^{\Phi T}
	\longrightarrow
	\mathrm{End}(W^{\Phi T}).
	\]
\end{lemma}
\begin{proof} The map is $\mathrm{Aut}(T)$-equivariant
by functoriality of geometric fixed points,
so we may as well replace
$\mathcal{O}$ by $\mathcal{O}_{/T}$
so that $T$ is terminal in
$\mathcal{O}$.
This leads to a recollement situation
for $\widetilde{\mathcal{F}}=\{T\}$ and
a fiber sequence
	\[
	j_!j^*W \longrightarrow W
	\longrightarrow i_*W^{\Phi T}
	\]
First, let $f: W^{\Phi T} \longrightarrow W^{\Phi T}$
be an endomorphism, and consider the diagram
of solid arrows:
	\[
	\xymatrix{
	W\ar[r] \ar@{-->}[d]& i_*W^{\Phi T}\ar[d]^{i_*f}
	&\\
	W\ar[r] & i_*W^{\Phi T} \ar[r] & \Sigma j_!j^*W
	}
	\]
The composite $W \to \Sigma j_!j^*W$
is null by Proposition \ref{prop-slice-htpy-vanish},
so the dotted arrow exists. Thus,
	\[
	[W, W]
	\to [W^{\Phi T}, W^{\Phi T}]
	\]
is surjective. 

Now suppose $f: W \to W$ is
an endomorphism such that
$f^{\Phi T} = 0$. That is, we have
a diagram of solid arrows:
	\[
	\xymatrix{
	&W\ar[d]^f\ar@{-->}[dl] \ar[r]& i_*W^{\Phi T}\ar[d]^0\\
	j_!j^*W\ar[r]&W \ar[r] & i_*W^{\Phi T}
	}
	\]
The dotted arrow exists since the composite
$W \to W \to i_*W^{\Phi T}$ is null
by commutativity of the diagram. Thus $f$ factors
as a composite:
	\[
	W \to j_!j^*W \to W.
	\]
Let $S = \coprod_{U \in \mathcal{O}-\{T\}} U \in \mathsf{Fin}_{\mathcal{O}}$.
By Proposition \ref{prop-nil-induced}, there is a natural equivalence
	\[
	\hocolim_{\Delta^{op}} (\mathrm{ind}_S \circ \mathrm{res}_S)^{\circ (n+1)}j^*W
	\cong j_!j^*W.
	\]
In particular, there is a map $\mathrm{ind}_S\mathrm{res}_Sj^*W
\to j_!j^*W$ and the cofiber has an associated graded which
is a wedge of nontrivial suspensions of slice $n$-connective objects.
By Proposition \ref{prop-nil-induced} again, we deduce that $f$
factors as
	\[
	W \stackrel{g}{\longrightarrow} \mathrm{ind}_S\mathrm{res}_SW 
	\stackrel{\varepsilon}{\longrightarrow} W,
	\]
where the second map is the counit of the adjunction. Now,
by the definition of the unit of an adjunction,
we may further factor $f$ as a composite:
	\[
	W \stackrel{\eta}{\longrightarrow}
	\mathrm{ind}_S\mathrm{res}_SW
	\stackrel{\mathrm{ind}_S\tilde{g}}{\longrightarrow}
	\mathrm{ind}_S\mathrm{res}_SW
	\stackrel{\varepsilon}{\longrightarrow}
	W	
	\]
where $\tilde{g}: \mathrm{res}_SW \to \mathrm{res}_SW$
is the adjoint of $g$. Breaking $S$ into its component
orbits, we learn that $f$ is a sum of maps
each transferred up from $U \in \mathcal{O}-\{T\}$.

Putting it all together, we've shown that
if $f^{\Phi T} = 0$ then $f$ lies in the transfer ideal
in $\underline{\mathrm{End}}(W)(T)$. The other
inclusion always holds, so the result follows.
\end{proof}

\begin{lemma}\label{lem:endo-geom-split}
There exists
an isotropic slice $n$-sphere $W$ such that,
for any $T \in \mathcal{O}$, the map
	\[
	[\downarrow_T\!W, \downarrow_T\!W]
	\to
	[W^{\Phi T}, W^{\Phi T}]
	\]
admits an $\mathrm{Aut}(T)$-equivariant ring section.
\end{lemma}
\begin{proof} 
Let's temporarily call an $\mathcal{O}$-spectrum
\emph{good} if it satisfies the conclusion of the lemma. 
Revisiting the proof of 
we see that it is enough to prove the following closure properties
for the class of good $\mathcal{O}$-spectra.
	\begin{enumerate}[(i)]
	\item If $U \in \mathcal{O}$ is minimal,
	and $X \in \mathsf{Sp}^{\mathcal{O}_{/U}}$
	is good, then so is $\mathrm{ind}_UX \in \mathsf{Sp}^{\mathcal{O}}$.
	\item If $X$ is good and $U \in \mathcal{O}$
	is arbitrary, then the cofiber of
	$\uparrow_U\!\downarrow_U\!X \to X$
	is also good.
	\item If $X$ is good and $U \in \mathcal{O}$
	is arbitrary, then the fiber
	of $X \to \uparrow_U\!\downarrow_U\!X$
	is also good.
	\end{enumerate}
Claim (i) follows from the more general observation
that if $X$ is good then so is $U_+ \wedge X$ for
any set $U$ because taking geometric fixed points 
commutes with colimits. 

To prove (ii), denote the cofiber by $Y$. We need
to find, for each $T$, an $\mathrm{Aut}(T)$-equivariant
ring section of $[\downarrow_T\!Y, \downarrow_T\!Y]
\to [Y^{\Phi T}, Y^{\Phi T}]$. The argument
depends on the relationship between $T$ and $U$.
	\begin{itemize}
	\item If $[U] \le [T]$, then the map
		\[
		\downarrow_T\uparrow_U\downarrow_U\!X 
		\to \downarrow_TX
		\]
	has a functorial splitting, so $\downarrow_T\!Y$ is
	an $\mathrm{Aut}(T)$-equivariant summand
	of $\downarrow_T\uparrow_U\downarrow_U\!X$
	and the conclusion follows.
	\item If $[U]$ and $[T]$ are incomparable,
	then $\downarrow_T\uparrow_U = 0$ and 
	the conclusion is vacuously satisfied.
	\item If $[T] \le [U]$ then we have
		\[
		\downarrow_T\uparrow_U\downarrow_U
		=
		\uparrow^T_U\downarrow^T_U\downarrow_T
		\]
	Naturality of the counit of an adjunction,
	together with (homotopical) functoriality of the cofiber
	provides us with maps
		\[
		\mathrm{map}\left(\downarrow_T\!X, \downarrow_T\!X\right)
		\longrightarrow
		\mathrm{map}
		\left(\uparrow^T_U\downarrow^T_U\downarrow_T\!X
		\to \downarrow_T\!X, 
		\uparrow^T_U\downarrow^T_U\downarrow_T\!X
		\to \downarrow_T\!X\right)
		\longrightarrow
		\mathrm{map}(\downarrow_T\!Y, \downarrow_T\!Y)
		\]
	\end{itemize}
Taking the composite on $\pi_0$ yields:
	\[
	[\downarrow_TX, \downarrow_TX]
	\to [\downarrow_T Y, \downarrow_TY]
	\]
and naturality ensures that the $\mathrm{Aut}(T)$
action and ring structure (from composition) are preserved.
Finally, in this case $Y^{\Phi T} = X^{\Phi T}$
so precomposing with the assumed splitting
$[X^{\Phi T}, X^{\Phi T}] \to [\downarrow_T\!X, \downarrow_T\!X]$
gives the result.

Claim (iii) is proved in the same way as claim (ii).
\end{proof}

\begin{remark} This proof can be modified
to treat 
the isotropic slice spheres constructed via the
method
described in Remark \ref{rmk:efficient-isotropic-spheres}.
As a corollary of our hunch in that remark, we
guess that \emph{every} isotropic slice $n$-sphere
satisfies the conclusion of the preceding lemma. 
We have not tried to prove this.
\end{remark}

\begin{remark} If $W^{\Phi T}$ is a single sphere,
the splitting trivially exists at $T$. If $T$ is maximal,
then the splitting also always exists because
$\downarrow_T\!W = W^{\Phi T}$. Combining these observations
we learn that if $W$ is a slice $n$-sphere
with the property that $W^{\Phi T}$ is a single sphere
for non-maximal $T$, then $W$ satisfies the conclusion of
the above lemma. This is enough to cover
the examples in the next section, for example.
\end{remark}

These results, combined
with those of \S\ref{ssec:split-modules} and
Theorem \ref{thm:slices-as-modules},
now reduce the study of $\mathsf{Slice}_n$
to the study of Mackey functors equipped
with compatible $\mathrm{Aut}(T)$-equivariant
actions of the rings $\mathrm{End}(W^{\Phi T})$.
Since $W^{\Phi T}$ is a wedge of spheres
of a single dimension, this endomorphism
ring is abstractly equivalent to the matrix
ring $M_k(\mathbb{Z})$. Moreover,
the action of $\mathrm{Aut}(T)$ comes
from an action on $\mathbb{Z}^{\oplus k}$. 

We now apply some Morita theory to
our problem.

\begin{lemma} Let $J$ be a finitely generated,
free abelian group with an action of a finite group $G$.
Let $R = \mathrm{End}(J)$ with its induced left
action of $G$ by conjugation. Denote by
$J^*$ the right
$R$-module $\mathrm{Hom}_{\mathbb{Z}}(J, \mathbb{Z})$
with its left, $G$-action by conjugation (which intertwines
the $R$-module structure).
Then the functors:
	\[
	\mathsf{RMod}_{R\text{-}G} \longrightarrow
	\mathsf{Mod}_G
	\]
	\[
	N \mapsto N \otimes_{R} J
	\]
	\[
	\mathsf{Mod}_G \longrightarrow
	\mathsf{RMod}_{R\text{-}G}
	\]
	\[
	M \mapsto M \otimes J^*
	\]
are inverse equivalences of categories.
\end{lemma}
\begin{proof} This is immediate from
classical Morita theory once one observes that
the unit and counit of the adjunction
on underlying modules respect the prescribed
$G$-actions. But, of course, they were defined
so that this is the case.
\end{proof}

\begin{definition}\label{defn:tw-mack} The category
of \textbf{$n$-twisted Mackey functors}
associated to an isotropic slice $n$-sphere
$W$ with chosen splittings $s_T$
is the category whose objects consist of
the following data:
	\begin{itemize}
	\item For each $T \in \mathcal{O}$,
	an abelian group
	$M_{(T)}$ with an action of $\mathrm{Aut}(T)$;
	\item The structure of
	an object in $\mathsf{RMod}_{\underline{\mathrm{End}}(W)^{\Phi}}$
	on the collection $\{M_{(T)} \otimes J_T^*\}$. 
	\end{itemize}
Here $J_T = \pi_{\nu(n, T)}W^{\Phi T}$. We denote by
$\mathsf{TwMack}^{\mathrm{loc}}_n$ the full subcategory
spanned by those $n$-twisted Mackey functors with
the property that, under the equivalence of Theorem \ref{thm:split-structure},
the associated $\underline{\mathrm{End}}(W)$-module
is slice local in the sense of Definition \ref{defn:slice-local}.
\end{definition}

\begin{remark} We can express the condition
that an $n$-twisted Mackey functor be slice
local directly as follows. For each $T \in \mathcal{O}$
write $T \times T^{jump} = \coprod U_{\alpha}$
in $\mathsf{Fin}_{\mathcal{O}}$
as a coproduct of orbits. Then an $n$-twisted
Mackey functor $M_{(-)}$ is slice local if and only
if, for every $T \in \mathcal{O}$,
the sum of restriction maps
	\[
	M_{(T)} \otimes J_T^*
	\longrightarrow
	\bigoplus_{\alpha} M_{(U_{\alpha})} \otimes J_{U_{\alpha}}^*
	\]
induced by the projection $T \times T^{jump} \to T$,
is injective.
\end{remark}

\begin{remark} Notice that there
are canonical isomorphisms
	\[
	 \mathrm{End}(W^{\Phi T})
\cong \mathrm{End}(J_T)
	\]
	\[
	[W^{\Phi T}, S^{\nu(n, T)}] \cong J_T^*
	\]
of $\mathrm{Aut}(T)$-modules,
given by assigning
to a map between wedges of spheres its 
behavior on $\pi_{\nu(n, T)}$. 
\end{remark}

\begin{remark}\label{rmk:lil-simpler} We will often modify this
definition somewhat by noting that giving an
$\mathrm{End}(W^{\Phi T'})$-module map
	\[
	M_{(T)} \otimes J^*_T
	\to M_{(T')} \otimes J_{T'}
	\]
is equivalent to giving a map of abelian groups
	\[
	M_{(T)} \otimes J^*_T \otimes_{\mathrm{End}(W^{\Phi T'})}J_{T'}
	\to 
	M_{(T')}.
	\]
A similar observation applies to the restriction maps.
Working out the relations between the maps presented this way
is more easily done in practice than in theory, as we
will see in the next section.
\end{remark}

\begin{construction}\label{cstr:idempotent-piece}
Let $W$ be an isotropic slice $n$-sphere
with a prescribed splitting
$s_T: \mathrm{End}(W^{\Phi T}) \to
\mathrm{End}({\downarrow_T\!W})$.
Choose an $S^{\nu(n, T)}$ summand of
the spectrum
$W^{\Phi T}$ with corresponding idempotent
$e \in \mathrm{End}(W^{\Phi T})$ and
retraction
	\[
	\mathrm{pr}: W^{\Phi T} \to S^{\nu(n, T)}.
	\]
Let $W_{(T)}$
be the summand of $\downarrow_T\!\!W$ obtained from the image
of $e$ in $\mathrm{End}(\downarrow_T\!\!W)$. 
Now define
a map, for any $X \in \mathsf{Sp}^{\mathcal{O}}$,
	\[
	[W_{(T)}, \downarrow_T\!X] \otimes [W^{\Phi T}, S^{\nu(n,T)}]
	\longrightarrow 
	[\downarrow_T\!W, \downarrow_T\!X].
	\]
Given an element $f \otimes g$ write $g$ as $\mathrm{pr}A$
where $A: W^{\Phi T} \to W^{\Phi T}$ is an endomorphism.
Then 
	\[
	f \otimes g \mapsto
	(\downarrow_T\!W \stackrel{s_T(A)}{\longrightarrow}
	\downarrow_T\!W \to W_{(T)} \stackrel{f}{\longrightarrow}
	\downarrow_T\!X).
	\]
\end{construction}

\begin{lemma}\label{lem:idempotent-piece} The map constructed in
(\ref{cstr:idempotent-piece}) is an 
$\mathrm{End}(W^{\Phi T})$-isomorphism.
\end{lemma}
\begin{proof} That this map is an isomorphism
is a general fact that belongs to Morita theory:
if $e \in \mathrm{End}(J)$ is a full idempotent
corresponding to a $\mathbb{Z}$ summand
of $J$,
then there
is a canonical isomorphism
of abelian groups
$M \cdot e \cong M \otimes_{\mathrm{End}(J)}J$.
In our case, $[\downarrow_T\!W, \downarrow_T\!X] \cdot e
\cong [W_{(T)}, \downarrow_T\!X]$ and the map
in Construction \ref{cstr:idempotent-piece}
is precisely the composite
	\[
	M\cdot e \otimes J^* \cong 
	M \otimes_{\mathrm{End}(J)}J \otimes J^*
	\cong M
	\]
in our setting.
\end{proof}

\begin{warning} The source
of the map in \ref{cstr:idempotent-piece}
does not usually have an obvious $\mathrm{Aut}(T)$-action,
and hence must inherit one from the target.
It is possible to compute what this action must be in terms
of the action on $J$, the chosen idempotent,
and the action on 
$[W_{(T)}, \downarrow_TX]$ designated by
placing the trivial action on $W_{(T)}$. However,
this is another procedure more easily carried
out in practice than in theory.
\end{warning}

\begin{definition} Let $W$ be an isotropic slice
$n$-sphere with chosen splittings $s_T$,
and summands $W_{(T)}$ of each $\downarrow_T\!W$
arising from an $S^{\nu(n,T)}$ summand of
$W^{\Phi T}$. Then define
	\[
	\hat{\pi}_n:
	\mathsf{Sp}^{\mathcal{O}}
	\longrightarrow
	\mathsf{TwMack}_n
	\]
by the assignment
	\[
	X \mapsto \{[W_{(T)}, \downarrow_T\!X]\}
	\]
equipped with the natural
$\underline{\mathrm{End}}(W)^{\Phi}$-module
structure on the collection
	\[
	\{[W_{(T)}, \downarrow_T\!X] \otimes [W^{\Phi T}, S^{\nu(n, T)}]\}
	\cong \{[\downarrow_T\!W, \downarrow_T\!X]\}.
	\]
\end{definition}

\begin{remark} Notice that $\hat{\pi}_n$ as defined
above contains essentially the same data
as the previously defined $\hat{\pi}_n$
for the case of the testing subcategory generated by the
$W_{(T)}$. We hope this justifies our recycling of the notation.
\end{remark}

Finally, combining Theorem \ref{thm:split-structure}
and Theorem \ref{thm:slices-as-modules}
we conclude:

\begin{theorem}\label{thm:slices-as-tw-mack}
The functor
$\hat{\pi}_n$ yields an equivalence of adjoint pairs:
 
	\[
	\xymatrix{
	\heartsuit_n \ar[rr]^-{\hat{\pi}_n}_-{\cong}\ar[dd]^{P^n} &&
	\mathsf{TwMack}_{n} \ar[dd]_{L^{inj}}\\
	&&\\
	\mathsf{Slice}_n 
	\ar[rr]_-{\hat{\pi}_n}^-{\cong} 
	\ar@/^1.5pc/[uu]&& 
	\mathsf{TwMack}^{\mathrm{loc}}_{n} 
	\ar@/_1.5pc/[uu]
	}
	\]
\end{theorem}

\section{Examples and special cases}\label{sec:examples}

We now apply the general theory of \S2 to several examples.
We begin in \S\ref{ssec:G} by collecting
together the known results for $G$-spectra in general.
Then, in \S\ref{ssec:Cp} we compare our classification
of slices for $C_p$ with the one due
to \cite{HY}. Finally, in \S\ref{ssec:C4} we apply
our machinery to a new example: the case of $C_4$.
The slices for $C_4$ are not all $RO(C_4)$-graded
suspensions of Eilenberg-MacLane spectra.
As a result, the previous methods for studying slices
fail in this case and something like theory
we've developed is necessary.

\subsection{$G$-spectra}\label{ssec:G}

In this section we will restrict attention to the original
slice filtration on $G$-spectra, which we recall is
the one associated to the dimension function
	\[
	\nu(n, H) = \left\lfloor \frac{n}{|H|}\right\rfloor.
	\]

\begin{remark}
Statements about the regular slice filtration
may be recovered from the equality
	\[
	\left\lceil \frac{n+1}{|H|}\right\rceil - 1
	= \left\lfloor \frac{n}{|H|}\right\rfloor.
	\]
That is, regular $n$-slices are the same as 
original $(n-1)$-slices. We warn the reader
that this does \emph{not} mean that the regular $n$-slice
of a spectrum $X$ is the same as its original $(n-1)$-slice.
\end{remark}

To test the effectiveness of our general theory, we
show how to recover the previously known results about the
slice filtration which hold for an arbitrary group
$G$. Of course, in many cases the original proof is simpler
or morally the same as the one given here,
this is only meant to be a proof of concept.

\begin{theorem}\label{thm:lit-review}
	\begin{enumerate}[(i)]
	\item (\cite{HHR, primer})
	The functor $\Sigma^{k\rho}$ yields an 
	equivalence
		\[
		\Sigma^{k\rho}:
		\mathsf{Slice}_n
		\stackrel{\cong}{\longrightarrow}
		\mathsf{Slice}_{n+ k|G|}.
		\]
	\item (\cite{HHR, primer}) The category of $(k|G| - 1)$-slices
	is equivalent via the
	functor
		\[
		\underline{M} \mapsto \Sigma^{k\rho -1}H\underline{M}.
		\]
	to the category of Mackey functors.
	\item (\cite{HHR, primer}) The category
	of $k|G|$-slices is equivalent via the functor
		\[
		\underline{M} \mapsto \Sigma^{k\rho}H\underline{M}
		\]
	to the category
	of Mackey functors all of whose restriction maps
	are injective.
	\item \cite{Ullman-thesis} The category of $(k|G|-2)$-slices
	is equivalent via the functor
		\[
		\underline{M} \mapsto \Sigma^{k\rho -2}H\underline{M}
		\]
	to the category of Mackey functors all of whose transfer
	maps are surjective.
	\item \cite{HY} Fix $n, k \in \mathbb{Z}$. Let $V$
	be a virtual representation with the property
	that, for all $H \subseteq G$, 
		\[
		\mathrm{dim}(V^H) + \left\lfloor \frac{n}{|H|}\right\rfloor
		\ge \left\lfloor \frac{n+k}{|H|}\right\rfloor. \quad (*)
		\]
	Then smashing with $S^V$ gives a functor
		\[
		\Sigma^V: \mathsf{Sp}^G_{\ge n}
		\longrightarrow
		\mathsf{Sp}^G_{\ge n+k}.
		\]
	This functor is an equivalence if and only if equality holds
	in $(*)$.
	\end{enumerate}
\end{theorem}
\begin{proof} Part $(v)$ implies part $(i)$ and both follow
immediately from the formula $(S^V \wedge X)^{\Phi H}
\cong (S^{V^H} \wedge X^{\Phi H})$. Using this,
the statements in $(ii)$-$(iv)$ follow from the special case
when $k=0$. We treat each in turn.
	\begin{enumerate}
	\item[(ii)] The spectrum $S^{-1}$ is an isotropic slice $(-1)$-sphere
	and $(-1)$ is a jump for the dimension function.
	The result now follows from, e.g, Theorem \ref{thm:slices-as-modules}.
	\item[(iii)] The spectrum $S^0$ is an isotropic slice $0$-sphere
	and the only $0$-jump is the trivial subgroup. So,
	by Theorem \ref{thm:slices-as-modules}, $0$-slices
	are the full subcategory of the category of Mackey
	functors spanned by those $\underline{M}$ such that
	the restriction
		\[
		\underline{M}(G/H) \to \underline{M}(G \times G/H)
		= \underline{M}(G)^{\oplus |G/H|}
		\]
	is injective for all $H \subseteq G$. But this restriction
	map is given by the usual restriction
	$\underline{M}(G/H) \to \underline{M}(G)$ followed
	by the diagonal, so it is injective exactly when the
	usual restriction is injective. Since any restriction map
	followed by restriction to $\underline{M}(G)$ must
	be injective, we conclude that all restriction maps are
	injective. The result follows.
	\item[(iv)] If $A$ is a $(-2)$-slice, then $\Sigma^2A \ge 0$
	and so $A$ is $(-2)$-connective in the usual sense 
	by Lemma \ref{lem:slice-vs-post}. On the other hand,
	$G/H_+ \wedge S^n \ge -1$ for all $n \ge -1$
	by inspection of the floor function, so that
	$\underline{\pi}_nA = 0$ for $n \ge -1$. We conclude
	that $A \cong \Sigma^{-2}H \underline{M}$ for some
	Mackey functor $\underline{M}$. We claim that
	$\Sigma^{-2}H \underline{M}$ is a $(-2)$-slice if and only
	if $\underline{M}$ has surjective transfer maps.
	
	Define $W_H$ by the cofiber sequence
		\[
		S^{-2} \stackrel{\mathrm{tr}}{\longrightarrow}
		G/H_+ \wedge S^{-2}
		\longrightarrow W_H.
		\]
	Then $W_H$ is a slice $(-2)$-sphere. Indeed, 
	the underlying cofiber sequence splits
	so that $W_H^{\Phi e}$ is a wedge of copies
	of $S^{-2}$, while, for $K \subsetneq G$,
	the middle term vanishes and $W_H^{\Phi K}$
	is a (possibly vanishing) wedge of copies of $S^{-1}$. 
	This is as prescribed by the floor function $\lfloor -2/ |H|\rfloor$. 
	
	But now Proposition \ref{prop-slice-htpy-vanish} implies
	that $[\Sigma^{-1}W_H, X] = 0$ for any $X \ge -2$
	and any subgroup $H \subseteq G$. This forces
	the transfers in $\underline{\pi}_{-2}X$ to be surjective
	by inspection of the long exact sequence associated to the
	defining cofiber sequence for $W_H$. We conclude that
	the condition on Mackey functors is necessary.
	
	Now suppose that $\underline{M}$ is a Mackey functor.
	If $W \ge -1$ then $\Sigma^2W \ge 1$ and, in particular,
	is $1$-connective by Lemma \ref{lem:slice-vs-post}.
	So $[\Sigma^2W, H\underline{M}]
	=[W, \Sigma^{-2}H\underline{M}] = 0$. Thus, we always have
	$\Sigma^{-2}H\underline{M} \le -2$. If moreover
	$\underline{M}$ has surjective transfer maps,
	we need to show that $\Sigma^{-2}H\underline{M}$
	is slice $(-2)$-connective. To that end,
	consider the diagram in $\mathsf{Sp}^H$
	for $|H| \ne 1$, where we use the usual recollement
	functors on $\mathsf{Sp}^H$ associated to 
	$\widetilde{\mathcal{F}} = \{H/H\}$:
		\[
		\xymatrix{
		\Sigma^{-2}j_!j^*\downarrow_H\!H\underline{M}\ar[r]
		&\Sigma^{-2}\downarrow_H\!H\underline{M}
		\ar[r]&\Sigma^{-2}i_*i^*\downarrow_H\!H\underline{M}\ar[r]&
		\Sigma^{-1}j_!j^*\downarrow_H\!H\underline{M}\\
		\Sigma^{-2}\uparrow_1^H\downarrow_1^GH\underline{M}
		\ar[u]\ar[ur]&&&
		}
		\]
	(Notice that if $H$ is the trivial subgroup, the top left
	object is zero and the vertical arrow does not exist
	because we are using the fact that $\{1\} \subset H$ is a 
	proper subgroup to define that map.)
	Applying $\pi_{-2}^H$ we get:
		\[
		\xymatrix{
		\pi_0^H(j_!j^*\downarrow_H\!H\underline{M}) \ar[r]^-f&
		\underline{M}(G/H) \ar[r]&
		\pi_{-2}\left(\Sigma^{-2}H\underline{M}\right)^{\Phi H}
		\ar[r] & 0\\
		\underline{M}(G)\ar[u]\ar[ur] &&&
		}
		\]
	Since the diagonal arrow is surjective by assumption, so is
	$f$, and hence $\pi_{-2}\left(\Sigma^{-2}H\underline{M}\right)^{\Phi H} = 0$
	for all nontrivial subgroups $H \subseteq G$. It follows that
	$(\Sigma^{-2}H\underline{M})^{\Phi H} \ge -1 = \lfloor (-2)/|H|\rfloor$
	when $|H| \ne 1$ so that $\Sigma^{-2}H\underline{M} \ge -2$,
	which was to be shown.
	\end{enumerate}
\end{proof}

\begin{remark} Though it is the case that $\mathsf{Slice}_{-2}$ is
a localization of the category of Mackey functors,
it is not true in general that $\heartsuit_{-2}$ is equivalent to the category
of Mackey functors, as we will see in the next section
when $G = C_p$.
\end{remark}

The above theorem is more than enough to recover the
known description of slices for the group $G=C_2$.

\begin{corollary}[\cite{HHR, primer}]
Let $G = C_2$.
	\begin{enumerate}[(i)]
	\item The functor
		\[
		\underline{\pi}_{n\rho-1}:
		\mathsf{Slice}_{2n-1} \longrightarrow 
		\mathsf{Mack}(C_2, \mathsf{Ab})
		\]
	is an equivalence of categories.
	\item The functor
		\[
		\underline{\pi}_{n\rho}:
		\mathsf{Slice}_{2n} \longrightarrow
		\mathsf{Mack}(C_2, \mathsf{Ab})
		\]
	is fully faithful. The essential image consists of those
	Mackey functors $\underline{M}$ such that the
	restriction map
		\[
		\mathrm{res}: \underline{M}(*) \longrightarrow \underline{M}(C_2)
		\]
	is injective.
	\item The slices of a $G$-spectrum $X$ are determined by
	the formulae:
		\[
		\underline{\pi}_{n\rho -1}P^{2n-1}_{2n-1}X
		=
		\underline{\pi}_{n\rho - 1}X.
		\]
		\[
		\underline{\pi}_{n\rho}P^{2n}_{2n}X
		=
		\frac{\underline{\pi}_{n\rho}X}{\mathrm{ker}(\mathrm{res})}.
		\]
	\end{enumerate}
\end{corollary}

\subsection{$C_p$-spectra}\label{ssec:Cp}

By the
results in \S\ref{ssec:G}, 
we can already
deduce a description of
the all the categories of slices
for $C_p$-spectra. We find,
as in \cite{HY}, that a description
is possible purely in terms
of $RO(C_p)$-graded
homotopy Mackey functors. In
this section we will employ
the following notation:
	\begin{itemize}
	\item We fix an odd prime $p$. 
	\item We fix a generator $\gamma$
	of $C_p$.
	\item We denote by $\lambda$
	the 2-dimensional real representation
	of $C_p$ where $\gamma$ acts
	by rotation through the angle $2\pi/p$. 
	\item Given a Mackey functor $\underline{M}$
	for $C_p$ valued in abelian groups,
	we denote by $\mathrm{tr}(\underline{M})$
	the sub-Mackey functor generated under
	the transfer by $\underline{M}(C_p)$. 
	Equivalently, $\mathrm{tr}(\underline{M})$
	is defined by the exact sequence:
		\[
		0 \to \mathrm{tr}(\underline{M})
		\to
		\underline{M} \to \Phi^{C_p}\underline{M} \to 0.
		\]
	\end{itemize}

\begin{theorem}[Hill-Yarnall \cite{HY}]
	\begin{enumerate}[(i)]
	\item The functor
		\[
		\underline{\pi}_{n\rho -1}:
		\mathsf{Slice}_{pn -1}
		\longrightarrow
		\mathsf{Mack}(C_p, \mathsf{Ab})
		\] 
	is an equivalence of categories for all $n \in \mathbb{Z}$.
	\item The functor
		\[
		\underline{\pi}_{n\rho +k\lambda}:
		\mathsf{Slice}_{pn + 2k}
		\longrightarrow
		\mathsf{Mack}(C_p, \mathsf{Ab})
		\]
	is fully faithful for all $n \in \mathbb{Z}$ and $0 \le k \le \dfrac{p-3}{2}$.
	The essential image is spanned by those Mackey functors all of whose
	restriction maps are injective.
	\item The functor
		\[
		\underline{\pi}_{n\rho + k\lambda - 1}:
		\mathsf{Slice}_{pn + 2k - 1}
		\longrightarrow
		\mathsf{Mack}(C_p, \mathsf{Ab})
		\]
	is fully faithful for all $n \in\mathbb{Z}$ and
	$1 \le k \le \dfrac{p-1}{2}$. The essential image is spanned
	by those Mackey functors all of whose transfer maps are surjective.
	\item Given a $C_p$-spectrum $X$, its slices are determined by
	the formulae:
		\[
		\underline{\pi}_{n\rho -1}P^{pn-1}_{pn-1}X = 
		\underline{\pi}_{n\rho -1}X.
		\]
		\[
		\underline{\pi}_{n\rho +k\lambda}P^{pn+2k}_{pn +2k}X
		=
		\frac{\underline{\pi}_{n\rho + k\lambda}X}{\mathrm{ker}(\mathrm{res})},
		\quad 0\le k \le \frac{p-3}{2}.
		\]
		\[
		\underline{\pi}_{n\rho + k\lambda - 1}P^{pn + 2k -1}_{pn + 2k-1}X
		=
		\mathrm{tr}\left(\underline{\pi}_{n\rho + k\lambda - 1}X\right),
		\quad
		1\le k\le \frac{p-1}{2}.
		\]
	\end{enumerate}
\end{theorem}
\begin{proof} This follows from the Hill-Yarnall result on periodicity
(seen above as Theorem \ref{thm:lit-review}(v)) applied
to the representation $\lambda$, together with parts
(ii)-(iv) of that same theorem.
\end{proof}

In this section we give a different take on this result.
We find that, even though $\mathsf{Slice}_1$ is a
localization of the category of Mackey functors,
it is not the case that $\heartsuit_1$ is equivalent
to the category of Mackey functors. We then describe
explicitly
how to move back and forth between the two
different descriptions of $\mathsf{Slice}_1$.

We begin by recalling an example of an isotropic slice $1$-sphere.

\begin{definition} Let $S^{\lambda/2}$ denote the cofiber
of the fold map
	\[
	C_{p+} \to S^0.
	\]
This is an isotropic slice $1$-sphere by inspection.
\end{definition}

The reason for the name is the following lemma.

\begin{lemma} There is a cofiber sequence
	\[
	C_{p+} \wedge S^1 \to S^{\lambda/2}
	\to S^{\lambda}.
	\]
\end{lemma}
\begin{proof} Let $S(\lambda)$ denote the
unit sphere in the representation $\lambda$.
Let $\mathrm{sk}_0S(\lambda) = \{z : z^p=1\}
\cong C_{p}$. And notice that
we have a cofiber sequence
	\[
	C_{p+} \to C_{p+} \to S(\lambda)_+
	\]
corresponding to attaching the 1-cell $C_p \times I$.
Now use
the cofiber sequence
	\[
	S(\lambda)_+ \hookrightarrow D(\lambda)_+
	\to S^{\lambda}
	\]
to induce a cell structure on $S^{\lambda}$.
The attaching maps for this cell structure
show that $\mathrm{sk}_1S^{\lambda} = S^{\lambda/2}$
(after taking suspension spectra) and
produce the desired cofiber sequence above.
\end{proof}

From \S2 we know that we must study the Weyl group
action on the geometric fixed points. Luckily, there aren't
many subgroups of $C_p$.

\begin{lemma}\label{lem:aug-ideal} Let $J$ denote the augmentation
ideal in $\mathbb{Z}[C_p]$. Then there is a canonical
$C_p$-equivariant isomorphism
	\[
	\pi_1(\downarrow_1\!S^{\lambda/2}) \stackrel{\cong}{\longrightarrow} J.
	\]
\end{lemma}
\begin{proof} Apply $\pi_1$ to the cofiber sequence:
	\[
	S^0 \to \downarrow_1\!S^{\lambda/2}
	\longrightarrow
	\downarrow_1\!C_{p+} \wedge S^1
	\longrightarrow
	S^1.
	\]
\end{proof}

\begin{proposition} The category $\mathsf{TwMack}_1$ is 
equivalent to
the category whose objects consist of the following data:
	\begin{itemize}
	\item An abelian group $M_{(C_p/C_p)}$,
	\item A $C_p$-module
	$M_{(C_p)}$,
	\item Maps of abelian groups:
		\[
		R: M_{(C_p/C_p)} \to M_{(C_p)}
		\]
		\[
		T: M_{(C_p)} \to M_{(C_p/C_p)}
		\]
	\end{itemize}
subject to the conditions:
	\begin{itemize}
	\item $T( (1+ \cdots + \gamma^{p-1})x) = 0$,
	\item $(1+ \cdots + \gamma^{p-1})R(x) = 0$,
	\item $TR(x) = (1-\gamma)x$.
	\end{itemize}

\end{proposition}
\begin{proof} By definition, an object of $\mathsf{TwMack}_1$
consists of
	\begin{itemize}
	\item An abelian group $M_{(C_p/C_p)}$,
	\item A $C_p$-module $M_{(C_p)}$,
	\item A commutative diagram:
		\[
		\xymatrix{
		\left(M_{(C_p)} \otimes J^*\right)_{C_p}\ar[rr]^{\mathrm{trace}}
		\ar[dr]_{T'}&&
		\left(M_{(C_p)}\otimes J^*\right)^{C_p}\\
		&M_{(C_p/C_p)}\ar[ur]_{R'}&
		}
		\]
	\end{itemize}
We compute the top two pieces of this diagram in more explicit terms.
Consider the exact sequence dual to the one defining $J$:
	\[
	0 \to \mathbb{Z} \to \mathbb{Z}[C_p] \to J^* \to 0
	\]
The first map is given by $1 \mapsto (1+ \cdots + \gamma^{p-1})$.
Since this is split exact as a sequence of abelian groups, we get
an exact sequence:
	\[
	0 \to M_{(C_p)} \to \mathbb{Z}[C_p] \otimes M_{(C_p)}
	\to M_{(C_p)} \otimes J^* \to 0.
	\]
Now apply $C_p$ coinvariants to get
	\[
	M_{(C_p)}/(1-\gamma) \stackrel{(1+ \cdots + \gamma^{p-1})}{\longrightarrow}
	M_{(C_p)} \to (M_{(C_p)} \otimes J^*)_{C_p} \to 0.
	\]
The map $\mathbb{Z}[C_p] \to J$ given by $1 \mapsto (1-\gamma)$
induces an isomorphism $J^* \cong J$, and a similar argument
with the defining exact sequence for $J$ yields
	\[
	\left(M_{(C_p)} \otimes J^*\right)^{C_p} \cong 
	\mathrm{ker}\left((1+\cdots + \gamma^{p-1}):
	M_{(C_p)} \to M_{(C_p)}^{C_p}\right).
	\]
Tracing through the identifications transforms the trace map
into $(1-\gamma)$, and the result is proved.
\end{proof}

For definiteness, we choose the $\mathbb{Z}$-summand
of $J$ corresponding to the element $(1-\gamma)$
in the basis 
$(1-\gamma), (\gamma- \gamma^2), ..., (\gamma^{p-2} - \gamma^{p-1})$.
This produces a specific inclusion and retraction:
	\[
	S^1 \to \downarrow_1\!S^{\lambda/2} \to S^1
	\]
and hence natural transformations:
	\[
	R: [S^{\lambda/2}, -] \to [S^1, \downarrow_1(-)],
	\]
	\[
	T: [S^1, \downarrow_1(-)] \to [S^{\lambda/2}, -].
	\]

Combining the previous result with Theorem \ref{thm:slices-as-tw-mack}
gives:

\begin{theorem} The assignment $\hat{\pi}_1$ given by
	\[
	X \mapsto 
	\begin{gathered}
	\xymatrix{
	[S^{\lambda/2}, X]\ar@/_/[d]_R\\
	[S^1, \downarrow_1\!X]\ar@/_/[u]_T
	}
	\end{gathered}
	\]
gives an equivalence of categories:
	\[
	\hat{\pi}_1:
	\heartsuit_1 \stackrel{\cong}{\longrightarrow}
	\mathsf{TwMack}_1.
	\]
Under this equivalence, the subcategory
$\mathsf{Slice}_1$ corresponds to precisely
those objects of $\mathsf{TwMack}_1$
for which the map $R$ is injective.
\end{theorem}

We now describe how to move back and forth
between this description and that of Hill-Yarnall.

\begin{proposition} Let $X$ be a $C_p$-spectrum. Then there
is a natural isomorphism
	\[
	\mathrm{im}\left(\mathrm{tr}:
	\pi_1(\downarrow_1\!X) \to
	\pi_{\lambda - 1}(X)\right)
	\cong
	\mathrm{cok}\left(R:
	[S^{\lambda/2}, X] \to
	\pi_1(\downarrow_1\!X)\right).
	\]
In particular, the 1-twisted Mackey functor
associated to $X$ determines 
$\mathrm{tr}(\underline{\pi}_{\lambda-1}X)$
and vice-versa.
\end{proposition}
\begin{proof} This follows from the
exact sequence associated to
the cofiber sequence:
	\[
	S^{\lambda -1}
	\to
	C_{p+} \wedge S^1
	\to
	S^{\lambda/2}
	\to
	S^{\lambda}.
	\]
\end{proof}

\subsection{$C_4$-spectra}\label{ssec:C4}


In this section we will see some phenomena not covered
by previous techniques. We will employ the following notation:
	\begin{itemize}
	\item $\gamma$ is a fixed generator of $C_4$.
	\item $\lambda$ is the two-dimensional
	real representation of $C_4$ where $\gamma$
	acts by rotation through the angle $2\pi/4$.
	\item $\sigma$ is the sign representation
	of $C_4$ (where $\gamma$ acts by $-1$).
	\item $\tau$ is the sign representation
	of the subgroup $C_2 \subset C_4$.
	\end{itemize}
 To orient the reader,
we begin with a counterexample to the
statement that every slice is an $RO(G)$-graded
suspension of an Eilenberg-MacLane spectrum.

\begin{counterexample}\label{counter:non-em-slice}
Let $\underline{M}$ denote the Mackey functor
for $C_2$ which is a copy of $\mathbb{Z}$ concentrated at $[C_2/C_2]$. 
Let $\tau$ denote the sign representation of $C_2$.
Then define
	\[
	A:= \,\uparrow^{C_4}\Sigma H\mathbb{Z} \,\,\vee\, \uparrow_{C_2}^{C_4} 
	\Sigma^{\tau}H\underline{M}.
	\]
Since slices are preserved under induction, $A$ is a 1-slice for $C_4$.
Now suppose $V$ is a virtual representation of $C_4$ with the property that
$\Sigma^{-V}A$ is an Eilenberg-MacLane spectrum. The collection
of Eilenberg-MacLane spectra is closed under retracts and restriction, from
which we conclude that the virtual dimension of $V$ is $1$ and that $\downarrow_{C_2}V$
makes
	\[
	\Sigma^{\tau-\downarrow_{C_2} V}H\underline{M} 
	\]
an Eilenberg-MacLane spectrum. Since $\underline{M}$ is concentrated
on $[C_2/C_2]$, it only sees the fixed points of the representations we suspend by.
That is, we may conclude that $\Sigma^{-a}H\underline{M}$ 
is an Eilenberg-MacLane
spectrum, where $a$ is the (virtual) dimension of the fixed points of $V$. 
This spectrum
is nonzero, and an Eilenberg-MacLane spectrum, which means $\pi_0 \ne 0$. This
forces $a=0$. There aren't many representations of $C_2$ with underlying
dimension 1 and fixed point dimension 0, so we 
conclude that $\downarrow_{C_2}V$ is 
equivalent to the sign representation $\tau$. But there is 
no virtual real representation of $C_4$ which restricts
to the regular representation of $C_2$, so no such $V$ exists.
\end{counterexample}

We now proceed with the program from \S2 to study slices for $C_4$.

\begin{construction} Let $S(\lambda)$ denote the
unit sphere in the representation $\lambda$. Then it has
a cell structure with $\mathrm{sk}_0S(\lambda) = \{z: z^4 = 1\}$,
and $\mathrm{sk}_1S(\lambda) = S(\lambda)$ obtained
by attaching a $(C_4 \times D^1)$-cell. We get an induced
cell structure on $\Sigma^{\infty}(D(\lambda)/S(\lambda)) \cong S^{\lambda}$
and define $S^{\lambda/2}$ as the $1$-skeleton. Notice that
this construction produces cofiber sequences:
	\[
	C_{4+} \to S^0 \to S^{\lambda/2},
	\]
	\[
	C_{4+} \wedge S^1 \to S^{\lambda/2} \to S^{\lambda}.
	\]
\end{construction}

We have already noted that the
cofiber $\mathrm{cof}(G_+ \to S^0)$
is always an isotropic slice 1-sphere. Indeed,
all of the geometric fixed points for proper subgroups
are just $S^0$ (with trivial Weyl group
action), while the underlying spectrum
is a wedge of $(|G|-1)$ copies of $S^1$. 

The following is proved exactly as in Lemma \ref{lem:aug-ideal}.

\begin{lemma} There is a canonical $C_4$-equivariant
isomorphism
	\[
	\pi_1(\downarrow_1\!S^{\lambda/2})
	\cong
	J:= \mathrm{ker}\left(\mathbb{Z}[C_4] 
	\stackrel{\varepsilon}{\longrightarrow}
	\mathbb{Z}\right)
	\]
\end{lemma}

Now we can give a concrete description of
$\mathsf{TwMack}_1$.

\begin{proposition} The category $\mathsf{TwMack}_1$
is equivalent to the category whose objects consist of
the data of the diagram of abelian groups:
	\[
	\xymatrix{
	M_{(C_4/C_4)}\ar@/_/[d]_{R^{C_4}_{C_2}}\\
	M_{(C_4/C_2)}\ar@/_/[d]_{R_1^{C_2}}
	\ar@/_/[u]_{T_{C_2}^{C_4}}\\
	M_{(C_4)}\ar@/_/[u]_{T_1^{C_2}}
	}
	\]
Where $M_{(C_4)}$ is a $C_4$-module,
$M_{(C_4/C_2)}$ is a $C_4/C_2$-module,
and the maps are additive and
subject to the following relations:
	\begin{itemize}
	\item (Group action and restrictions)
		\[
		(1+\gamma^2)\cdot R_1^{C_2} = 0
		\]
		\[
		(1-\gamma^2)\cdot R_{C_2}^{C_4} = 0
		\]
		\[
		(1+\gamma
		+\gamma^2+\gamma^3) \cdot R_1^{C_2}R_{C_2}^{C_4} = 0
		\]
	\item (Group action and transfers)
		\[
		T_1^{C_2} \circ (1+\gamma^2) = 0
		\]
		\[
		T_{C_2}^{C_4} \circ (1-\gamma^2) = 0
		\]
		\[
		T_{C_2}^{C_4}T_1^{C_2}\circ(1 + \gamma +
		\gamma^2 + \gamma^3) = 0
		\]
	\item (Double coset formulae)
		\[
		R^{C_2}_1T_1^{C_2} = (1-\gamma)
		\]
		\[
		R^{C_4}_{C_2}T^{C_4}_{C_2} = (1+\gamma)
		\]
		\[
		R_1^{C_2}R_{C_2}^{C_4}T_{C_2}^{C_4}T_1^{C_2}
		=
		(1 -\gamma)
		\]
	\end{itemize}
Under this equivalence, the category of $1$-slices is equivalent
to the subcategory of those diagrams for which the map
	\[
	M_{(C_4/C_4)}
	\stackrel{(R_{C_2}^{C_4}, R^{C_4}_1)}{\longrightarrow}
	M_{(C_4/C_2)} \oplus M_{(C_4)}
	\]
is injective.
\end{proposition}
\begin{proof} Arguing as in the case of $C_p$, one 
establishes isomorphisms:
	\[
	(M \otimes J^*)_{C_2} \cong \uparrow^{C_4/C_2}(M/(1+\gamma^2)),
	\quad\quad
	(M\otimes J^*)^{C_2} \cong 
	\uparrow^{C_4/C_2}(\mathrm{ker}( (1+\gamma^2):
	M \to M).
	\]
	\[
	(M \otimes J^*)_{C_4} \cong
	M/(1+\gamma+ \gamma^2+\gamma^3),
	\quad\quad
	(M \otimes J^*)^{C_4}
	\cong
	\mathrm{ker}((1+\gamma + \gamma^2+\gamma^3):
	M \to M)
	\]
and does a diagram chase.
\end{proof}

Now we establish the link with homotopy theory.

\begin{lemma} There is a $C_2$-equivalence
	\[
	\downarrow_{C_2}S^{\lambda/2} \cong S^{\tau}
	\vee \uparrow_1^{C_2}S^1.
	\]
\end{lemma}
\begin{proof} As a $C_2$-set, $C_{4+} = (C_2 \amalg C_2)_+$.
It follows that the cofiber
$\mathrm{cof}(C_{2+} \to S^0) \cong S^{\tau}$
is a summand, and the remaining piece is the cofiber
$\mathrm{cof}(C_{2+} \to 0)$ which is $\uparrow_1^{C_2}S^1$.
\end{proof}

While this splitting does not behave nicely
with respect to the $(C_4/C_2)$-action, it does
tell us what groups one needs to compute.
As in the previous section, the element
$(1-\gamma) \in J$ together with the standard basis
$(\gamma^i - \gamma^{i+1})$ of $J$ gives
a splitting:
	\[
	S^1 \to \downarrow_1\!S^{\lambda/2} \to S^1.
	\]
We get the following theorem as a corollary of our
main results:

\begin{theorem} The assignment
	\[
	X \mapsto \begin{gathered}
	\xymatrix{
	[S^{\lambda/2}, X]\ar@/_/[d]_{R^{C_4}_{C_2}}\\
	[S^{\tau}\vee \uparrow^{C_2}\!S^1, \downarrow_{C_2}\!X]
	\ar@/_/[d]_{R_1^{C_2}}
	\ar@/_/[u]_{T_{C_2}^{C_4}}\\
	[S^1, \downarrow_{C_4}\!X]\ar@/_/[u]_{T_1^{C_2}}
	}
	\end{gathered}
	\]
gives an equivalence
of categories
	\[
	\heartsuit_1 \stackrel{\cong}{\longrightarrow} 
	\mathsf{TwMack}_1.
	\]
Under this equivalence, the $1$-slice of a spectrum is
computed by forcing $R_1^{C_2} \oplus R_{C_2}^{C_4}$
to be injective.
\end{theorem}

\printbibliography

@article{BenRoub,
    AUTHOR = {B\'enabou, Jean and Roubaud, Jacques},
     TITLE = {Monades et descente},
   JOURNAL = {C. R. Acad. Sci. Paris S\'er. A-B},
    VOLUME = {270},
      YEAR = {1970},
     PAGES = {A96--A98},
   MRCLASS = {18.10},
  MRNUMBER = {0255631},
MRREVIEWER = {R. Hoobler},
}

@article {atiyah,
    AUTHOR = {Atiyah, M. F.},
     TITLE = {{$K$}-theory and reality},
   JOURNAL = {Quart. J. Math. Oxford Ser. (2)},
  FJOURNAL = {The Quarterly Journal of Mathematics. Oxford. Second Series},
    VOLUME = {17},
      YEAR = {1966},
}

@article {dug,
    AUTHOR = {Dugger, Daniel},
     TITLE = {An {A}tiyah-{H}irzebruch spectral sequence for {$KR$}-theory},
   JOURNAL = {$K$-Theory},
  FJOURNAL = {$K$-Theory. An Interdisciplinary Journal for the Development,
              Application, and Influence of $K$-Theory in the Mathematical
              Sciences},
    VOLUME = {35},
      YEAR = {2005},
    NUMBER = {3-4},
}

@ARTICLE{BDGNS-basics,
   author = {{Barwick}, C. and {Dotto}, E. and {Glasman}, S. and {Nardin}, D. and 
	{Shah}, J.},
    title = "{Parametrized higher category theory and higher algebra: Expos$\backslash$'e I -- Elements of parametrized higher category theory}",
  journal = {ArXiv e-prints},
archivePrefix = "arXiv",
   eprint = {1608.03657},
 primaryClass = "math.AT",
     year = 2016,
    month = aug
}

@article {HHR,
    AUTHOR = {Hill, M. A. and Hopkins, M. J. and Ravenel, D. C.},
     TITLE = {On the nonexistence of elements of {K}ervaire invariant one},
   JOURNAL = {Ann. of Math. (2)},
      YEAR = {2016},
}

@ARTICLE{lax,
   author = {{Gepner}, D. and {Haugseng}, R. and {Nikolaus}, T.},
    title = "{Lax colimits and free fibrations in $\infty$-categories}",
  journal = {ArXiv e-prints},
archivePrefix = "arXiv",
   eprint = {1501.02161},
 primaryClass = "math.CT",
     year = 2015,
    month = jan,
}

@book {td,
    AUTHOR = {tom Dieck, Tammo},
     TITLE = {Transformation groups and representation theory},
    SERIES = {Lecture Notes in Mathematics},
    VOLUME = {766},
 PUBLISHER = {Springer, Berlin},
      YEAR = {1979},
}

@ARTICLE{recoll,
   author = {{Barwick}, C. and {Glasman}, S.},
    title = "{A note on stable recollements}",
  journal = {ArXiv e-prints},
archivePrefix = "arXiv",
   eprint = {1607.02064},
 primaryClass = "math.CT",
     year = 2016,
    month = jul,
}

@ARTICLE{BG-cycl,
   author = {{Barwick}, C. and {Glasman}, S.},
    title = "{Cyclonic spectra, cyclotomic spectra, and a conjecture of Kaledin}",
  journal = {ArXiv e-prints},
archivePrefix = "arXiv",
   eprint = {1602.02163},
 primaryClass = "math.AT",
     year = 2016,
    month = feb,
}

@ARTICLE{mackI,
   author = {{Barwick}, C.},
    title = "{Spectral Mackey functors and equivariant algebraic K-theory (I)}",
  journal = {ArXiv e-prints},
archivePrefix = "arXiv",
   eprint = {1404.0108},
 primaryClass = "math.AT",
     year = 2014,
    month = mar,
}

@ARTICLE{glas-calc,
   author = {{Glasman}, S.},
    title = "{Goodwillie calculus and Mackey functors}",
  journal = {ArXiv e-prints},
archivePrefix = "arXiv",
   eprint = {1610.03127},
 primaryClass = "math.AT",
     year = 2016,
    month = oct,
   adsurl = {http://adsabs.harvard.edu/abs/2016arXiv161003127G}
}

@ARTICLE{denis-stable,
   author = {{Nardin}, D.},
    title = "{Parametrized higher category theory and higher algebra: Expos$\backslash$'e IV -- Stability with respect to an orbital $\infty$-category}",
  journal = {ArXiv e-prints},
archivePrefix = "arXiv",
   eprint = {1608.07704},
 primaryClass = "math.AT",
     year = 2016,
    month = aug,
}

@ARTICLE{GuM,
   author = {{Guillou}, B. and {May}, J.~P.},
    title = "{Models of G-spectra as presheaves of spectra}",
  journal = {ArXiv e-prints},
archivePrefix = "arXiv",
   eprint = {1110.3571},
 primaryClass = "math.AT",
     year = 2011,
    month = oct,
}

@article {I,
    AUTHOR = {Inoue, Masateru},
     TITLE = {Odd primary {S}teenrod algebra, additive formal group laws,
              and modular invariants},
   JOURNAL = {J. Math. Soc. Japan},
  FJOURNAL = {Journal of the Mathematical Society of Japan},
    VOLUME = {58},
      YEAR = {2006},
    NUMBER = {2},
     PAGES = {311--332},
}

@book {LMS,
    AUTHOR = {Lewis, Jr., L. G. and May, J. P. and Steinberger, M. and
              McClure, J. E.},
     TITLE = {Equivariant stable homotopy theory},
    SERIES = {Lecture Notes in Mathematics},
    VOLUME = {1213},
      NOTE = {With contributions by J. E. McClure},
 PUBLISHER = {Springer-Verlag, Berlin},
      YEAR = {1986},
}

@article {R,
    AUTHOR = {Ricka, Nicolas},
     TITLE = {Subalgebras of the {$\Bbb Z/2$}-equivariant {S}teenrod
              algebra},
   JOURNAL = {Homology Homotopy Appl.},
  FJOURNAL = {Homology, Homotopy and Applications},
    VOLUME = {17},
      YEAR = {2015},
    NUMBER = {1},
     PAGES = {281--305},
}

@ARTICLE{Ullman,
   author = {Ullman, J.},
    title = "{On the Slice Spectral Sequence}",
  journal = {ArXiv e-prints},
archivePrefix = "arXiv",
   eprint = {1206.0058},
 primaryClass = "math.AT",
     year = 2012,
    month = may,
}

@book {T,
    AUTHOR = {Tilson, Sean Michael James},
     TITLE = {Power operations in the {K}uenneth and {C}2-equivariant
              {A}dams spectral sequences with applications},
      NOTE = {Thesis (Ph.D.)--Wayne State University},
 PUBLISHER = {ProQuest LLC, Ann Arbor, MI},
      YEAR = {2013},
}

@article {FP,
    AUTHOR = {Franjou, Vincent and Pirashvili, Teimuraz},
     TITLE = {Comparison of abelian categories recollements},
   JOURNAL = {Doc. Math.},
  FJOURNAL = {Documenta Mathematica},
    VOLUME = {9},
      YEAR = {2004},
}

@article {C,
    AUTHOR = {Caruso, Jeffrey L.},
     TITLE = {Operations in equivariant {${\bf Z}/p$}-cohomology},
   JOURNAL = {Math. Proc. Cambridge Philos. Soc.},
  FJOURNAL = {Mathematical Proceedings of the Cambridge Philosophical
              Society},
    VOLUME = {126},
      YEAR = {1999},
    NUMBER = {3},
     PAGES = {521--541},
}

@book{HTT,
author="Lurie, J.",
title={Higher Topos Theory.},
year=2009,
}

@book{HA,
author="Lurie, J.",
title={Higher Algebra.},
year=2016,
}

@article {P,
    AUTHOR = {Priddy, Stewart},
     TITLE = {A cellular construction of {BP} and other irreducible spectra},
   JOURNAL = {Math. Z.},
  FJOURNAL = {Mathematische Zeitschrift},
    VOLUME = {173},
      YEAR = {1980},
    NUMBER = {1},
     PAGES = {29--34},
}

@article {Kuhn,
    AUTHOR = {Kuhn, Nicholas J.},
     TITLE = {Generic representations of the finite general linear groups
              and the {S}teenrod algebra. {I}},
   JOURNAL = {Amer. J. Math.},
  FJOURNAL = {American Journal of Mathematics},
    VOLUME = {116},
      YEAR = {1994},
    NUMBER = {2},
     PAGES = {327--360},
      ISSN = {0002-9327},
   MRCLASS = {55S10 (20G05)},
  MRNUMBER = {1269607},
MRREVIEWER = {Haynes R. Miller},
       URL = {https://doi.org/10.2307/2374932},
}

@book{Ullman-thesis,
	author={Ullman, John},
	title={On the Regular Slice Spectral Sequence},
	note={PhD thesis, Massachusetts Institute of Technology},
	year={2013},
}

@book{freyd-thesis,
	author={Freyd, Peter},
	title={Functor Theory},
	note={PhD Thesis, Princeton University},
	year={1960},
}

@article {MV,
    AUTHOR = {MacPherson, Robert and Vilonen, Kari},
     TITLE = {Elementary construction of perverse sheaves},
   JOURNAL = {Invent. Math.},
  FJOURNAL = {Inventiones Mathematicae},
    VOLUME = {84},
      YEAR = {1986},
    NUMBER = {2},
     PAGES = {403--435},
      ISSN = {0020-9910},
   MRCLASS = {32C38 (32C42 55N35 57N65)},
  MRNUMBER = {833195},
MRREVIEWER = {P. Schapira},
       URL = {https://doi.org/10.1007/BF01388812},
}

@ARTICLE{mack2,
   author = {{Barwick}, C. and {Glasman}, S. and {Shah}, J.},
    title = "{Spectral Mackey functors and equivariant algebraic K-theory (II)}",
  journal = {ArXiv e-prints},
archivePrefix = "arXiv",
   eprint = {1505.03098},
 primaryClass = "math.AT",
     year = 2015,
    month = may,
   adsurl = {http://adsabs.harvard.edu/abs/2015arXiv150503098B}
}

@article {GP,
    AUTHOR = {Popesco, Nicolae and Gabriel, Pierre},
     TITLE = {Caract\'erisation des cat\'egories ab\'eliennes avec g\'en\'erateurs et
              limites inductives exactes},
   JOURNAL = {C. R. Acad. Sci. Paris},
    VOLUME = {258},
      YEAR = {1964},
     PAGES = {4188--4190},
   MRCLASS = {18.15},
  MRNUMBER = {0166241},
MRREVIEWER = {A. Heller},
}

@article {gabriel,
    AUTHOR = {Gabriel, Pierre},
     TITLE = {Des cat\'egories ab\'eliennes},
   JOURNAL = {Bull. Soc. Math. France},
  FJOURNAL = {Bulletin de la Soci\'et\'e Math\'ematique de France},
    VOLUME = {90},
      YEAR = {1962},
     PAGES = {323--448},
      ISSN = {0037-9484},
   MRCLASS = {18.20},
  MRNUMBER = {0232821},
MRREVIEWER = {T.-Y. Lam},
       URL = {http://www.numdam.org/item?id=BSMF_1962__90__323_0},
}

@ARTICLE{HY,
   author = {{Hill}, M.~A. and {Yarnall}, C.},
    title = "{A new formulation of the equivariant slice filtration with applications to $C\_p$-slices}",
  journal = {ArXiv e-prints},
archivePrefix = "arXiv",
   eprint = {1703.10526},
 primaryClass = "math.AT",
     year = 2017,
    month = mar,
   adsurl = {http://adsabs.harvard.edu/abs/2017arXiv170310526H}
}

@ARTICLE{primer,
   author = {{Hill}, M.~A.},
    title = "{The Equivariant Slice Filtration: a Primer}",
  journal = {ArXiv e-prints},
archivePrefix = "arXiv",
   eprint = {1107.3582},
 primaryClass = "math.AT",
     year = 2011,
    month = jul,
   adsurl = {http://adsabs.harvard.edu/abs/2011arXiv1107.3582H}
}

@ARTICLE{Glas,
   author = {{Glasman}, S.},
    title = "{Stratified categories, geometric fixed points and a generalized Arone-Ching theorem}",
  journal = {ArXiv e-prints},
archivePrefix = "arXiv",
   eprint = {1507.01976},
 primaryClass = "math.AT",
     year = 2015,
    month = jul,
   adsurl = {http://adsabs.harvard.edu/abs/2015arXiv150701976G}
}

@ARTICLE{BM,
   author = {{Blumberg}, A.~J. and {Mandell}, M.~A.},
    title = "{The homotopy theory of cyclotomic spectra}",
  journal = {ArXiv e-prints},
archivePrefix = "arXiv",
   eprint = {1303.1694},
 primaryClass = "math.KT",
     year = 2013,
    month = mar,
   adsurl = {http://adsabs.harvard.edu/abs/2013arXiv1303.1694B}
}

\end{document}